\documentclass[a4paper,11pt]{article}

\usepackage{amssymb, amstext, amsmath, amsthm, latexsym, textcomp}

\usepackage[utf8]{inputenc} 
\usepackage[english]{babel}

\newtheorem{defi}{Definition}[section]
\newtheorem{teor}[defi]{Theorem}
\newtheorem{lemma}[defi]{Lemma}
\newtheorem{prop}[defi]{Proposition}
\newtheorem{rem}[defi]{Remark}

\newtheorem{co}[defi]{Corollary}

\DeclareMathOperator{\ad}{ad}

\DeclareMathOperator{\Ad}{Ad}

\DeclareMathOperator{\supp}{supp}
\DeclareMathOperator{\prr}{Rad}

\DeclareMathOperator{\Ch}{char}

\DeclareMathOperator{\Spec}{Spec}

\DeclareMathOperator{\Id}{Id}

\DeclareMathOperator{\Ker}{Ker}
\DeclareMathOperator{\Hom}{Hom}
\DeclareMathOperator{\End}{End}

\DeclareMathOperator{\Der}{Der}

\DeclareMathOperator{\Sym}{Sym}

\DeclareMathOperator{\CM}{CM}
\DeclareMathOperator{\Ann}{Ann}

\DeclareMathOperator{\TKK}{TKK}

\DeclareMathOperator{\PDer}{PDer}

\DeclareMathOperator{\Jr}{J}

\begin{document}

\date{}

\title{Orthogonal completion of algebraic systems} 

\author{A.~Yu.~Golubkov}

\maketitle

\begin{abstract} 
The paper considers the generalization of the orthogonal completion construction to
all semiprime linear algebras (systems), its implementations in
maximal algebras (systems) of quotients similar to the symmetric Martindale ring of quotients for graded Lie algebras,
linear Jordan algebras (triple systems, pairs), triple Lie systems,
and in maximal right algebras of quotients for alternative algebras.
\end{abstract}

\section{Introduction}

The construction of the orthogonal completion of semiprime associative rings and the method of transferring statements in terms of Horn formulas from the prime to the semiprime case are described in detail in \cite{B6, BMich, BMM}, its topological interpretation is in \cite{Har2}, and a detailed background to the question of its origin is in \cite{BMich}. This paper describes a generalization of the
orthogonal completion for all semiprime linear algebras from the perspective of injective hulls of modules (Proposition 2.4). Its main
results are related to the realization of orthogonal completion in maximal algebras, triple systems and pairs of quotients similar to Martindale ones for homogeneously semiprime graded Lie algebras, nondegenerate linear Jordan algebras, triple Jordan systems and 
pairs, $\mathcal{R}$--semiprime triple Lie systems, and in maximal right algebras of quotients for semiprime alternative algebras (Propositions 3.1 (3.2), 3.3, 3.5, 3.7, 3.9, Theorem 3.10).
Technical results on the primeness (strong primeness) of
Pierce fibers (Lemmas 2.5, 2.6, 2.8, 2.9, 3.4, 3.6, 3.8, Corollary 2.7) are necessary for applying the conclusions of 
\cite[\S~5,~6]{BMich} in 
the corresponding cases according to \cite[\S~8]{BMich}. They will be 
used in a subsequent paper, where orthogonal completion of 
algebraic systems is considered within the framework of almost classical localizations.

Taking this opportunity, I would like to express my gratitude to Professor 
A.~V.~Mikhalev for valuable motivating discussions on writing the paper and 
to Professor E.~Garc{\'\i}a for the presented material on systems of 
$\mathfrak{M}$--quotients of triple Lie systems.

Everywhere below, unless additional conditions are imposed, $F$ is any 
associative commutative ring with 1, all $F$--modules are unitary, 
both left and right with identical action of $F$ and the action of their 
endomorphisms will be written on the right, all algebras, triple systems 
and pairs over $F$ are linear, classes of $F$--algebras contain the zero 
algebra and are closed under taking isomorphic copies.

For any $F$--algebra $R$ and set ${B\subseteq R}$, $F B$, $\langle B\rangle$ 
and $(B)_R$ are the $F$--submodule, subalgebra and ideal of $R$ generated by 
$B$, $\End_F(R)$ is the $F$--algebra of endomorphisms of the 
$F$--module $R$, $\Der(R)$ ($\Der_F(R)$ if necessary indicate $F$) is the Lie 
algebra of $F$--derivations of $R$ (a subalgebra of the Lie algebra 
$\End_F(R)^{(-)}$ obtained from $\End_F(R)$ by replacing multiplication with 
the ring commutator), ${M^R(B)=\langle t_x\mid t=l, r,\ x\in B\rangle}$ and, 
in particular, ${M(R)=M^R(R)}$ (${M(R)'=F \Id_R+M(R)}$) is the \emph{algebra of 
multiplications} (\emph{with unity}) of $R$, $Z(R)$ is the center of $R$, 
\begin{multline*}
Z(R)\ =\ \{x\in R\mid (x, R, R)=(R, x, R)=(R, R, x)=[x, R]=\{0\}\}\ =\ 
\\
\{x\in R\mid l_x=r_x\in Z(M(R))\}\,,
\end{multline*}
${\Ann_t B=\{x\in R\mid (B)t_x=\{0\}\}, t=l, r}$, where 
${l_x: y\longmapsto x y}$ and ${r_x: y\longmapsto y x, x, y\in R}$, are the 
operators of left and right multiplication by $x$, $\Id_M$ is the identical 
isomorphism of the module $M$, ${(x, y, z)=(x y) z-x (y z)}$ and 
${[x, y]=x y-y x}$ are the associator and ring commutator of ${x, y, z\in R}$. 
If $R$ is a Lie algebra, then ${[\ ,\ ]}$ is the multiplication operation of 
$R$, ${\ad_x=r_x=-l_x}$, ${x\in R}$, ${\Ad(R)=M(R)}$, ${\Ad(R)'=M(R)'}$, 
${\ad(R)=\{\ad_x\mid x\in R\}}$ is the ideal of $\Der(R)$ of inner derivations 
of $R$, ${\ad: x\longmapsto \ad_x, x\in R}$, is the canonical epimorphism of 
$R$ onto $\ad(R)$, ${\Ker \ad=\Ann_r R=\Ann_l R}$ (the center of the Lie 
algebra $R$; the choice of ${\ad=r}$ is related to the right action of 
$\End_F(R)$ on $R$ in the paper).  

An algebra $R$ is \emph{prime} (\emph{semiprime}) if ${I J\ne \{0\}}$ 
(${I^2\ne \{0\}}$) for all ${\{0\}\ne I, J\lhd R}$. The primeness of $R$ is 
equivalent to its semiprimeness and ${I\cap J\ne \{0\}}$ for all  
${\{0\}\ne I, J\lhd R}$. An ideal ${P\lhd R}$ is a \emph{prime ideal of $R$} 
if $R/P$ is prime, $\Spec(R)$ is the set of all prime ${P\lhd R}$, 
${\prr(R)=\bigcap\limits_{P\in \Spec(R)} P}$ is the \emph{prime radical of $R$}, $\prr(R)$ is the smallest of ${I\lhd R}$, 
$R/I$ is semiprime \cite[Proposition~4,~Theorem~6,~p.~192,~193]{ZShS}. 

An algebra (module, system) without $k$--torsion, ${k\geq 1}$, in an 
additive group (${k x\ne 0}$ for all ${x\ne 0}$) will be called an 
\emph{algebra} (\emph{module, system}) \emph{without $k$--torsion}.

The Martindale centroid $\CM(R)$, the extended centroid ${}_R C$ and the central 
closures $P(R)$ and ${}_R Q$ of the semiprime algebra $R$, 
${\CM(R)\cong {}_R C, P(R)\cong {}_R Q}$, will be considered in the second part. 
When speaking about right modules over an associative ring (algebra) $A$, we will 
use standard notations: ${\Ann_A L=\{x\in A\mid L x=\{0\}\}}$ is the annihilator 
of the set ${L\subseteq M_A}$, ${\Hom(M, N)_A}$ is the group of homomorphisms of 
the $A$--modules $M_A$ and $N_A$, $\End(M)_A$ is the ring of endomorphisms of 
$M_A$. A submodule of $M_A$ that has non-zero intersections with all\linebreak non-zero 
submodules of $M_A$ is called \emph{essential} (\emph{large}), a submodule of 
$M_A$ that is invariant under the action of $\End(M)_A$ is called 
\emph{fully invariant}. The set of all essential ideals ($M(R)'$--submodules) 
of the algebra (system, pair) $R$ will be denoted by $\mathcal{E}(R)$. Following 
\cite{Ut, Lamb, FeL}, we present a number of constructions of rings of quotients 
of associative rings. 

\smallskip
An associative ring $A$ is called \emph{left exact} if ${{}_A A}$ is an exact 
$A$--module, ${\Ann_l A=\{0\}}$. An associative ring $B$ is a \emph{right ring 
of quotients} of a ring $A$ if $A$ is a subring of $B$, 
${x (y^{-1} A)\ne \{0\}}$ for any ${x, y\in B, x\ne 0}$, 
${y^{-1} A=\{a\in A\mid y a\in A\}}$. A right ideal $I$ of a ring $A$ is called 
\emph{dense} if $A$ is a right ring of quotients of $I$. The latter is equivalent 
to the existence for any ${x, y\in A}$, ${x\ne 0}$, such a ${z\in A}$ that 
${x z\ne 0}$, ${y z\in I}$ \cite[(1.6)]{Ut} (for left exact $A$ one can use 
the condition of \cite[Proposition~2,~p.~58]{FeL} with 
${z\in A^1=\mathbb{Z}\oplus A}$, since in this case there is ${a\in A}$, 
${x a\ne 0}$, and for ${z\in A^1, x a z\ne 0}$, ${y a z\in I, a z\in A}$). 

The left exactness of the ring $A$ is necessary and sufficient for the existence of 
a right ring of quotients for $A$ \cite[(1.11)]{Ut}. If $B$ is a right 
ring of quotients of $A$, $C$ is a subring of $B$, ${A\subseteq C}$, then 
$C$ and $B$ are right rings of quotients of $A$ and $C$, $A$ is an essential 
submodule of $C_A$ and $C$ inherits the semiprimeness (primeness) of $A$, 
${Z(B)=\{x\in B\mid [x, A]=\{0\}\}}$ \cite[(1.3)]{Ut} and for $A$ with 1 
the last one is the unity of $B$ (${\{0\}=1 (1 b-b) A=
A\cap (1 b-b) A=(b 1-b) A}$, ${b=1 b=b 1}$ for any ${b\in B}$). The left 
exact $A$ has a unique, up to an isomorphism identical on $A$, \emph{complete} 
(\emph{maximal}) \emph{right ring of quotients} $Q(A)$, into which any right 
ring of quotients of $A$ is embedded identically on $A$ 
\cite[(1.2),~(1.11),~Theorem~1]{Ut}.
If $\mathcal{D}_r(A)$ and $\mathcal{C}_{dr}(A)$ are sets of dense right ideals 
of $A$ and pairs ${(\phi, I), I\in \mathcal{D}_r(A), \phi\in \Hom(I, A)_A}$, 
then $Q(A)$ can be realized as a ring ${\mathcal{C}_{dr}(A)/\sim}$ with unity 
${[(\Id_A, A)]}$, constructed similarly to ${}_A C$ with 
${A\cong \{[(l_x, A)]\mid x\in A\}}$ (${x\longmapsto [(l_x, A)]}$). If $A$ with 
1, then ${Q(A)\cong \End_{\End(I(A_A))_A}(I(A_A))}$, $I(A_A)$ is the injective 
hull of $A_A$ \cite[Proposition~6,~p.~162]{Lamb} (for $A$ without 1 we can 
go to ${Q(A)=Q(Q(A))}$ \cite[(1.15)]{Ut}). It is known that $Q(A)$ is 
defined by the following axioms:
\begin{enumerate}

\item $A$ is a subring of $Q(A)$; 

\item for any ${q\in Q(A)}$ there exists ${I\in \mathcal{D}_r(A)}$, 
${q I\subseteq A}$;

\item if ${q\in Q(A)}$, ${q I=\{0\}}$ for some ${I\in \mathcal{D}_r(A)}$, 
then ${q=0}$;

\item for any ${I\in \mathcal{D}_r(A)}$, ${\phi\in \Hom(I, A)_A}$ there 
is ${q\in Q(A)}$, ${\phi=l_q|_I}$.

\end{enumerate}
Having selected the set ${\mathcal{D}_{ir}(A)=\{I\lhd A\mid I\in \mathcal{D}_r(A)\}=
\{I\lhd A\mid \Ann_A I=\{0\}\}}$ of ideals of $A$ with zero left annihilators 
(${\mathcal{D}_{ir}(A)=\mathcal{E}(A)}$ for semiprime $A$), we can write
\[
Z(Q(A))\ =\ \{q\in Q(A)\mid \exists I\in \mathcal{D}_{ir}(A),\ 
I\subseteq q^{-1} A,\ l_q|_I\in \Hom_A(I, A)_A\}\,. 
\]

The left exactness of $A$ is equivalent to any of the conditions: 
${\mathcal{D}_r(A)\ne \emptyset}$; 
${A\in \mathcal{D}_r(A)}$; ${\mathcal{D}_{ir}(A)\ne \emptyset}$; 
${A\in \mathcal{D}_{ir}(A)}$. 
If $B$ is a subring and a dense right $A$--submodule of $Q(A)$ (i.e. for any 
${x, y\in Q(A)}$, ${x\ne 0}$, there is ${a\in A, x a\ne 0, y a\in B}$), then 
$B$ is left exact and ${Q(B)=Q(A)}$.

The \emph{right Martindale ring of quotients} $Q^r_m(A)$ of the ring $A$ is 
defined by axioms 1--4 with $Q^r_m(A)$ and $\mathcal{D}_{ir}(A)$ instead 
of $Q(A)$ and $\mathcal{D}_r(A)$ uniquely, up to an isomorphism identical 
on $A$, and is realized as a subring 
${\mathcal{C}_{idr}(A)/\sim\subseteq \mathcal{C}_{dr}(A)/\sim\cong Q(A)}$ with 
\[
[(\phi, I)]+[(\psi, J)]\ =\ [(\phi+\psi, I\cap J)]\,,\quad 
[(\phi, I)] [(\psi, J)]\ =\ [(\phi \psi, J I)]
\]
${(\phi, I), (\psi, J)\in \mathcal{C}_{idr}(A)=\{(\phi, I)\mid 
I\in \mathcal{D}_{ir}(A),\ \phi\in \Hom(I, A)_A\}}$. To complement $Q(A)$ and 
$Q^r_m(A)$ in the list of constructions of right rings of quotients of the 
left exact $A$ one can add a subring 
\[
Q^s_m(A)\ =\ \{q\in Q^r_m(A)\mid \exists I\in \mathcal{D}_{ir}(A),\ 
q I+I q\subseteq A\}\subseteq Q^r_m(A)\,,
\]
which is called the \emph{symmetric Martindale ring of quotients of} $A$
\cite[(1.7)]{Ut}. The given realizations of $Q(A)$ and $Q^r_m(A)$ allow 
us to identify $Q^s_m(A)$ and $Q^r_m(A)$ with subrings of $Q(A)$, 
${A\subseteq Q^s_m(A)\subseteq Q^r_m(A)\subseteq Q(A)}$, with center 
${Z(Q^s_m(A))=Z(Q^r_m(A))=Z(Q(A))}$ and common unity. If $A$ is semiprime, 
then $Q^s_m(A)$ is uniquely determined, up to an isomorphism identical on $A$, 
by the axioms: 
\begin{enumerate}

\item $A$ is a subring of $Q^s_m(A)$; 

\item for any ${q\in Q^s_m(A)}$ there exists ${I\in \mathcal{E}(A)}$, 
${q I, I q\subseteq A}$; 

\item if ${q\in Q^s_m(A)}$, ${q I=\{0\}}$ or ${I q=\{0\}}$ (only  
${q I=\{0\}}$ (${I q=\{0\}}$)) for some ${I\in \mathcal{E}(A)}$, 
then ${q=0}$; 

\item for any ${I, J\in \mathcal{E}(A)}$, ${\phi\in \Hom(I, A)_A}$, 
${\psi\in \Hom_A(J, A)}$ such that ${y (\phi x)=(y \psi) x}$, ${x\in I}$, 
${y\in J}$, there is ${q\in Q^s_m(A), \phi=l_q|_I, \psi=r_q|_J}$

\end{enumerate}
\cite[Propositions~1.4,~1.6,~Lemma 1.5 and their proofs]{Pas}, 
\cite[Propositions~2.2.3]{BMM}. If $A$ is a $F$--algebra, then  
${Q(A), Q^r_m(A), Q^s_m(A)}$ are $F$--algebras (for their implementations, 
the action of ${f\in F}$ is multiplication by ${[(f \Id_A, A)]}$) and 
can be defined for $A$ as $F$--algebra, since for the ring $A$ and any 
${I\in \mathcal{D}_r(A), \phi\in \Hom(I, A)_A}$ we have 
${F I\in \mathcal{D}_r(A)}$, $\phi$ can be extended to a $F$--homomorphism 
${\phi\in \Hom(F I, A)_A}$, ${(\phi, I)\sim (\phi, F I)}$ 
(from ${\sum\limits_i f_i x_i=0}$ for some ${f_i\in F, x_i\in I}$ it follows 
that ${\Bigl(\sum\limits_i f_i \phi x_i\Bigr) y=0}$ for all ${y\in A}$ and 
${\sum\limits_i f_i \phi x_i=0}$ (left exactness of $A$)). 
If $A$ is semiprime, then ${Z(Q^s_m(A))={}_A C}$ and there is an isomorphism 
of ${}_A C$--algebras ${{}_A C A\subseteq Q^s_m(A)}$ and ${}_A Q$ 
identical on $A$ \cite[Lemma~3.1]{BM1}, \cite[Theorem~3.17]{BM}. 

\section{Central closure and orthogonal completeness of semiprime 
algebra}

We will begin a brief discussion of the construction of the central closure 
of a semiprime algebra with the equivalence of the approaches to its 
construction from \cite{EMO, BM, BMM} and \cite{Raz, Wis}. The information 
from module theory is given for right unitary modules over one associative 
ring $A$ with 1. A module $I_A$ is called \emph{injective} if for any 
${\alpha\in \Hom(M, N)_A}$, ${\Ker \alpha=\{0\}}$, and 
${\beta\in \Hom(M, I)_A}$ there exists ${\gamma\in \Hom(N, I)_A}$, 
${\beta=\gamma \alpha}$. The injectivity of $I_A$ is equivalent to the 
impossibility of embedding $I_A$ as a proper essential submodule into 
$A$--module and the separability of the images of $I_A$ under 
embeddings into $A$--modules as direct summands. 

Any module $M_A$ can be embedded in a suitable injective module $I_A$ and, 
moreover, in an injective $I_A$ such that the image $M_A$ in $I_A$ is an 
essential submodule of $I_A$. In the latter case, $I_A$ is called the 
\emph{injective hull of} $M_A$. Its choice is unique up to an isomorphism 
identical on $M_A$ (here and below $M_A$ is a submodule of its injective 
hulls). Further, $I(M_A)$ is any fixed\linebreak injective hull of $M_A$ and 
${P(M_A)=\End(I(M_A))_A M}$ is the smallest of the fully invariant 
submodules of $I(M_A)$ containing $M_A$. The Jacobson radical 
$\mathcal{J}(\End(I(M_A))_A)$ of the ring $\End(I(M_A))_A$ is the set of 
all ${\phi\in \End(I(M_A))_A}$, $\Ker \phi$ is an essential submodule of 
$I(M_A)$,\linebreak the ring $\End(I(M_A))_A/\mathcal{J}(\End(I(M_A))_A)$ is regular 
in the sense of von Neumann. A module\linebreak $Q_A$ is called \emph{quasi-injective} if 
the homomorphisms into $Q_A$ of its submodules can be extended to 
endomorphisms of $Q_A$. This is equivalent to fully invariance of $Q_A$ in 
$I(Q_A)$, ${Q_A=P(Q_A)}$. Each module $M_A$ is included in a quasi-injective 
submodule ${P(M_A)\subseteq I(M_A)}$, which is called its 
\emph{quasi-injective hull}. An isomorphism of injective hulls of $M_A$, 
identical\linebreak on $M_A$, induces an isomorphism of the quasi-injective hulls of 
$M_A$ contained in them. From the quasi-injectivity of $I(M_A)$ and the 
fully invariance of $P(M_A)$ in $I(M_A)$ it follows that 
${P(M_A)=\End(P(M_A))_A M}$ and ${\alpha: \phi\longmapsto \phi|_{P(M_A)}}$, ${\phi\in \End(I(M_A))_A}$, is an epimorphism of $\End(I(M_A))_A$ onto 
$\End(P(M_A))_A$. Up to an isomorphism identical on $M_A$, the quasi-injec\-tive 
hull of $M_A$ is any quasi-injective module $Q_A$ such that $M_A$ is an 
essential submodule of $Q_A$ and ${Q=\End(Q_A)_A M}$. It can also be defined as 
a \emph{minimal quasi-injective extension of $M_A$}, i.e. a quasi-injective 
module ${Q_A\supseteq M_A}$ such that embeddings of $M_A$ into quasi-injective 
$A$--modules extend to embeddings of $Q_A$ into them 
\cite[Ch.~4,~p.~146--153,~168]{Lamb}, 
\cite[Vol.~1,~Ch.~3,~p.~218--220,~Vol.~2,~Ch.~19,~ 
p.~103--107]{Fe}. The simplest example of the difference between the conditions of 
injectivity and quasi-injectivity is the cyclic groups of orders $p^n$, 
$p$ is prime and ${n\geq 1}$, the common injective hull of these quasi-injective 
$\mathbb{Z}$--modules is the quasi-cyclic group of exponent $p$.  

Let $R$ be a non-zero semiprime $F$--algebra, ${R_{M(R)'}={}_F R_{M(R)'}}$ be 
$R$ as a right unitary module over the $F$--algebra $M(R)'$. We will consider 
all unitary modules over the ring $M(R)'$ and, in particular, 
${I(R)=I(R_{M(R)'})}$, ${P(R)=P(R_{M(R)'})}$ as modules over the $F$--algebra 
$M(R)'$, identifying the left and right actions of ${f\in F}$ with the 
action of ${f \Id_R\in M(R)'}$, and the rings of their $M(R)'$--endomorphisms 
as $F$--algebras. The commutative semiprime $F$--algebra 
${\CM(R)=\End(P(R))_{M(R)'}}$ is called the \emph{Martindale centroid of $R$}. 
Since homomorphisms in $P(R)$ of its (essential) submodules extend to 
(uniquely determined) elements of 
${\CM(R), \Ker \alpha=\mathcal{J}(\End(I(R))_{M(R)'}), \CM(R)}$ is regular 
\cite[Lemma~3.1,~Proposition~3.1,~p.~42,~43]{Raz}, 
${\Ann_F R=\Ann_F P(R)}$, the $F$--algebra  
\[
\End(R)_{M(R)'}\ =\ \{\phi\in \End_F(R)\mid [\phi, M(R)]=\{0\}\} 
\]
(\emph{centroid of $R$}) is embedded in $\CM(R)$, 
${\CM(R)\cong \End(I)_{M(R)'}}$ for any essential fully invariant 
submodule ${I\subseteq P(R)}$. On $P(R)$ one can introduce the structure of 
a $\CM(R)$--algebra\linebreak by extending the operation of multiplication of 
$R$ on $P(R)$ by $\CM(R)$--linearity: 
\begin{multline*}
\biggl(\sum_i \phi_i a_i\biggr) \biggl(\sum_j \psi_j b_j\biggr)\ =\ 
\sum_{i, j} (\phi_i \psi_j)(a_i b_j)\ =
\\ 
\biggl(\sum_j \psi_j b_j\biggr)\biggl(\sum_i \phi_i l_{a_i}\biggr)
\ =\ \biggl(\sum_i \phi_i a_i\biggr)\biggl(\sum_j \psi_j r_{b_j}\biggr)
\quad (\phi_i, \psi_j\in \CM(R),\ a_i, b_j\in R)\,.
\end{multline*}
The algebra $P(R)$ is called the \emph{central closure of $R$}, it inherits 
semiprimeness and all homogeneous identities of $R$ 
\cite[Proposition~3.1,~p.~43]{Raz}, 
\begin{gather*}
M(R)'\cong M^{P(R)}(R)'\ =\ F \Id_{P(R)}+M^{P(R)}(R)\,,\quad 
M(P(R))'\ =\ \CM(R) M^{P(R)}(R)'\,,
\\ 
\CM(R)\ =\ \End(P(R))_{M(P(R))'}\ =\ Z(M(P(R))')\,,\quad 
P(R)\ =\ P(P(R)_{M(P(R))'})\,.
\end{gather*}
If $R$ is prime, then $P(R)$ is prime, $\CM(R)$ is a field 
\cite[Proposition~3.2,~p.~44]{Raz} and for $R$ with a minimal ideal $I$, such $I$ is 
the smallest non-zero ideal (fully invariant submodule) of $P(R)$, 
${\CM(R)\cong \End(I)_{M(R)'}}$, ${I(I_{M(R)'})=I(R)\ne I=P(I_{M(R)'})}$.
As a consequence, if $R$ is simple, then ${R=P(R)}$, ${\CM(R)=\End(R)_{M(R)'}}$.

If $P'$ is another quasi-injective hull of $R_{M(R)'}$, then there exist 
isomorphisms of $M(R)'$--modules ${\tau: P(R)\longrightarrow P'}$, identical on 
$R$, and ${\hat{\tau}: \phi\longmapsto \tau \phi \tau^{-1}, \phi\in \CM(R)}$, 
of $F$--algebras $\CM(R)$ and ${\CM(R)'=\End(P')_{M(R)'}}$, allowing us to 
consider $P'$ as a $\CM(R)$--algebra and $\tau$ as an isomorphism of 
$\CM(R)$--algebras, 
\[
\tau\,:\ \sum_i \phi_i a_i\longmapsto 
\sum_i (\hat{\tau} \phi)(\tau a_i)\quad (\phi_i\in \CM(R),\ a_i\in R)\,.
\]
In what follows, isomorphisms of central closures of $R$ are isomorphisms of 
this type. 

Let us move on to the construction of the central closure and extended centroid 
$R$ from \cite{Mart, EMO, BM}. On the set  
${\mathcal{C}(R)=\{(\phi, I)\mid I\in \mathcal{E}(R),\ 
\phi\in \Hom(I, R)_{M(R)'}\}}$ we introduce the equi\-valence relation $\sim$: 
${(\phi, I)\sim (\phi', I')}$ if ${\phi=\phi'}$ on 
${I\cap I'}$ (for some ${I''\in \mathcal{E}(R)}$, ${I''\subseteq I\cap I'}$, 
\cite[Corollary~2.3]{BM}). Then the quotient set 
${{}_R C=\mathcal{C}(R)/\sim}$ with the operations  
\[
[(\phi, I)]+[(\psi, J)]\ =\ [(\phi+\psi, I\cap J)]\,,\quad 
[(\phi, I)] [(\psi, J)]\ =\ [(\phi \psi, \psi^{-1}(I))]
\]
on equivalence classes of pairs ${(\phi, I), (\psi, J)\in \mathcal{C}(R), 
\psi^{-1}(I)=\{r\in J\mid \psi r\in I\}\in \mathcal{E}(R)}$, is a regular 
associative commutative ring with unity ${1_{{}_R C}=[(\Id_R, R)]}$  
\cite[Lemma~2.4,~Theorem~2.5]{BM}. The ring ${}_R C$ is called the 
\emph{extended centroid of $R$}. To any ${I\in \mathcal{E}(R)}$ there 
corresponds an embedding of $\End(I)_{M(R)'}$ into ${}_R C$, 
${\gamma\longmapsto [(\gamma, I)]}$, ${\gamma\in \End(I)_{M(R)'}}$, and, 
in particular, ${\End(R)_{M(R)'}\hookrightarrow {}_R C}$, ${}_R C$ can be 
considered as an $F$--algebra, 
\[
f [(\phi, I)]\ =\ [(f \Id_R, R)] [(\phi, I)]\ =\ [(f \phi, I)]\quad 
(f\in F,\ (\phi, I)\in \mathcal{C}(R))\,,
\]
${\Ann_F {}_R C=\Ann_F R}$. The tensor product of $F$--modules 
${{}_R C\mathbin{\otimes_F} R}$ can be transformed into 
a ${}_R C$--algebra with  
\begin{gather*}
[(\phi, I)] \biggl(\sum_i [(\phi_i, I_i)]\otimes a_i\biggr)\ =\ 
\sum_i [(\phi \phi_i, \phi_i^{-1}(I))]\otimes a_i\,,
\\
\biggl(\sum_i [(\phi_i, I_i)]\otimes a_i\biggr) 
\biggl(\sum_j [(\psi_j, J_j)]\otimes b_j\biggr)\ =\ 
\sum_{i, j} [(\phi_i \psi_j, \psi_j^{-1}(I_i))]\otimes (a_i b_j)\,,
\end{gather*}
${(\phi, I), (\phi_i, I_i), (\psi_j, J_j)\in \mathcal{C}(R), a_i, b_j\in R}$. 
An element ${x\in {}_R C\mathbin{\otimes_F} R}$ such that for some ${x_k\in R}$, 
${(\tau_k, H_k)\in \mathcal{C}(R), H\in \mathcal{E}(R), 
H\subseteq \bigcap\limits_k H_k, x=\sum\limits_k [(\tau_k, H_k)]\otimes x_k, 
\sum\limits_k (\tau_k y) (x_k \alpha)=0}$ for all ${y\in H}$, ${\alpha\in M(R)'}$ 
is called \emph{vanishing}, and this record of $x$ is called its 
\emph{$H$--vanishing representation}. Any record of $x$ is its $H$--vanishing 
representation of $x$ for suitable ${H\in \mathcal{E}(R)}$ 
\cite[Lemma~2.7]{BM}. The vanishing elements of ${{}_R C\mathbin{\otimes_F} R}$ form 
an ideal $M$, which is the greatest among ${J\lhd {}_R C\mathbin{\otimes_F} R}$, 
\[
J\cap (1_{{}_R C}\otimes R)\ =\ \{0\}\,,\quad 
I_0\ =\ \{[(\lambda, I)]\otimes r-1_{{}_R C}\otimes \lambda r\mid 
(\lambda, I)\in \mathcal{C}(R),\ r\in I\}\subseteq J
\]
\cite[Lemmas~2.6,~2.11]{BM}. We identify $R$ with its image in the 
algebra ${{}_R Q=({}_R C\mathbin{\otimes_F} R)/M}$ under the action of the 
$F$--embedding ${r\longmapsto 1_{{}_R C}\otimes r+M}$, ${r\in R}$ 
\cite[Lemma~2.8,~Corollary~2.9]{BM} and, for brevity, denote  
${1_{{}_R C}\otimes r+M}$ and ${[(\phi, I)]\otimes r+M}$ by $r$ and 
${[(\phi, I)] r}$, ${r\in R}$, ${(\phi, I)\in \mathcal{C}(R)}$, where 
${[(\phi, I)] r=\phi r}$ for ${r\in I}$. The algebra ${{}_R Q={}_R C R}$ is 
also called the \emph{central closure of $R$}, 
\[
M(R)'\cong M^{{}_R Q}(R)'\ =\ F \Id_{{}_R Q}+M^{{}_R Q}(R)\,,\quad 
M({}_R Q)'\ =\ {}_R C M^{{}_R Q}(R)'\,.
\]
Since $R$ is an essential $M(R)'$--submodule 
($M^{{}_R Q}(R)'$--submodule) of ${}_R Q$, 
\[
\biggl(\sum_m [(\chi_m, U_m)] r_m\biggr) \alpha\ =\ 
\sum_m [(\chi_m, U_m)] (r_m \alpha)\quad (\alpha\in M(R)',\ 
(\chi_m, U_m)\in \mathcal{C}(R),\ r_m\in R)
\]
\cite[Lemma~2.10]{BM}, non-zero $M(R)'$--submodules of ${}_R Q$ 
intersect $R$ in non-zero ideals of $R$, ${}_R Q$ inherits the 
semiprimeness (primeness) of $R$. If $R$ is prime, then ${}_R C$ is a field 
\cite[Theorem~2.1]{EMO}. According to \cite[Theorem~2.15]{BM}, 
\cite[Theorem~2.4]{Cab3}, ${}_R Q$ has the following properties: 
\begin{enumerate}

\item $R$ is embedded in ${}_R Q$ as an $F$--subalgebra ($R$ and its image in 
${}_R Q$ are identified), ${{}_R Q=\End({}_R Q)_{M({}_R Q)'} R}$; 

\item for any ${q\in {}_R Q}$ there is ${I\in \mathcal{E}(R)}$, 
${I (q M(R)')+(q M(R)') I\subseteq R}$, where $q M(R)'$ is a 
$M(R)'$--submodule of ${}_R Q$ generated by $q$;

\item ${I (q M(R)')\ne \{0\}}$ and ${(q M(R)') I\ne \{0\}}$ for any 
${0\ne q\in {}_R Q}$, ${I\in \mathcal{E}(R)}$;

\item any ${\phi\in \Hom(I, R)_{M(R)'}}$, ${I\lhd R}$, extends to some 
${\overline{\phi}\in \End({}_R Q)_{M({}_R Q)'}}$;

\item the action ${}_R C$ on ${}_R Q$ is exact, 
${{}_R C\cong {}_R C \Id_{{}_R Q}=\End({}_R Q)_{M({}_R Q)'}\cong 
{}_{{}_R Q} C}$, where the last isomorphism 
${\gamma\longmapsto [(\gamma, {}_R Q)]}$, 
${\gamma\in \End({}_R Q)_{M({}_R Q)'}}$, and ${{}_R Q={}_{{}_R Q} Q}$.

\end{enumerate}
Properties 1--4 determine ${}_R Q$ uniquely up to isomorphism in the sense 
that for any $F$--algebra $Q$ with such properties, 
${\End(Q)_{M(Q)'}\cong {}_R C}$ as $F$--algebras and 
${Q\cong {}_R Q}$ as ${}_R C$--algebras via an isomorphism  
identical on $R$ \cite[the~proof~of~Theorem~2.4]{Cab3}. 

\begin{prop}
The algebra $P(R)$ has properties 1--4, the algebra ${}_R Q$ is the 
quasi-injective hull of the $M(R)'$--module $R$, ${\CM(R)\cong {}_R C}$ and 
${P(R)\cong {}_R Q}$.
\end{prop}

\begin{proof}
If ${\phi: N\longrightarrow M}$ is a homomorphism into a module $M$ of its 
submodule $N$, $K$ is an essential submodule of $M$, then 
${{}_{\phi, K} N=\phi^{-1}(\phi(N)\cap K)=\{a\in N\mid \phi a\in K\}}$
is an essential submodule of $N$ ($M$ for an essential $N$ in it), since for 
any submodule ${\{0\}\ne N'\subseteq N}$ either 
${\phi(N')=\{0\}}$, ${N'\subseteq {}_{\phi, K} N}$, or ${\phi(N')\ne \{0\}}$, 
${\phi(N')\cap K\ne \{0\}}$,
\[
{}_{\phi, K} N\cap N'\supseteq \phi^{-1}(\phi(N')\cap K)\cap N'\ \ne\ \{0\}\,.
\]
So, ${{}_{\phi} R={}_{\phi, R} R\in \mathcal{E}(R)}$ for any ${\phi\in \CM(R)}$. 
If ${q=\sum\limits_{i=1}^k \nu_i q_i\in P(R)}$ for some 
${k\geq 1, \nu_i\in \CM(R)}$, ${q_i\in R}$, then 
${D=\bigcap\limits_{i=1}^k {}_{\nu_i} R\in \mathcal{E}(R), 
D (q M(R)')+(q M(R)') D\subseteq \sum\limits_{i=1}^k \nu_i(D)\subseteq R}$,
where $q M(R)'$ is a $M(R)'$--submodule of $P(R)$ generated by $q$. Due to 
the semiprimeness of $R$, for ${q\ne 0}$ 
\[
\{0\}\ \ne\ q M(R)'\cap J\lhd R\,,\quad 
\{0\}\ \ne\ (q M(R)'\cap J)^2\subseteq J (q M(R)')\cap (q M(R)') J
\quad (J\in \mathcal{E}(R))\,.
\]
The presence of property 4 in $P(R)$ guarantees its quasi-injectivity as 
a $M(R)'$--module and ${\CM(R)=\End(P(R))_{M(P(R))'}}$.
From this and the arguments of 
\cite[the~proof~of~Theorem~2.4]{Cab3}, it 
follows that 
${\delta: \phi\longmapsto [(\phi|_{{}_{\phi} R}, {}_{\phi} R)], 
\phi\in \CM(R)}$,
is an isomorphism of $F$--algebras $\CM(R)$ and ${}_R C$, 
${\rho: \sum\limits_i \phi_i r_i\longmapsto 
\sum\limits_i \delta(\phi_i) r_i, \phi_i\in \CM(R), r_i\in R}$, is an 
isomorphism of $\CM(R)$--algebras $P(R)$ and ${}_R Q$, identical on $R$.

Let us prove that ${}_R Q$ is the quasi-injective hull of the $M(R)'$--module 
$R$. Since $R$ is an essential $M(R)'$--submodule of ${}_R Q$ and
${{}_R C \Id_{{}_R Q}=\End({}_R Q)_{M({}_R Q)'}}$ (see above),
\[
{}_R Q\ =\ {}_R C R\ =\ \End({}_R Q)_{M({}_R Q)'} R\ =\ 
\End({}_R Q)_{M(R)'} R\,,
\]
it suffices to show that any ${\phi\in \Hom(M, {}_R Q)_{M(R)'}}$, 
${M_{M(R)'}\subseteq {}_R Q}$, extends to an element of $\End({}_R Q)_{M(R)'}$. 
We can consider $M$ to be an essential $M(R)'$--submodule of ${}_R Q$, since 
otherwise it can be replaced by ${M\oplus K}$, where $K$ is the maximal among 
${L_{M(R)'}\subseteq {}_R Q}$, ${L\cap M=\{0\}}$, and $\phi$ can be 
extended to ${\phi\in \Hom(M\oplus K, {}_R Q)_{M(R)'}}$, setting ${\phi|_K=0}$ 
\cite[Lemma~1,~p.~104]{Lamb}. Then 
${M'={}_{\phi, R} M\cap R\in \mathcal{E}(R)}$ is an essential 
$M(R)'$--submodule of ${}_R Q$, by property 4 
${\phi'=\phi|_{M'}\in \Hom(M', R)_{M(R)'}}$ can be extended to 
${\overline{\phi'}\in \End({}_R Q)_{M({}_R Q)'}}$. Identifying ${}_R Q$ with the 
submodule of $I(R)$, we obtain ${P(R)=P({}_R Q)=\End(I(R))_{M(R)'} R}$ (elements 
$\End({}_R Q)_{M(R)'}$ extend to elements of $\End(I(R))_{M(R)'}$). 
In view of the uniqueness of the definition of extensions of homomorphisms in $P(R)$ 
of its essential submodules to endomorphisms of $P(R)$, 
${\psi=\phi-\overline{\phi'}|_M\in \Hom(M, {}_R Q)_{M(R)'}}$ with an essential 
kernel ${\Ker \psi\supseteq M'}$ in ${}_R Q$ and $P(R)$ is equal to zero,
${\phi=\overline{\phi'}}$ on $M$. Hence ${{}_R Q=P(R)}$, 
\[
\CM(R)\ =\ \End({}_R Q)_{M(R)'}\ =\ 
\End({}_R Q)_{M({}_R Q)'}\ =\ {}_R C \Id_{{}_R Q}\cong {}_R C\,.
\]
\end{proof}                   

If $P(R)$ is a finitely generated $\CM(R)$--module (endofinite 
$M(R)'$--module), ${P(R)=I(R)}$ 
\cite[Vol.~2,~Theorem 19.14~A,~p.~111]{Fe}.

Based on \cite{B6, BMich, BMM}, we will construct a general construction of 
the orthogonal completion of the semiprime algebra $R$. Let us transform 
${B(R)=\{\alpha\in \CM(R)\mid \alpha=\alpha^2\}}$ into a Boolean ring with 
multiplication of $\CM(R)$, addition 
${\alpha\oplus \beta=(\alpha-\beta)^2=\alpha+\beta-2 \alpha \beta}$ and 
partial order ${\alpha\leq \beta}$ if 
${\alpha \beta=\alpha}$, ${\alpha, \beta\in B(R)}$
\cite[Sec.~1.24]{BMich}, \cite[Proposition~11,~p.~49]{Lamb}. The set 
${B\subseteq B(R)}$ is \emph{dense} in $B(R)$ if ${\Ann_{B(R)} B=\{0\}}$ 
(${\Ann_{\CM(R)} B=\{0\}}$, due to the regularity of $\CM(R)$), and 
\emph{orthogonal} if ${\alpha \beta=0}$ for all ${\alpha, \beta\in B}$, 
${\alpha\ne \beta}$. If ${B\subseteq B(R)\setminus \{0\}}$ and 
\begin{enumerate}

\item ${\alpha \beta\in B}$ for all ${\alpha, \beta\in B}$; 

\item if ${\alpha\in B}$, then ${\beta\in B}$ for all ${\beta\in B(R)}$, 
${\alpha\leq \beta}$; 

\item for any ${\alpha\in B(R)}$ either ${\alpha\in B}$ or 
${\Id_{P(R)}-\alpha\in B}$,

\end{enumerate}
then $B$ is an \emph{ultrafilter} in $B(R)$. Note that ultrafilters in $B(R)$ 
are exhausted by sets ${B(R)\setminus P}$,\linebreak $P$ is the maximal ideal of 
$B(R)$ ($B(R)/P$ is a field of 2 elements) 
\cite[Proposition~1,~p.~59]{Lamb}. 
As before for $R$, we choose ${I(P(R))=I(P(R)_{M(P(R))'})}$ for the 
$\CM(R)$--algebra $P(R)$.

An associative ring $A$ with 1 is called \emph{right self-injective} if the 
$A$--module $A_A$ is injective, i.e. for any right ideal ${I\subseteq A}$ 
and ${\phi\in \Hom(I, A)_A}$ there is ${x\in A}$, ${\phi y=x y}$ for all 
${y\in I}$. The\linebreak right self-injectivity of $A$ is equivalent to the 
quasi-injectivity of $A_A$ \cite[Lemma~1,~p.~147]{Lamb}, 
\cite[Example~19.1.4,~p.~103]{Fe}. Naturally, for a commutative $A$, indicating 
the side of self-injec-\linebreak tivity is superfluous. 
Notice that in the associative $A$ 
with 1 for any ${e_i=e_i^2\in A}$, ${i=1, 2}$, ${e_1 e_2=e_2 e_1}$, 
${e=e_1+e_2-e_1 e_2=(e_1 (1-e_2)-e_2)^2=e^2}$, ${e e_i=e_i e=e_i}$, ${i=1, 2}$. 
Therefore, if ${e_i=e_i^2\in A}$, ${e_i e_j=e_j e_i}$, 
${i, j=1, \ldots, k}$, ${k\geq 1}$, then there exists ${e=e^2\in A}$, 
${e e_i=e_i e=e_i}$, ${i=1, \ldots, k}$.

\begin{rem} 
If ${E\subseteq \CM(R)}$, ${\Ann_{\CM(R)} E=\{0\}}$, then 
\[
E^{\perp}\ =\ \{x\in I(P(R))\mid E x=\{0\}\}\ =\ \{0\}\,.
\]
\end{rem}

\begin{proof}
Since $E^{\perp}$ is a $M(P(R))'$--submodule of $I(P(R))$, from 
${E^{\perp}\ne \{0\}}$ it follows that\linebreak ${E^{\perp}\cap P(R)\ne \{0\}}$, 
${M=E^{\perp}\cap R\ne \{0\}}$. By complementing the $M(R)'$--module 
($M^{P(R)}(R)'$--module) $M$ to an essential $M(R)'$--submodule 
${M\oplus M'\subseteq P(R)}$ and extending the projection 
${M\oplus M'\longrightarrow M}$ to an element of $\CM(R)$, one can choose ${\alpha\in B(R)}$ such 
that ${\alpha|_M=\Id_M}$, ${\alpha(M')=\{0\}}$. Therefore, 
${\alpha \beta(M\oplus M')=\{0\}}$, ${\alpha \beta=0}$ for all 
${\beta\in E}$, ${\alpha\in \Ann_{\CM(R)} E=\{0\}}$, a contradiction.
\end{proof}

\begin{lemma}
The ring $\CM(R)$ is self-injective and for any ${S\subseteq I(P(R))}$ there 
is a unique ${\alpha\in B(R)}$, ${\Ann_{\CM(R)} S=\alpha \CM(R)}$.
\end{lemma}

\begin{proof}
For any ${I\lhd \CM(R)}$, ${\phi\in \Hom_{\CM(R)}(I, \CM(R))}$ the mapping 
\[
\tau_{\phi}\,:\ \sum_i \alpha_i x_i\longmapsto \sum_i (\phi \alpha_i) x_i
\quad (\alpha_i\in I, x_i\in R)
\]
is correctly defined, since from ${\sum\limits_i \alpha_i x_i=0}$ it follows 
that ${\alpha_i=\alpha_i^2 \beta_i}$ for some ${\beta_i\in \CM(R)}$, there is 
${\gamma\in I\cap B(R)}$, ${\gamma \alpha_i \beta_i=\alpha_i \beta_i\in I\cap B(R)}$, 
${\gamma \alpha_i=\alpha_i}$ for some $i$, 
\[
\sum_i (\phi \alpha_i) x_i\ =\ \sum_i (\phi (\gamma \alpha_i)) x_i\ =\ 
(\phi \gamma) \sum_i \alpha_i x_i\ =\ 0\,.
\]
Hence ${\tau_{\phi}\in \Hom(I R, P(R))_{M(R)'}}$ extends to 
${\overline{\tau_{\phi}}\in \CM(R)}$, 
${\overline{\tau_{\phi}} \alpha x=(\phi \alpha) x}$, 
${\phi \alpha=\overline{\tau_{\phi}} \alpha}$ for all ${\alpha\in I}$, 
${x\in R}$, the ring $\CM(R)$ is self-injective (this can be directly 
deduced from \cite[Theorem~19.27,~p.~123]{Fe}). 
If ${S\subseteq I(P(R))}$, $V$ and $V'$ are maximal orthogonal subsets of 
${\Ann_{B(R)} S}$ and ${\Ann_{B(R)} \Ann_{B(R)} S}$, then ${W=V\cup V'}$ 
is a dense orthogonal subset of $B(R)$  
\cite[the~proof of~Lemma~1.6]{BMich}, there is an embedding 
${\mu\in \Hom\Bigl(\CM(R), 
\prod\limits_{\gamma\in W} \gamma \CM(R)\Bigr)_{\CM(R)}, 
\mu \phi: \gamma\longmapsto \gamma \phi}$, 
${\phi\in \CM(R), \gamma\in W,}$ and 
${\nu\in \Hom(\prod\limits_{\gamma\in W} \gamma \CM(R), \CM(R))_{\CM(R)}, 
\nu \mu=\Id_{\CM(R)}}$ (self-injectivity\linebreak of $\CM(R)$). Therefore for 
${\theta\in \prod\limits_{\gamma\in W} \gamma \CM(R)}$, 
${\theta(V)=\{0\}}$, ${\theta|_{V'}=\Id_{V'}}$, and 
${\nu \theta=\eta\in \CM(R)}$
\begin{gather*}
\nu (\theta \alpha)\ =\ \nu 0\ =\ 0\ =\ \eta \alpha\,,\quad 
\quad \nu (\theta \beta)\ =\ \nu \mu \beta\ =\ \beta\ =\ \eta \beta\quad 
(\alpha\in V,\ \beta\in V')\,,
\\
\eta^2-\eta\in \Ann_{\CM(R)} W\ =\ \Ann_{B(R)} W\ =\ \{0\}\,,\quad 
\eta\in \Ann_{B(R)} V\ =\ \Ann_{B(R)} \Ann_{B(R)} S
\end{gather*} 
\cite[Sec.~1.5]{BMich},
${(\Id_{P(R)}-\eta) \Ann_{B(R)} \Ann_{B(R)} S\subseteq \Ann_{B(R)} W=\{0\}}$ and 
$\Ann_{B(R)} \Ann_{B(R)} S=\eta B(R)$. According to Remark 2.2, 
${S (\Id_{P(R)}-\eta)\subseteq W^{\perp}=\{0\}}$ and so, 
\[
\Ann_{B(R)} S\ =\ (\Id_{P(R)}-\eta) B(R)\,,\ 
\Ann_{\CM(R)} S\ =\ \CM(R) \Ann_{B(R)} S\ =\ (\Id_{P(R)}-\eta) \CM(R)\,.
\]
It remains to note that the generators of the principal ideal of $B(R)$ are 
uniquely determined, since if ${\alpha B(R)=\alpha' B(R)\lhd B(R)}$, then 
${\alpha \alpha'=\alpha=\alpha'}$. 
\end{proof}

It follows that for any ${B\subseteq \CM(R)}$ there exists a unique 
${\alpha\in B(R)}$, 
\begin{gather*}
\Ann_{\CM(R)} B P(R)\ =\ \Ann_{\CM(R)} B\ =\ \alpha \CM(R)\,,
\\ 
\Ann_{B(R)} B\ =\ B(R)\cap \alpha \CM(R)\ =\ \alpha B(R)\,, 
\end{gather*}
i.e. the Boolean ring $B(R)$ is \emph{orthogonally complete} in the sense of 
\cite[Sec.~1.7]{BMich} (this can also be deduced from the proof of Lemma 2.3 and 
condition 1 of \cite[Lemma 1.6]{BMich}). For a $m.s.p.$--algebra $R$, the self-injectivity of $\CM(R)$ at once follows from 
\cite[Theorem~4.3]{Cab3} 
(in the form of \cite[before~Corollary~2.11]{Gol10}) and 
\cite[Theorem~8.4]{BMich}.
We call a set ${S\subseteq I(P(R))}$ \emph{orthogonally complete} if for any 
dense orthogonal ${\{\xi_a\}_{a\in I}\subseteq B(R), 
\{x_a\}_{a\in I}\subseteq S}$ there exists ${x\in S, \xi_a x=\xi_a x_a}$ 
for all\linebreak ${a\in I}$. Following \cite{BMich}, we denote such $x$ by 
${{\sum\limits_{a\in I}}^{\perp} \xi_a x_a}$, the equalities 
${\xi_a x=\xi_a x_a}$, ${a\in I}$, determine the choice of $x$ uniquely 
(Remark 2.2). The \emph{orthogonal completion} $O(S)$ of the set 
${S\subseteq I(P(R))}$ is the intersection of all orthogonally complete 
${S'\subseteq I(P(R)), S\subseteq S'}$.
As in \cite[Lemma~1]{B6}, \cite[Proposition~8.4]{BMich}, we obtain 

\begin{prop}
1) ${I(P(R))=O(I(P(R)))}$ is orthogonally complete; 2) if 
${S\subseteq I(P(R))}$, ${O(S)=O(O(S))}$ is the set of all 
${{\sum\limits_{a\in I}}^{\perp} \xi_a x_a\in I(P(R))}$ for a dense 
orthogonal ${\{\xi_a\}_{a\in I}\subseteq B(R)}$, 
${\{x_a\}_{a\in I}\subseteq S}$; 
3) $O(P(R))$ is a semiprime $\CM(R)$--algebra with the operations 
\[
{\sum\limits_{a\in I}}^{\perp} \xi_a x_a\star 
{\sum\limits_{b\in J}}^{\perp} \chi_b y_b\ =\ 
{\sum\limits_{a\in I,\ b\in J}}^{\perp} \xi_a \chi_b (x_a\star y_b)\quad 
(\star=+, \cdot)
\] 
for dense orthogonal ${\{\xi_a\}_{a\in I}, 
\{\chi_b\}_{b\in J}\subseteq B(R), \{x_a\}_{a\in I}, 
\{y_b\}_{b\in J}\subseteq P(R)}$; 4) $O(P(R))$ is a fully invariant 
$M(P(R))'$--submodule $I(P(R))$, 
\[
\End(O(P(R)))_{M(P(R))'}\ =\ \End(O(P(R)))_{M(O(P(R)))'}\ =\ 
\CM(R) \Id_{O(P(R))}\,;
\] 
5) $O(R)$ is a semiprime $F$--subalgebra of $O(P(R))$ and if any element of 
$P(R)$ is a $\CM(R)$--linear combination of $n$ elements of $R$ for 
some ${n\geq 1}$, then ${O(P(R))=\CM(R) O(R)}$; 
6) for any ${S=O(S)\subseteq I(P(R))}$ there is ${x\in S}$, 
${\Ann_{\CM(R)} S=\Ann_{\CM(R)} x}$. 
\end{prop}

\begin{proof} To each dense orthogonal 
${\Xi=\{\xi_a\}_{a\in I}\subseteq B(R)}$ there correspond an embedding 
\[
\psi_{\Xi}\in 
\Hom\Bigl(I(P(R)), \prod\limits_{a\in I} \xi_a I(P(R))\Bigr)_{M(P(R))'}\,,\quad
\psi_{\Xi} x\,:\ a\longmapsto \xi_a x\quad 
(x\in I(P(R)),\ a\in I) 
\]
(Remark 2.2), and ${\phi_{\Xi}\in 
\Hom\Bigl(\prod\limits_{a\in I} \xi_a I(P(R)), I(P(R))\Bigr)_{M(P(R))'}}$, 
${\phi_{\Xi} \psi_{\Xi}=\Id_{I(P(R))}}$. For any 
${\{x_a\}_{a\in I}\subseteq I(P(R))}$ and 
${x\in \prod\limits_{a\in I} \xi_a I(P(R))}$, 
${x: a\longmapsto \xi_a x_a}$, ${a\in I}$, 
\[
\phi_{\Xi} (\xi_a x)\ =\ \xi_a (\phi_{\Xi} x)\ =\ 
\phi_{\Xi} \psi_{\Xi}(\xi_a x_a)\ =\ \xi_a x_a\quad (a\in I)\,,
\]
${\phi_{\Xi} x={\sum\limits_{a\in I}}^{\perp} \xi_a x_a\in I(P(R))}$. So,  
${I(P(R))=O(I(P(R)))}$, the orthogonal completion $O(S)$ is defined for 
any ${S\subseteq I(P(R))}$. By definition, $O(S)$ includes the set $\hat{S}$ 
of all ${{\sum\limits_{a\in I}}^{\perp} \xi_a x_a\in I(P(R))}$\linebreak for 
a dense orthogonal 
${\{\xi_a\}_{a\in I}\subseteq B(R)}$, ${\{x_a\}_{a\in I}\subseteq S}$. 
Since for any dense orthogonal 
$\{\xi_a\}_{a\in I}$, ${\{\xi_{a b}\}_{b\in I_a}\subseteq B(R)}$, 
${\{y_{a b}\}_{b\in I_a}\subseteq S}$, ${a\in I}$, 
${\{\xi_a \xi_{a b}\}_{a\in I,\ b\in I_a}}$ is dense and orthogonal in $B(R)$, 
${{\sum\limits_{a\in I}}^{\perp} \xi_a
\Bigl({\sum\limits_{b\in I_a}}^{\perp} \xi_{a b} y_{a b}\Bigr)=
{\sum\limits_{a\in I,\ b\in I_a}}^{\perp} \xi_a \xi_{a b} y_{a b}}$, 
${S\subseteq \hat{S}=O(\hat{S})=O(S)=O(O(S))}$.

If ${{\sum\limits_{a\in I}}^{\perp} \xi_a x_a=
{\sum\limits_{a'\in I'}}^{\perp} \xi'_{a'} x'_{a'}}$, 
${{\sum\limits_{b\in J}}^{\perp} \chi_b y_b=
{\sum\limits_{b'\in J'}}^{\perp} \chi'_{b'} y'_{b'}}$ for dense orthogonal 
${\{\xi_a\}_{a\in I}, \{\xi'_{a'}\}_{a'\in I'}}$, 
${\{\chi_b\}_{b\in J}, \{\chi_{b'}\}_{b'\in J'}\subseteq B(R), 
\{x_a\}_{a\in I}, \{x'_{a'}\}_{a'\in I'}, \{y_b\}_{b\in J}, 
\{y'_{b'}\}_{b'\in J'}\subseteq P(R)}$, then for all 
${a\in I}$, ${a'\in I', b\in J, b'\in J', \star=+, \cdot}$, 
${\xi_a \xi'_{a'} x_a=\xi_a \xi'_{a'} x_{a'}}$, 
${\chi_b \chi'_{b'} y_b=\chi_b \chi'_{b'} y'_{b'}}$, 
\begin{multline*}
{\sum_{a\in I,\ b\in J}}^{\perp} \xi_a \chi_b (x_a\star y_b) 
\ =\ 
{\sum_{a\in I,\ a'\in I',\ b\in J,\ b'\in J'}}^{\perp}
\xi_a \chi_b \xi'_{a'} \chi'_{b'} (x_a\star y_b)\ =
\\ 
{\sum_{a\in I,\ a'\in I',\ b\in J,\ b'\in J'}}^{\perp} 
\xi_a \chi_b \xi'_{a'} \chi'_{b'} (x'_{a'}\star y'_{b'})\ =\ 
{\sum_{a'\in I',\ b'\in J'}}^{\perp} \xi'_{a'} \chi'_{b'} 
(x'_{a'}\star y'_{b'})\,,
\end{multline*}
the extension of the operations ${\star=+, \cdot}$ from $P(R)$ to $O(P(R))$ 
in item 3 is defined correctly and, in view of 
${\alpha {\sum\limits_{a\in I}}^{\perp} \xi_a x_a=
{\sum\limits_{a\in I}}^{\perp} \xi_a \alpha x_a}$, ${\alpha\in \CM(R)}$, 
$\CM(R)$--linearly. Hence 
\[
{\sum_{a\in I}}^{\perp} \xi_a x_a\star {\sum_{b\in J}}^{\perp} \chi_a y_b\ =\ 
{\sum_{a\in I,\ b\in J}}^{\perp} \xi_a \chi_b (x_a\star x_b)\,,\quad 
\Bigl({\sum_{a\in I}}^{\perp} \xi_a x_a\Bigr) \psi\ =\ 
{\sum_{a\in I}}^{\perp} \xi_a (x_a \psi)
\]
for any dense orthogonal ${\{\xi_a\}_{a\in I}, \{\chi_b\}_{b\in J}
\subseteq B(R)}$, ${\{x_a\}_{a\in I}, \{y_b\}_{b\in J}\subseteq O(P(R))}$, 
${\star=+, \cdot}$ and ${\psi\in M(P(R))', M(O(P(R)))'}$, 
$M^{O(P(R))}(P(R))'$ is isomorphic to $M(P(R))'$ by means of 
${\psi\longmapsto \psi|_{P(R)}\in M(P(R))'}$, ${x \psi=x \psi|_{P(R)}}$, 
${\psi\in M^{O(P(R))}(P(R))'}$, ${x\in O(P(R))}$, 
$O(P(R))$ is a $M(P(R))'$--submodule of $I(P(R))$, the action $M(P(R))'$ on 
$I(P(R))$ is extended to the action $M(O(P(R)))'$:  
${x \psi={\sum\limits_{a\in I}}^{\perp} \xi_a (x (\xi_a \psi)|_{P(R)})}$ 
for all ${x\in I(P(R))}$, ${\psi\in M(O(P(R)))'}$ and a dense orthogonal 
${\{\xi_a\}_{a\in I}\subseteq B(R)}$, 
${\xi_a \psi\in M^{O(P(R))}(P(R))'}$, ${a\in I}$, 
$O(P(R))_{M(O(P(R)))'}$ (with the natural action of $M(O(P(R)))'$) is a 
submodule of $I(P(R))_{M(O(P(R)))'}$. For any ${k\geq 1}$, 
${\alpha_{a i}\in \CM(R)}$, ${\{x_{a i}\}_{a\in I}\subseteq I(P(R))}$, 
${i=1, \ldots, k}$, 
\[
{\sum_{a\in I}}^{\perp} \xi_a 
\Bigl(\sum_{i=1}^k \alpha_{a i} x_{a i}\Bigr)\ =\ 
{\sum_{a\in I}}^{\perp} \xi_a \Bigl(\sum_{i=1}^k \alpha_i x_{a i}\Bigr)\ =\ 
\sum_{i=1}^k \alpha_i {\sum_{a\in I}}^{\perp} \xi_a x_{a i}\,,
\]
where ${\alpha_i\in \CM(R)}$, ${\xi_a \alpha_i=\xi_a \alpha_{a i}}$ for all 
${a\in I}$ (completeness of $B(R)$). In particular, if any element of $P(R)$ 
is a $\CM(R)$--linear combination of $n$ elements of $R$ for some ${n\geq 1}$, 
then ${O(P(R))=\CM(R) O(R)}$. This is true, for example, in the case of finitely 
generated $\CM(R)$ over $F$ and/or $P(R)$ over $\CM(R)$. Thus, $O(P(R))$ and 
$O(R)$ are a $\CM(R)$--algebra and its $F$--subalgebra, their semiprimeness 
(primeness for the prime $R$) follows from the semiprimeness (primeness) of 
$P(R)$, the essentiality of the $M(P(R))'$--submodule 
${P(R)\subseteq I(P(R))}$ and\linebreak  
${O(\CM(R) T)\lhd O(P(R))}$ for all ${T\lhd O(R)}$. 
Since $P(R)$ is a fully invariant $M(P(R))'$--sub\-module of $I(P(R))$, 
for any ${\phi\in \End(I(P(R)))_{M(P(R))'}}$, dense orthogonal 
${\{\xi_a\}_{a\in I}\subseteq B(R)}$, ${\{x_a\}_{a\in I}\subseteq P(R)}$
\[
\phi {\sum\limits_{a\in I}}^{\perp} \xi_a x_a\ =\ 
{\sum\limits_{a\in I}}^{\perp} \xi_a \phi|_{P(R)}(\xi_a x_a)\ =\ 
\phi|_{P(R)} {\sum\limits_{a\in I}}^{\perp} \xi_a x_a\,,
\]
${\phi|_{P(R)}\in \CM(R)}$, ${\phi(O(P(R)))\subseteq O(P(R))}$, 
${\phi|_{O(P(R))}=\phi|_{P(R)} \Id_{O(P(R))}}$. As a consequence, 
\[
\End(O(P(R)))_{M(P(R))'}\ =\ \End(O(P(R)))_{M(O(P(R)))'}\ =\ 
\CM(R) \Id_{O(P(R))}\,.
\]

Like many facts about the non-singular $\CM(R)$--module $I(P(R))$ (Lemma 2.3), 
item 6 about $\Ann_{\CM(R)} S$, ${S=O(S)\subseteq I(P(R))}$, is included in the 
constructions of \cite{B6}; it suffices to cite item 2 of Lemma 2. Applying 
Zorn's lemma, we choose the maximal element $\{x_a\}_{a\in I}$ in the  
partially ordered by inclusion set $\mathcal{T}$ of all  
${T=\{x_a\}_{a\in I_T}\subseteq S, x_a\ne 0, \beta_a \beta_{a'}=0}$ for all 
${a'\ne a\in I_T}$, where ${\Ann_{\CM(R)} x_a=(\Id_{P(R)}-\beta_a) \CM(R)}$, 
${\beta_a\in B(R)}$, 
and ${x\in S}$, ${\beta_a x=\beta_a x_a}$ 
for all ${a\in I}$ ($\{\beta_a\}_{a\in I}$ can be complemented to a dense 
orthogonal subset of $B(R)$). If there is ${y\in S}$, 
${(\Id_{P(R)}-\beta_x) y\ne 0}$,  
${\Ann_{\CM(R)} x=(\Id_{P(R)}-\beta_x) \CM(R)}$, ${\beta_x\in B(R)}$, 
then one can add 
${b\notin I}$ to $I$, ${I'=I\cup \{b\}}$, set ${x_b=(\Id_{P(R)}-\beta_x) y}$ 
and obtain 
\[
\beta_x x_a\ =\ \beta_x \beta_a x_a\ =\ \beta_x \beta_a x\ =\ \beta_a x\ =\ 
x_a\,,\quad \beta_x \beta_a\ =\ \beta_a\quad (a\in I)
\]
(${\Id_{P(R)}-\beta_x\in (\Id_{P(R)}-\beta_a) \CM(R), 
\Id_{P(R)}-\beta_x\leq \Id_{P(R)}-\beta_a, \beta_a\leq \beta_x}$), 
${x_b=(\Id_{P(R)}-\beta_x) x_b}$, ${\beta_x\leq \Id_{P(R)}-\beta_b}$, 
${\Ann_{\CM(R)} x_b=(\Id_{P(R)}-\beta_b) \CM(R)}$, 
${0=\beta_x \beta_b=\beta_a \beta_x \beta_b=\beta_a \beta_b}$ for all 
${a\in I}$ and ${\{x_a\}_{a\in I}\subsetneq \{x_c\}_{c\in I'}\in \mathcal{T}}$, a contradiction. Therefore 
${\Ann_{\CM(R)} S=\Ann_{\CM(R)} x}$. 
\end{proof}

From the definition of the operations of $O(P(R))$ follows that $O(P(R))$ 
satisfies all homogeneous identities of $R$.
The orthogonal completeness over $B(R)$ of all ${S=O(S)\subseteq I(P(R))}$ in the 
sense of \cite[Definition~2.1]{BMich} follows from 
\cite[Remark~8.8]{BMich}. 
Any ${D\in \Der_F(P(R))}$ extends to ${D\in \Der_F(O(P(R)))}$ by the rule: 
${\Bigl({\sum\limits_{a\in I}}^{\perp} \xi_a x_a\Bigr) D=
{\sum\limits_{a\in I}}^{\perp} \xi_a (x_a D)}$
for any dense orthogonal ${\{\xi_a\}_{a\in I}\subseteq B(R), 
\{x_a\}_{a\in I}\subseteq P(R)}$. As 
${[\Der_F(P(R)), B(R)]=\{0\}}$ \cite[Remark~2.4]{Gol10}, the fulfillment 
of ${{\sum\limits_{a\in I}}^{\perp} \xi_a x_a={\sum\limits_{a'\in I'}}^{\perp} 
\xi'_{a'} x'_{a'}}$ for dense orthogonal 
${\{\xi_a\}_{a\in I}, \{\xi'_{a'}\}_{a'\in I'}\subseteq B(R), 
\{x_a\}_{a\in I},}$ ${\{x'_{a'}\}_{a'\in I'}\subseteq P(R)}$ is equivalent to 
${\xi_a \xi'_{a'} x_a=\xi_a \xi'_{a'} x'_{a'}}$ for all ${a\in I}$, 
${a'\in I'}$ and 
\[
(\xi_a \xi'_{a'} x_a) D\ =\ \xi_a \xi'_{a'} (x_a D)\ =\ 
(\xi_a \xi'_{a'} x'_{a'}) D\ =\ \xi_a \xi'_{a'} (x'_{a'} D)\quad 
(a\in I,\ a'\in I')\,,
\] 
${{\sum\limits_{a\in I}}^{\perp} \xi_a (x_a D)=
{\sum\limits_{a\in I'}}^{\perp} \xi'_{a'} (x_{a'} D)}$, such an extension 
is correct. If $D$ is $\CM(R)$--linear, then its extension to $O(P(R))$ 
is also $\CM(R)$--linear. 

Any ${B\subseteq B(R)}$ has the exact lower and upper bounds ${\nu(B)=\inf B}$ 
and ${\mu(B)=\sup B}$ in $B(R)$ \cite[Lemma~1.16]{BMich}. 
If $B$ is an ultrafilter in $B(R)$, then ${\mu(B)=\Id_{P(R)}\in B}$ and\linebreak 
either ${0\ne \nu(B)\in B}$ (if ${\Id_{P(R)}-\nu(B)\in B}$,  
${0=\nu(B)=\nu(B) (\Id_{P(R)}-\nu(B))}$) or ${\nu(B)=0}$, 
${\Ann_{\CM(R)} B=\{0\}}$, ${\mu(P)=\Id_{P(R)}}$, 
${\Ann_{\CM(R)} P=\{0\}}$ for ${P=B(R)\setminus B\in \Spec(B(R))}$, the maximal 
orthogonal subsets of $P$ are dense in $B(R)$ 
\cite[Lemmas~1.16,~1.17]{BMich}.

Due to the regularity of $\CM(R)$, for all ${I\lhd \CM(R)}$, ${J\lhd B(R)}$ 
\begin{gather*}
I\ =\ (I\cap B(R))_{\CM(R)}\ =\ (I\cap B(R)) \CM(R)\,,\quad 
I\cap B(R)\lhd B(R)\,,
\\
(J)_{\CM(R)}\ =\ J \CM(R)\ =\ 
\{\alpha \beta\mid \alpha\in J,\ \beta\in \CM(R)\}\,,\quad 
J\ =\ J \CM(R)\cap B(R)\,,
\end{gather*}
the proper prime ideals of $B(R)$ and $\CM(R)$ are maximal, 
$\Spec(B(R))$ and $\Spec(\CM(R))$ consist of  
${P\cap B(R), P\in \Spec(\CM(R))}$, and ${Q \CM(R), Q\in \Spec(B(R))}$,  
respectively (the observation before Remark 2.2, 
${\alpha, \beta\leq \alpha+\beta-\alpha \beta=
\alpha\oplus \beta (\alpha\oplus \Id_{P(R)})}$, ${\alpha, \beta\in B(R)}$, 
\cite[Propositions~2,~3,~p.~60,~61]{Lamb}).

To each ultrafilter $B$ in $B(R)$ there corresponds a congruence of the 
$M(P(R))'$--module $I(P(R))$: ${x\sim_B y, x, y\in I(P(R))}$, if 
there is ${\xi\in B, \xi x=\xi y}$ \cite[\S~4]{BMich}. 
If ${x\in I(P(R)), \alpha\in B}$, then ${\alpha x=0}$ is 
equivalent to ${x=(\Id_{P(R)}-\alpha) x}$, 
${\Id_{P(R)}-\alpha\in P=B(R)\setminus B\in \Spec(B(R))}$. As a 
consequence, $\sim_B$ is 
a congruence modulo the $M(P(R))'$--submodule 
${P I(P(R))\subseteq I(P(R))}$. The restrictions of $\sim_B$ to the 
algebras ${S=R, P(R), O(R), O(P(R))}$ are the congruences modulo their ideals ${S\cap P S}$, 
${P S=\{\alpha x\mid \alpha\in P,\ x\in S\}}$, $
{S\cap P S=S\cap P I(P(R))}$, 
${P S=S\cap P S}$ for ${S=O(S)}$ (${B(R) S\subseteq S}$ for  
${S=O(S)\subseteq I(P(R))}$, \cite[Remark~8.8]{BMich}), ${S_B=S/(S\cap P S)}$, 
${S=R, P(R), O(R), O(P(R)), I(P(R))}$. 

If ${\nu(B)\in B}$, then ${\Id_{P(R)}-\nu(B)=\mu(P)\in P=\mu(P) B(R)}$ and  
${x+P I(P(R))\longmapsto \nu(B) x}$, ${x\in I(P(R))}$, is an isomorphism of 
$I(P(R))_B$ and $\nu(B) I(P(R))$, ${S_B\cong \nu(B) S}$, ${S=R, P(R)}$, 
$O(R)$, $O(P(R))$, $I(P(R))$. Following 
\cite[Theorem~1~(1)]{B6}, \cite[Theorem~8.9~(4)]{BMich}, we prove 
              
\begin{lemma}
If $B$ is an ultrafilter in $B(R)$, ${S=O(R), O(P(R))}$, ${\nu(B)\in B}$ 
and/or there is such ${n\geq 1}$ that for any ${Q\in \Spec(S)}$ the 
elements of $M(P(S/Q))$ are finite sums of the products of at most $n$ 
operators ${t_x, t=l, r, x\in P(S/Q)}$, then the algebra $S_B$ is prime.
\end{lemma}

\begin{proof}
If ${S'=R, P(R), O(R), O(P(R))}$, ${K H\subseteq S'\cap P S'
\subsetneqq K, H\lhd S'}$, ${P=B(R)\setminus B}$, then 
${\alpha K H\ne \{0\}}$ for all ${\alpha\in B}$. Otherwise, there is 
${\alpha\in B, \alpha K H=(\alpha \CM(R) K) (\CM(R) H)=\{0\}}$,  
${T\cap P(R)\lhd P(R)}$ for ${T=\alpha \CM(R) K\cap \CM(R) H}$, 
${\{0\}=(T\cap P(R))^2=T\cap P(R)=T}$, we can complement 
$\alpha \CM(R) K$ in the $M(P(R))'$--module $O(P(R))$ to an essential 
submodule ${\alpha \CM(R) K\oplus M, \CM(R) H\subseteq M}$, and find 
${\beta\in B(R)}$, ${\beta x=x, \beta y=0}$ for all 
${x\in \alpha \CM(R) K}$, ${y\in M}$ 
(Proposition 2.4 (3, 4); quasi-injectivity of 
$O(P(R))_{M(P(R))'}$). Thus, either ${\beta\in B}$,\linebreak ${\beta H=\{0\}, 
H\subseteq P S,}$ or ${\Id_{P(R)}-\beta\in B,}$ 
${(\Id_{P(R)}-\beta) \alpha K=\{0\}, K\subseteq P S}$, a contradiction. If ${\nu(B)\in B}$, then ${\nu(B) K H=\{0\}}$ for any such ${K, H}$, a contradiction. Therefore, in this case $S'_B$ is prime. 

If ${n\geq 1}$ from the second condition, ${x, y\in S}$, 
${(x)_S (y)_S\subseteq P S}$, 
\begin{gather*}
M_{T, T'}\ =\ \{z z'\mid z\in M_T(x),\ z'\in M_{T'}(y)\}\quad 
(T, T'\in \mathcal{S}=\{\emptyset\}\cup \{l, r\}\cup\ldots \cup \{l, r\}^n)\,,
\\
M_{\emptyset}(z)\ =\ \{z\}\,,\ 
M_T(z)\ =\ \{z t^{(1)}_{x_1}\cdots t^{(k)}_{x_k}\mid x_i\in S\}\quad 
(z\in S,\ T=(t^{(1)}, \ldots, t^{(k)})\in \{l, r\}^k)\,,
\end{gather*}
then ${M_{T, T'}=O(M_{T, T'})\subseteq P S}$, there exists 
${\alpha_{T, T'}\in B, \alpha_{T, T'} M_{T, T'}=\{0\}}$, for all ${T, T'\in \mathcal{S}}$  
(the proof of Proposition 2.4 and (6); ${M_T(z)=O(M_T(z))}$) and for 
${\alpha=\prod\limits_{T, T'\in \mathcal{S}} \alpha_{T, T'}\in B}$
\begin{multline*}
(\alpha (x)_S (y)_S+Q)/Q\ =\ ((\alpha x)_S (y)_S+Q)/Q\subseteq 
(\alpha x+Q)_{P(S/Q)} (y+Q)_{P(S/Q)}\subseteq
\\ 
\sum_{T, T'\in \mathcal{S}} \CM(S/Q) (\alpha M_{T, T'}+Q)/Q\ =\ \{0\}
\quad (Q\in \Spec(S))\,,
\end{multline*}
${\alpha (x)_S (y)_S\subseteq \bigcap\limits_{Q\in \Spec(S)} Q=
\prr(S)=\{0\}}$ (the choice of $n$; Proposition 2.4 (3, 4, 5)), 
$(x)_S\subseteq P S$ and/or ${(y)_S\subseteq P S}$ (see the 
beginning of the proof). 
\end{proof}

Let $M_A$, $I(M_A)$ and ${P(M_A)=\End(I(M_A))_A M}$ be from the beginning 
of the section, 
${0\ne e=e^2\in Z(A)}$. If ${\phi: A\longrightarrow A'}$ is a homomorphism 
of associative rings $A$ and $A'$, then $A'$--modules can be considered as 
$A$--modules on which ${x\in A}$ acts as $\phi x$. In particular, 
$A e$--modules are $A$--modules for ${\phi=l_e=r_e}$. If 
$I_{A e}$ is an injective $A e$--module, $J$ is a right ideal\linebreak of $A$, 
${\psi\in \Hom(J, I)_A}$, then ${J=J e\oplus J(1-e), 
\psi(J (1-e))=\{0\}, \psi|_{J e}\in \Hom(J e, I)_{A e}}$, according to 
the Baer's criterion \cite[Lemma~1,~p.~147]{Lamb}, 
\cite[Theorem~3.41,~p.~205]{Fe} 
there exists\linebreak ${y=y e\in I}$, ${\psi x=\psi (x e+x (1-e))=\psi (x e)=y x e=y x}$ 
for all ${x\in J}$, $I_A$ is an injective $A$--module. If $I'_A$ is an 
injective $A$--module, $J'=J' e$ is a right ideal of $A e$, 
${\tau\in \Hom(J', I' e)_{A e}}$, then $J'$ is a right ideal of $A$, 
${\tau\in \Hom(J', I' e)_A}$, there is ${z\in I'}$, 
${\tau x=\tau (x e)=z (x e)=(z e) x}$ for all ${x\in J'}$ (Baer's criterion), 
$I' e$ is an injective $A e$--module. Since ${I(M e_A)=I(M e_A) e}$ 
(${I(M e_A) (1-e)\cap M e=\{0\}}$), $A e$--submodules of $I(M e_A)$ are 
$A$--submodules. From here, taking into account the essentiality of the 
$A e$--submodule ($A$--submodule) $M e$ in $I(M_A) e$ 
(${K\cap M e=K\cap M\ne \{0\}}$ for all ${\{0\}\ne K_A\subseteq I(M_A) e}$) 
and the quasi-injectivity of injective modules, we obtain, up to isomorphism, 
that ${I(M e_A)=I(M e_{A e})=I(M_A) e=I(I(M_A) e_{A e})}$, 
\begin{multline*}
P(M e_{A e})\ =\ P(M e_A)\ =\ P(M_A) e\ =\ \End(I(M_A))_A M e\ =
\\ 
\End(I(M_A) e)_A M e\ =\ \End(I(M_A) e)_{A e} M e\,.
\end{multline*}

\begin{lemma}
If $B$ is an ultrafilter in $B(R)$, ${\nu(B)\in B}$ and/or there exist 
${k, n\geq 1}$ such that the elements of $M(O(P(R)))$ are the sums of $k$ 
products of at most $n$ operators $t_x$, 
${t=l, r}$, ${x\in O(P(R))}$, then $O(P(R))_B$ is an essential 
$M(O(P(R)))'_B$--submodule of $I(P(R))_B$ and 
${Z(M(O(P(R))_B)')\cong Z(M(O(P(R)))')_B\cong \CM(R)/P \CM(R)}$, 
${P=B(R)\setminus B}$.
\end{lemma}

\begin{proof}
The canonical epimorphism ${\psi_B: O(P(R))\longrightarrow O(P(R))_B}$ induces 
an epimorphism ${\psi_{M, B}: M(O(P(R)))'\longrightarrow M(O(P(R))_B)', 
\psi_{M, B} \Id_{O(P(R))}=\Id_{O(P(R))_B}, 
\psi_{M, B} t_x=t_{\psi_B x}}$ for all ${t=l, r}$, ${x\in O(P(R))}$, 
\begin{multline*}
\Ker \psi_{M, B}\ =\ \Ann_{M(O(P(R)))'} O(P(R))_B\ =
\\ 
\{\phi\in M(O(P(R)))'\mid (O(P(R)))\phi=O((O(P(R)))\phi)
\subseteq P O(P(R))\}\ =\ P M(O(P(R)))'
\end{multline*}
(there is ${\alpha\in B}$, ${\alpha (O(P(R)))\phi=\{0\}}$, ${\alpha \phi=0}$, 
${\phi\in P M(O(P(R)))'}$ (Proposition 2.4 and its proof)),
${M(O(P(R))_B)'\cong M(O(P(R)))'_B=M(O(P(R)))'/P M(O(P(R)))'}$. 
The structure of the $M(O(P(R)))'$--module is introduced on $I(P(R))$ in the 
proof of Proposition 2.4. When the second condition is satisfied for any dense 
orthogonal ${\{\xi_a\}_{a\in I}\subseteq B(R)}$, 
${\{x_a\}_{a\in I}\subseteq I(P(R))}$, 
${\{\phi_a\}_{a\in I}\subseteq M(O(P(R)))'}$, writing 
\[
\phi_a\ =\ \tau_a \Id_{O(P(R))}+
\sum_{j=1}^n \sum_{i=1}^k \sum_{T=(t^{(1)}, \ldots, t^{(j)})\in 
\{l, r\}^j} t^{(1)}_{x(a, i, T, 1)}\cdots t^{(j)}_{x(a, i, T, j)}
\]
for ${\tau_a\in \CM(R)}$, ${x(a, i, T, j)\in O(P(R)), a\in I}$, where the sum 
of ${1+k (2^{n+1}-2)}$ terms has at most ${k+1}$ non-zero terms, we get 
\[
{\sum_{a\in I}}^{\perp} \xi_a (x_a \phi_a)\ =\ 
\tau x+
\sum_{j=1}^n \sum_{i=1}^k \sum_{T=(t^{(1)}, \ldots, t^{(j)})\in 
\{l, r\}^j} x t^{(1)}_{x(i, T, 1)}\cdots t^{(j)}_{x(i, T, j)}\,,
\]
${\tau\in \CM(R)}$, ${\xi_a \tau=\xi_a \tau_a, a\in I}$ (Lemma 2.3), 
${x={\sum\limits_{a\in I}}^{\perp} \xi_a x_a}$, 
${x(i, T, l)={\sum\limits_{a\in I}}^{\perp} \xi_a x(a, i, T, l)}.$ 
In particular, it follows that ${x M(O(P(R)))'=O(x M(O(P(R)))')}$, 
\[
(S)[\phi, M(O(P(R)))]\ =\ 
\bigcup\limits_{s\in S} s [\phi, M(O(P(R)))]\ =\ O((S)[\phi, M(O(P(R)))])
\]
for all ${x\in I(P(R))}$, ${S=O(S)\subseteq I(P(R))}$, ${\phi\in M(O(P(R)))'}$.

If ${x\in I(P(R))}$, ${x M(O(P(R)))'\cap O(P(R))\subseteq P I(P(R))}$, then for 
${\beta=\nu(B)\in B}$ for the first condition and some ${\beta\in B}$ 
for the second ${\beta x M(O(P(R)))'\cap O(P(R))=\{0\}}$, since for the latter 
${x M(O(P(R)))'\cap O(P(R))=O(x M(O(P(R)))'\cap O(P(R)))}$, 
\[
\beta x M(O(P(R)))'\cap O(P(R))\subseteq \beta (x M(O(P(R)))'\cap O(P(R)))\ =\ 
\{0\}\,,
\]
(Proposition 2.4 (2, 6)). In any case, ${\beta x=0}$, ${x\in P I(P(R))}$. 
Hence 
\[
(x M(O(P(R)))')_B\cap O(P(R))_B\ =\ (x+P I(P(R))) M(O(P(R)))'_B\cap O(P(R))_B
\ \ne\ \{0\}
\] 
for all ${x\notin P I(P(R))}$. If ${\phi\in M(O(P(R)))'}$ and 
${\psi_{M, B} \phi\in Z(M(O(P(R))_B)')}$, then  
\[
(O(P(R)))[\phi, M(O(P(R)))]\subseteq P O(P(R))\,,\quad 
\gamma (O(P(R)))[\phi, M(O(P(R)))]\ =\ 
\{0\}
\]
with ${\gamma=\nu(B)}$ for the first condition and some ${\gamma\in B}$ for 
the second, ${\gamma \phi\in Z(M(O(P(R)))')}$, 
\begin{multline*}
\phi+P M(O(P(R)))'\in (Z(M(O(P(R)))')+P M(O(P(R)))')/P M(O(P(R)))'\cong
\\
Z(M(O(P(R))))'_B\cong \psi_{M, B}(Z(M(O(P(R)))'))
\end{multline*}
(Proposition 2.4 (2, 6)). Since 
${\psi_{M, B}(Z(M(O(P(R)))'))\subseteq Z(M(O(P(R))_B)')}$, 
\begin{multline*}
Z(M(O(P(R))_B)')\ =\ \psi_{M, B}(Z(M(O(P(R)))'))\ =
\\ 
\psi_{M, B}(\CM(R) \Id_{O(P(R))})\cong \CM(R)/P \CM(R)
\end{multline*}
(Proposition 2.4 (4)).

For ${\nu(B)\in B}$, the essentiality of the $M(R)'_B$--submodule 
${R_B\subseteq P(R)_B}$ and $M(P(R))'_B$--submodule 
${P(R)_B\subseteq I(P(R))_B}$ together with   
\[
Z(M(R_B)') \cong Z(M(R)')_B\,,\quad 
Z(M(P(R)_B)')\cong Z(M(P(R))')_B\cong \CM(R)/P \CM(R)
\]
(${M(R)'_B=M(R)'/(M^{P(R)}(R)'\cap P M^{P(R)}(R)')|_R}$) is established in a 
similar way. In view of ${S_B\cong \nu(B) S}$, ${S=R, P(R), O(R), 
O(P(R)), I(P(R))}$, Lemma 2.5 and the observations before Lemma 2.6, up to 
isomorphism ${P(R)_B=P(R_B), I(P(R))_B=I(P(R)_B)}$, 
\begin{multline*}
\End(\nu(B) P(R))_{M(R)'}\ =\ \CM(R)|_{\nu(B) P(R)}\ =
\\ 
\End(\nu(B) P(R))_{\nu(B) M(R)'}\cong 
\nu(B) \CM(R)\cong \CM(R_B)
\end{multline*}
and, as a consequence, ${O(R)_B=O(R_B), O(P(R))_B=O(P(R)_B)=O(P(R_B))}$,
\begin{multline*}
\End(\nu(B) O(P(R)))_{M(P(R))'}\ =\ 
\End(O(P(R)))_{M(P(R))'}|_{\nu(B) O(P(R))}\ =
\\ 
\CM(R) \Id_{\nu(B) O(P(R))}\ =\ 
\End(\nu(B) O(P(R)))_{\nu(B) M(P(R))'}\ =
\\  
\End(\nu(B) O(P(R)))_{M(O(P(R)))'}\ =\ 
\End(\nu(B) O(P(R)))_{\nu(B) M(O(P(R)))'}\cong
\\ 
\End(O(P(R_B)))_{M(P(R_B))'}\ =\ \End(O(P(R_B)))_{M(O(P(R_B)))'}\cong 
\nu(B) \CM(R)
\end{multline*}
(Proposition 2.4 (4); quasi-injectivity of $P(R)_{M(R)'}$ and 
$O(P(R))_{M(P(R))'}$, complete in-\linebreak variance of ${\nu(B) P(R)_{M(R)'}\subseteq 
P(R)}$ and ${\nu(B) O(P(R))_{M(P(R))'}\subseteq O(P(R))}$). 
\end{proof}

If $F\langle X\rangle$ is a free non-associative $F$--algebra with a set of 
free generators ${X=\{x_i\}_{i=1}^{\infty}}$, $x_1$ is included in all 
non-zero monomials of ${0\ne f(x_1, \ldots, x_n)\in F\langle X\rangle}$, 
then the absence of ${0\ne x\in R=P(R)}$, ${f(x, y_2, \ldots, y_n)=0}$ 
for all ${y_i\in R}$, is inherited by $O(R)$ and $O(R)_B$ for any ultrafilter 
$B$ in $B(R)$. It is enough to note that for a dense orthogonal 
${\{\xi_a\}_{a\in I}\subseteq B(R)}$, ${\{x_a\}_{a\in I}\subseteq R}$, 
${x={\sum\limits_{a\in I}}^{\perp} \xi_a x_a\in O(R)}$, from  
${f(x, z_2, \ldots, z_n)=0}$ for all ${z_i\in O(R)}$ it follows that 
${\xi_a f(x, y_2, \ldots, y_n)=f(\xi_a x_a, y_2, \ldots, y_n)=0}$ for all 
${y_i\in R}$, ${\xi_a x_a=0}$ for all ${a\in I}$, ${x=0}$. Hence if 
${J=\{f(x, y_2, \ldots, y_n)\mid y_i\in O(R)\}\subseteq P O(R), 
P=B(R)\setminus B}$, then ${J=O(J)}$, ${\alpha J=\{0\}}$ for some 
${\alpha\in B}$ (Proposition 2.4 (6)), ${\alpha x=0}$, 
${x\in P O(R)}$. For a homogeneous variety of $F$--algebras $\mathfrak{M}$ 
on which the non-degeneracy condition is defined 
(alternative, right-alternative, linear Jordan (over $F$ with $1/2$), 
Lie and Mal'tsev algebras), the non-degeneracy of ${P(R)\in \mathfrak{M}}$ 
is inherited by ${R, O(S), O(S)_B\in \mathfrak{M}, S=R, P(R)}$, for all ultrafilters $B$ in 
$B(R)$, where non-degeneracy in the right-alternative, linear Jordan, Lie, and $2$--torsion-free Mal'tsev case is the absence of ${0\ne x\in R}$ under 
the above condition for ${f(x_1, x_2)=(x_1 x_2) x_1}$, 
$f(x_1, x_2)=2 x_1 (x_1 x_2)-x_1^2 x_2$, 
${f(x_1, x_2, x_3, x_4)=(x_2 x_1) x_1+((x_3 x_1) x_4) x_1}$ and 
${f(x_1, x_2)=(x_2 x_1) x_1}$, respectively. If non-degenerate algebras from $\mathfrak{M}$ 
are subdirect products of non-degenerate prime (strongly prime) algebras, then 
from Lemma 2.5 we can deduce the strong primeness of $O(R)_B$ for any 
ultrafilter $B$ in $B(R)$ and non-degenerate ${P(R)\in \mathfrak{M}}$ with 
${\nu(B)\in B}$ and/or in the presence of ${n\geq 1}$ such that the elements  
$M(P(O(R)/Q))$ are finite sums of products of at most $n$ operators 
${t_x, t=l, r, x\in P(O(R)/Q)}$ for all ${Q\lhd O(R)}$, 
$O(R)/Q$ is strongly prime.

If the non-degenerate algebras from $\mathfrak{M}$ form a semisimple class and 
the upper (strongly degenerate) radical determined by them is hereditary on 
subalgebras, the non-degeneracy of ${R\in \mathfrak{M}}$ 
inherits ${P(R)\in \mathfrak{M}}$. In particular, this is true for the 
Slin'ko --- McCrimmon (McCrim-\linebreak mon) radical $Mc$ of alternative 
(right-alternative, linear Jordan over $F$ with $1/2$) algebras and the 
Kostrikin radical $K$ of Lie and Mal'tsev algebras over fields of 
characteristic zero \cite[Remark~3.8]{Gol10}, \cite{ZelM}, \cite[Lemma~2.8]{Gol5}. Below it will be shown that $P(R)$ also inherits the 
non-degeneracy of the Lie algebra $R$ without $6$--torsion. 

Examples of algebras $R$ such that for some ${n\geq 1}$ the elements of 
$M(R)$ are finite sums of the products of at most $n$ operators 
${t_x, t=l, r, x\in R}$, are any algebra $R$ over a field 
$\mathbb{F}$, ${k=\dim_{\mathbb{F}} R< \infty}$ with ${n\leq \max\{1, k^2\}}$,  
the associative algebra $A$ with ${n=2}$, 
\[
M(A)\ =\ \biggl\{l_x+r_y+\sum_{i=1}^k l_{x_i} r_{y_i}\biggl| 
x, y, x_i, y_i\in A,\ k\geq 1\biggr\}\,,
\]
the Jordan algebra ${\Jr(V, f)}$ of the symmetric bilinear form $f$ on the 
vector space $V$ over the field $\mathbb{F}$ with ${n=4}$, where 
${\Jr(V, f)=\mathbb{F}\cdot 1\oplus V}$ is the direct sum of the spaces 
${\mathbb{F}=\mathbb{F}\cdot 1}$ and $V$ with 1 and multiplication 
\[
(\alpha\cdot 1+x)\cdot (\beta\cdot 1+y)\ =\ (\alpha \beta+f(x, y))\cdot 1+
(\beta x+\alpha y)\quad (\alpha, \beta\in \mathbb{F},\ x, y\in V)\,,
\]
and, in particular, ${(\alpha\cdot 1+x) y=f(x, y)\cdot 1+\alpha y}$. As a 
consequence, for all ${y_i\in V}$
\begin{multline*}
(\alpha\cdot 1+x) l_{y_1} l_{y_2}\ =\ \alpha f(y_1, y_2)\cdot 1+f(x, y_1) y_2\,,
\ 
(\alpha\cdot 1+x) [l_{y_1}, l_{y_2}]\ =\ f(x, y_1) y_2-f(x, y_2) y_1\,,
\\
\shoveleft{
(\alpha\cdot 1+x) (l_{y_1} l_{y_2} l_{y_3}-f(y_1, y_2) l_{y_3})\ =\ 
f(x, y_1 f(y_2, y_3)-y_3 f(y_1, y_2))\cdot 1\ =}
\\
(\alpha\cdot 1+x) [l_{y_1}, l_{y_3}] l_{y_2}\,,\quad  
l_{y_1} l_{y_2} l_{y_3}\ =\ [l_{y_1}, l_{y_3}] l_{y_2}+f(y_1, y_2) l_{y_3} 
\end{multline*}
and, due to ${[l_{y_1}, l_{y_2}] l_{y_3} [l_{y_4}, l_{y_5}]=0,
[l_x, l_y]\in \Der(J), x, y\in J}$ for any linear Jordan algebra $J$,
\begin{multline*}
l_{y_1} l_{y_2} l_{y_3} l_{y_4} l_{y_5}\ =\ 
[l_{y_1}, l_{y_3}] [l_{y_2}, l_{y_5}] l_{y_4}+
f(y_1, y_2) l_{y_3} l_{y_4} l_{y_5}+
f(y_2, y_4) [l_{y_1}, l_{y_3}] l_{y_5}\ =
\\
[l_{y_1}, l_{y_3}] l_{y_4 [l_{y_5}, l_{y_2}]}+
f(y_1, y_2) l_{y_3} l_{y_4} l_{y_5}+
f(y_2, y_4) [l_{y_1}, l_{y_3}] l_{y_5}\,,
\end{multline*}
${M(\Jr(V, f))=\Bigl\{\sum\limits_{i=1}^k 
l_{x_{1 i}} l_{x_{2 i}} l_{x_{3 i}} l_{x_{4 i}}\Bigl| x_{j i}\in 
\Jr(V, f),\ k\geq 1\Bigr\}}$. From this and lemma 2.5 we obtain 

\begin{co}
If $B$ is an ultrafilter in $B(R)$, ${S=O(R), O(P(R))}$, ${\nu(B)\in B}$ or 
(and) ${\sup\limits_{S\ne Q\in \Spec(S)} \dim_{\CM(S/Q)} P(S/Q)< \infty}$, 
then the algebra $S_B$ is prime.
\end{co}

This is true, for example, for the generalized special Lie algebra $R$ and the 
subalgebra $R$ of the algebra $A^{(+)}$ of the associative $PI$--algebra $A$ 
over $F$ with $1/2$, where $A^{(+)}$ is a linear Jordan algebra obtained from 
$A$ by replacing the multiplication operation with 
${x\cdot y=1/2(x y+y x), x, y\in A}$ 
\cite[Theorem~4.1,~p.~47]{Raz} and 
its corollaries in \cite{BP}, \cite[Remark~3.5]{Gol6}. 

If $R$ is a strongly prime non-associative alternative algebra or a strongly 
prime linear Jordan $PI$--algebra over $F$ with $1/2,$ then 
${Z(R)\ne \{0\},}$ the ring of fractions ${Q=S^{-1} R\cong P(R)}$ re\-lative to 
${S=Z(R)\setminus \{0\}}$ is a central simple algebra over the field 
${Z(Q)=S^{-1} Z(R)\cong \CM(R),}$ $Q$ in the alternative 
case is a Cayley --- Dickson algebra over $Z(Q)$, ${\dim_{Z(Q)} Q=8}$, and in 
the Jordan case is isomorphic to one of the algebras: 
\begin{enumerate}

\item $A^{(+)}$, $A$ is a simple associative $PI$--algebra over 
${Z(Q)=Z(A^{(+)})=Z(A)}$,\\ ${\dim_{Z(Q)} A\leq [\frac{n}2]^2}$;

\item ${\Sym(A, *)=\{x+x^*\mid x\in A\}}$ is a subalgebra of 
$*$--symmetric elements of ${A^{(+)}, A}$ is a simple associative 
$PI$--algebra over ${Z(Q)=Z(\Sym(A, *))=Z(A)\cap \Sym(A, *)}$ 
with involution $*$, ${\dim_{Z(Q)} A\leq 2 \dim_{Z(A)} A\leq 2 n^2}$;

\item an Albert algebra over $Z(Q)$, ${\dim_{Z(Q)} Q=27}$; 

\item ${\Jr(V, f)}$ is the Jordan algebra of the non-degenerate bilinear 
symmetric form $f$ on the $Z(Q)$--space $V$, ${\dim_{Z(Q)} V> 1}$,

\end{enumerate}
$n$ is the least degree of an essential Jordan identity of $R$ 
\cite[Theorem~9,~p.~229]{ZShS}, \cite[Theorems~4,~5]{ZelP}, \cite[Theorem~5]{Am1}, \cite[Theorems~3.16,~3.19]{BM}, 
\cite[Theorems~3.5,~3.7]{BM1}, 
\cite[the~observations~before~Remark~2.1]{Gol4}. Hence for alternative and 
linear Jordan $PI$--algebras, Lemma 2.8 follows immediately from the remarks 
made above. 

\begin{lemma}
If a non-degenerate algebra $R$ is alternative (linear Jordan over $F$ with 
$1/2$), $B$ is an ultrafilter in $B(R)$, then the algebras ${O(R)_B, 
O(P(R))_B}$ are strongly prime.
\end{lemma}

\begin{proof}
Let ${P=B(R)\setminus B, S=O(R), O(P(R))}$ and ${I_{x, y}\subseteq P S}$ 
for some ${x, y\in S, I_{x, y}=(x S) y}$ and ${I_{x, y}=\{x, S, y\}}$ 
for the alternative and linear Jordan algebra $R$, respectively, 
\[
\{x, z ,y\}\ =\ z 1/2 (U_{x+y}-U_x-U_y)\ =\ (x z) y+(y z) x-(x y) z\,. 
\]
Since ${I_{x, y}=O(I_{x, y})}$, there is ${\alpha\in B}$, ${\alpha I_{x, y}=
I_{\alpha x, y}=\{0\}}$ (Proposition 2.4 (6)). Then for any ${Q\lhd S}$ 
such that the algebra $S/Q$ is strongly prime, ${I_{\alpha x, y}\subseteq Q}$, 
${\alpha x\in Q}$ and/or ${y\in Q}$, ${(\alpha x)_S (y)_S\subseteq Q}$ 
\cite[Theorems~1,~2]{BMS} and, due to the non-degeneracy of the algebra $S$ 
and the speciality of the radical $Mc$ \cite{ZelM}, 
\cite[Remark~3.8]{Gol10}, 
${(\alpha x)_S (y)_S=\{0\}}$. 
Hence there is ${\beta\in B(R)}$, ${\beta u=u}$, ${\beta v=0}$ for all 
${u\in (\alpha x)_S, v\in (y)_S}$ (see the beginning of the proof of 
Lemma 2.5) and either ${\beta\in B, \beta y=0, y\in P S}$, or 
${\Id_{P(R)}-\beta\in B, (\Id_{P(R)}-\beta) \alpha x=0, x\in P S}$. 
Therefore, the algebra $S_B$ is strongly prime 
\cite[Theorems~1,~2]{BMS}.
\end{proof}
    
\begin{lemma}
If $R$ and $O(R)$ are non-degenerate Lie algebras, $B$ is an ultrafilter in 
$B(R)$, then the Lie algebra $O(R)_B$ is strongly prime.
\end{lemma}

\begin{proof}
From \cite[the~proof~of~Lemma~2.5]{GK} we can conclude that if there are 
$x, y\in O(R)\setminus P O(R)$, 
${I_{x, y}=[[[O(R), y], O(R)], x]\subseteq P O(R)}$, ${P=B(R)\setminus B}$, 
then there exist $u\in (x)_{O(R)}\setminus P O(R)$, 
${v\in (y)_{O(R)}\setminus P O(R), (u)_{O(R)}\cap (v)_{O(R)}=\{0\}}$. 
As ${I_{x, y}=O(I_{x, y})}$, ${\alpha I_{x, y}=I_{\alpha x, y}=\{0\}}$ 
for some ${\alpha\in B}$ (Proposition 2.4 (6)) and without loss of 
generality we can assume that ${I_{x, y}=\{0\}}$. 
If ${z\in O(R), J_z=[[O(R), z], z]=O(J_z)\subseteq P O(R)}$, then 
${\beta_z J_z=J_{\beta_z z}=\{0\}}$ for some ${\beta_z\in B}$, 
${\ad_{\beta_z z}^2=0, 2 \beta_z z=0}$ 
(Proposition 2.4 (6); non-degeneracy of the Lie algebra 
$O(R)$, ${\ad_{r \ad_{\beta_z z}^2}=-2 \beta_z \ad_z \ad_r \ad_z=0}$, 
${r\in O(R)}$). If ${z\in O(R)\setminus P O(R)}$, ${\beta\in B}$, 
${\ad_{\beta z}^2=0}$,\linebreak then 
${I_z=[[O(R), z], [O(R), z]]=O(I_z)\not\subseteq P O(R)}$ (otherwise there 
exists ${\gamma\in B}$, ${\gamma I_z=\{0\}}$ (Proposition 2.4 (6)), 
${\gamma \beta I_z=I_{\gamma \beta z, \gamma \beta z}=\{0\}}$, 
${\gamma \beta z=0}$, ${z\in P O(R)}$, a contradiction). So, as in\linebreak \cite[the~proof~of~Lemma~2.5]{GK}, we can assume that for any 
${x'\in (x)_{O(R)}\setminus P O(R)}$, 
$y'\in (y)_{O(R)}\setminus P O(R)$, ${I_{x', y'}=\{0\}}$, there exist 
${\beta_{x'}, \beta_{y'}\in B}$, 
${\ad_{\beta_{x'} x'}^2=\ad_{\beta_{y'} y'}^2=0}$, 
${2 \beta_{x'} x'=2 \beta_{y'} y'=0}$. In particular, this is true for all 
${x'\in I_x\setminus P O(R)}$, ${y'\in I_y\setminus P O(R)}$. It remains to 
take $x'\in I_{\beta_x x}\setminus P O(R)$, 
${x''\in I_{\beta_{x'} x'}\setminus P O(R)}$, 
${y'\in I_{\beta_y y}\setminus P O(R)}$, 
${y''\in I_{\beta_{y'} y'}\setminus P O(R)}$, 
${y'''\in I_{\beta_{y''} y''}\setminus P O(R)}$ and similarly to \cite[the~proof~of~Lemma~2.5]{GK}, 
obtain ${(x'')_{O(R)}\cap (y''')_{O(R)}=\{0\}}$. 
For $u, v\in O(R)\setminus P O(R)$, ${(u)_{O(R)}\cap (v)_{O(R)}=\{0\}}$, we 
can choose ${\delta\in B(R)}$, ${\delta u=u, \delta v=0}$ (the end of 
the proof of Lemma 2.8), and either ${\delta\in B, v\in P O(R)}$ or 
${\Id_{P(R)}-\delta\in B, u\in P O(R)}$, a contradiction. 

Thus, ${I_{x, y}\not\subseteq P O(R)}$ for all ${x, y\in O(R)\setminus P O(R)}$, 
the Lie algebra $O(R)_B$ is strongly prime \cite[Theorem~1.3]{GK}. For 
$R$ without $6$--torsion, the non-degeneracy condition for $O(R)$ can be 
omitted (the non-degeneracy of $R$ is inherited by 
$P(R)$, $O(P(R))$, $O(R)$ (Proposition 3.2, 
\cite[Propositions~2.4,~2.7]{OMol})) and the proof can be shortened using \cite[Theorem~7.2.2]{LopJ}.
\end{proof}

\section{Orthogonal completeness and systems of $\mathfrak{M}$--quotients} 

In \cite{B6, BMich} the orthogonal completion $O(A)$ of the semiprime 
associative algebra $A$ is constructed within the complete 
ring of quotients $Q(A)$ and, in view of ${Z(Q(A))={}_A C\cong \CM(A)}$, 
is canonically isomorphic to the construction described here through the 
identification of ${{\sum\limits_{a\in I}}^{\perp} \xi_a x_a}$ in $Q(A)$ 
and ${I(P(A))\cong I({}_A Q)}$. From \cite[the~proof~of~Lemma~1~(1)]{B6} it follows that for any orthogonal ${\{\xi_a\}_{a\in I}\subseteq B(A), 
\{x_a\}_{a\in I}\subseteq {}_A Q}$ we can complement   
${J=\sum\limits_{a\in I} \xi_a {}_A Q\lhd {}_A Q}$ to  
${K=J\oplus J'\in \mathcal{E}({}_A Q)}$, ${J'\lhd {}_A Q}$, consider 
${\phi\in \Hom(K, {}_A Q)_{{}_A Q}}$, ${\psi\in \Hom_{{}_A Q}(K, {}_A Q)}$, 
\[
\phi\,:\ \sum_{i=1}^n \xi_{a_i} y_{a_i}+y\longmapsto 
\sum_{i=1}^n \xi_{a_i} x_{a_i} y_{a_i}\,,\ 
\psi\,:\ \sum_{i=1}^n \xi_{a_i} y_{a_i}+y\longmapsto 
\sum_{i=1}^n \xi_{a_i} y_{a_i} x_{a_i}\quad
(y_a\in {}_A Q,\ y\in J')\,,
\]
${x' (\phi x)=(x' \psi) x}$ for any ${x, x'\in K}$, and select  
${q\in Q^s_m({}_A Q), \phi=l_q|_K, \psi=r_q|_K}$ (see introduction), 
${\xi_a \phi=\psi \xi_a=\xi_a x_a=\xi_a q}$ for all ${a\in I}$. Thus, 
\[
O({}_A Q)\subseteq Q^s_m({}_A Q)\subseteq 
Q({}_A Q)\ =\ Q(A)\ =\ Q(Q(A))\ =\ O(Q(A))
\]
\cite[(1.14),~(1.15)]{Ut}, \cite[Lemma~1]{B6}, \cite[Proposition~8.7]{BMich}. This also follows from \cite[Lemmas~3,~4]{Har2}.
For any dense orthogonal ${\{\chi_a\}_{a\in I}\subseteq B(A)}$, 
${\{z_a\}_{a\in I}\subseteq Q^r_m(A)}$ there exist 
${J_a\in \mathcal{E}(A)}$, ${\chi_a J_a, \chi_a z_a J_a\subseteq A}$, 
${a\in I}$. If ${J=\sum\limits_{a\in I} \chi_a J_a}$ and 
${x\in A}$, ${x J=\{0\}}$, then ${\chi_a x J_a=\{0\}}$, ${\chi_a x=0}$ 
for all ${a\in I}$, ${x=0}$,
${\Ann_A J=\{0\}}$, ${J\in \mathcal{E}(A)}$. For ${\nu\in \Hom(J, A)_A}$, 
\[
\nu\,:\ \sum_{i=1}^n \chi_{a_i} x_{a_i}\longmapsto 
\sum_{i=1}^n \chi_{a_i} z_{a_i} x_{a_i}\quad (x_a\in J_a)\,,
\]
we can find ${z\in Q^r_m(A), \nu=l_z|_J, 
(\chi_a z-\chi_a z_a) J_a=\{0\}, \chi_a z=\chi_a z_a}$ for all ${a\in I}$.
If $\{z_a\}_{a\in I}\subseteq Q^s_m(A)$, ${J_a\in \mathcal{E}(R)}$, 
${\chi_a J_a, \chi_a z_a J_a, J_a \chi_a z_a\subseteq A}$, ${a\in I}$, then 
${J=\sum\limits_{a\in I} \chi_a J_a\in \mathcal{E}(A)}$, ${z J\subseteq A}$ 
(see above), ${J z=\sum\limits_{a\in I} J_a \chi_a z=\sum\limits_{a\in I} 
J_a \chi_a z_a\subseteq A}$, ${z\in Q^s_m(A)}$. It follows that 
${Q^*_m(A)=O(Q^*_m(A))}$, ${*=r, s}$.

A similar construction of the orthogonal completion in the symmetric Martindale 
ring of quotients is also possible for Lie and Jordan analogues of the 
latter from \cite{Mol, JMQ}.
The definitions of the graded central closure $P_{gr}(R)$ and the Martindale 
centroid $\CM_{gr}(R)$ of a homogeneous semiprime $\Gamma$--graded 
$F$--algebra ${(R, \{R_{\gamma}\}_{\gamma\in \Gamma})}$ can be found in 
\cite[before~Proposition~4.37]{Gol6}, where $\Gamma$ is an Abelian group in 
additive notation, $\{R_{\gamma}\}_{\gamma\in \Gamma}$ are $F$--submodules of 
${R=\bigoplus\limits_{\gamma\in \Gamma} R_{\gamma}, 
R_{\alpha} R_{\beta}\subseteq R_{\alpha+\beta}, 
\alpha, \beta\in \Gamma}$, ${I\lhd R}$ is homogeneous (${I\lhd_{gr} R}$) if 
${I=\bigoplus\limits_{\gamma\in \Gamma} (I\cap R_{\gamma})}$, 
$R$ is homogeneously prime (semiprime) if ${I J\ne \{0\}}$ 
(${I^2\ne \{0\}}$) for all ${\{0\}\ne I, J\lhd_{gr} R}$, 
$I_{gr}(R)$ and ${P_{gr}(R)=\End(I_{gr}(R))_{M(R)'-gr} R}$ are the graded 
injective and quasi-injective hulls of the $\Gamma$--graded module $R$ over the 
$\Gamma$--graded algebra ${(M(R)', \{M_{\gamma}\}_{\gamma\in \Gamma})}$ with 
${M_0=F \Id_R+M(R)_0}$, ${M_{\gamma}=M(R)_{\gamma}}$ for 
${0\ne \gamma\in \Gamma}$,  
\[
M(R)_{\alpha}\ =\ \Biggl\{\sum_{k=1}^m t_{k 1}\cdots t_{kn_k}\Biggl| 
t_{ki}\in \{l_{x_{ki}}, r_{x_{ki}}\},\ x_{ki}\in R_{\gamma_{ki}},\ 
\sum_{i=1}^{n_k} \gamma_{ki}=\alpha,\ m, n_k\geq 1\Biggr\}\quad 
(\alpha\in \Gamma)\,,
\]
$\End(I_{gr}(R))_{M(R)'-gr}$ and  
\[
\CM_{gr}(R)\ =\ \End(P_{gr}(R))_{M(R)'-gr}\ =\ 
\{\psi|_{P_{gr}(R)}\mid \psi\in \End(I_{gr}(R))_{M(R)'-gr}\}
\]
are the algebras of homogeneous $M(R)'$--endomorphisms of $I_{gr}(R)$ and 
$P_{gr}(R)$. The conclu\-sion for the homogeneously semiprime $R$ of the 
commutativity and regularity of 
$\CM_{gr}(R)$, the definition of the $\Gamma$--graded $\CM_{gr}(R)$--algebra 
${(P_{gr}(R)=\CM_{gr}(R) R, \{\CM_{gr}(R) R_{\gamma}\}_{\gamma\in \Gamma})}$,  
the proofs of the homogeneous semiprimeness of $P_{gr}(R)$ and, for the 
homogeneous prime $R$, the fact that $\CM_{gr}(R)$ is a field and $P_{gr}(R)$ 
is homogeneously prime, are similar to the ungraded\linebreak case 
\cite[Lemma~3.1,~Proposition~3.1,~3.2,~p.~42--44]{Raz}, 
\cite[Proposition~1,~p.~168]{Lamb}. 
By analogy with the previous one, for the homogeneously semiprime 
${(R, \{R_{\gamma}\}_{\gamma\in \Gamma})}$ we define a Boolean ring 
${B_{gr}(R)=\{\alpha\in \CM_{gr}(R)\mid \alpha=\alpha^2\}}$, an 
orthogonally complete ${S\subseteq I_{gr}(P_{gr}(R))}$, and an 
orthogonal completion $O_{gr}(S)$ of ${S\subseteq I_{gr}(P_{gr}(R))}$.
The proofs of 
\[
E^{\perp}\ =\ \{x\in I_{gr}(P_{gr}(R))\mid E x=\{0\}\}\ =\ \{0\}\quad 
(E\subseteq \CM_{gr}(R),\ \Ann_{\CM_{gr}(R)} E=\{0\})\,,
\]
the self-injectivity of $\CM_{gr}(R)$, the existence for any 
${S\subseteq I_{gr}(P_{gr}(R))}$ of a unique 
${\alpha\in B_{gr}(R)}$, ${\Ann_{\CM_{gr}(R)} S=\alpha \CM_{gr}(R)}$, 
and the completeness of 
$B_{gr}(R)$ are analogous to Remark 2.2 ($E^{\perp}$ is a homogeneous 
$M(P_{gr}(R))'$--submodule of $I_{gr}(P_{gr}(R))$, 
${M=E^{\perp}\cap P_{gr}(R)}$ should be complemented to the essential 
among the homogeneous $M(R)'$--submodules of $P_{gr}(R)$ and select 
${\alpha\in B_{gr}(R)}$; homogeneity of submodules and subalgebras is 
defined as for ideals) and Lemma 2.3 ($\tau_{\phi}$, ${I\lhd \CM_{gr}(R)}$, 
${\phi\in \Hom_{\CM_{gr}(R)}(I, \CM_{gr}(R))}$, is homogeneous).

If ${(R, \{R_{\gamma}\}_{\gamma\in \Gamma})}$ is a homogeneously semiprime 
algebra with finite $\Gamma$--grading, that is, ${|\supp(R)|< \infty}$, 
${\supp(R)=\{\gamma\in \Gamma\mid R_{\gamma}\ne \{0\}\}}$, then the 
$\Gamma$--gradings $P_{gr}(R)$, $M(P_{gr}(R))'$ and $I_{gr}(P_{gr}(R))$ 
are finite, since ${\supp(R)=\supp(P_{gr}(R)), P_{gr}(R)}$ is an 
essential homogeneous $M(P_{gr}(R))'$--submodule of $I_{gr}(P_{gr}(R))$, 
\begin{gather*}
\supp(M(P_{gr}(R))')\subseteq 
\{\gamma-\gamma'\mid \gamma, \gamma'\in 
\supp(R)\}\,,
\\ 
\supp(I_{gr}(P_{gr}(R))) \subseteq 
\{\gamma-\gamma'\mid \gamma\in \supp(R),\ \gamma'\in \supp(M(P_{gr}(R))')\}\,,
\end{gather*}
and any dense orthogonal $\Xi=\{\xi_a\}_{a\in I}\subseteq B_{gr}(R)$ 
corresponds to a homogeneous $M(P_{gr}(R))'$--embedding 
${\psi_{\Xi}: I_{gr}(P_{gr}(R))\hookrightarrow 
\bigoplus\limits_{\gamma\in \Gamma} 
\prod\limits_{a\in I} \xi_a I_{gr}(P_{gr}(R))_{\gamma}}$, 
\[
\psi_{\Xi}\,:\ \sum_{i=1}^k x_{\gamma_i}\longmapsto 
\sum_{i=1}^k \prod_{a\in I} \xi_a x_{\gamma_i}\quad 
\biggl(x=\sum_{i=1}^k x_{\gamma_i}\in I_{gr}(P_{gr}(R))\biggr)\,,
\]
$M(P_{gr}(R))'$--epimorphism 
${\phi_{\Xi}: \bigoplus\limits_{\gamma\in \Gamma} 
\prod\limits_{a\in I} \xi_a I_{gr}(P_{gr}(R))_{\gamma}\longrightarrow 
I_{gr}(P_{gr}(R))}$, ${\phi_{\Xi} \psi_{\Xi}=\Id_{I_{gr}(P_{gr}(R))}}$. 
If ${\{x_a\}_{a\in I}\subseteq I_{gr}(P_{gr}(R))}$, 
${z_{\gamma}\in \prod\limits_{a\in I} \xi_a I_{gr}(P_{gr}(R))_{\gamma}}$, 
${z_{\gamma}: a\longmapsto \xi_a x_{a, \gamma}}$, ${a\in I}$, 
${\gamma\in \Gamma}$, then  
\[
\phi_{\Xi} (\xi_a z_{\gamma})\ =\ \xi_a (\phi_{\Xi} z_{\gamma})\ =\ 
\phi_{\Xi} \psi_{\Xi}(\xi_a x_{a, \gamma})\ =\ \xi_a x_{a, \gamma}\in 
I_{gr}(P_{gr}(R))_{\gamma}\quad (a\in I)\,,
\]
${\phi_{\Xi} z_{\gamma}={\sum\limits_{a\in I}}^{\perp} \xi_a x_{a,\gamma}\in 
I_{gr}(P_{gr}(R))_{\gamma}}$, ${\gamma\in \Gamma}$, and 
${{\sum\limits_{a\in I}}^{\perp} \xi_a x_a=
\sum\limits_{\gamma\in \Gamma} 
{\sum\limits_{a\in I}}^{\perp} \xi_a x_{a, \gamma}}$ (${\Xi^{\perp}=\{0\}}$). 
As a consequence, with the necessary changes in the proof, Proposition 2.4 is 
transferred to $R$: 1) ${I_{gr}(P_{gr}(R))=O_{gr}(I_{gr}(P_{gr}(R)))}$; 2) 
if ${S\subseteq I_{gr}(P_{gr}(R))}$, then ${O(S)=O(O(S))}$ is the set of 
all ${{\sum\limits_{a\in I}}^{\perp} \xi_a x_a\in I_{gr}(P_{gr}(R))}$ for a 
dense orthogonal 
${\{\xi_a\}_{a\in I}\subseteq B_{gr}(R), \{x_a\}_{a\in I}\subseteq S}$; 3) 
\[
(O_{gr}(P_{gr}(R)),\ \{O_{gr}(P_{gr}(R)_{\gamma})=O_{gr}(P_{gr}(R))\cap 
I_{gr}(P_{gr}(R))_{\gamma}\}_{\gamma\in \Gamma})
\]
is a homogeneously semiprime $\CM_{gr}(R)$--algebra with the operations 
\[
{\sum\limits_{a\in I}}^{\perp} \xi_a x_a\star 
{\sum\limits_{b\in J}}^{\perp} \chi_b y_b\ =\ 
{\sum\limits_{a\in I,\ b\in J}}^{\perp} \xi_a \chi_b (x_a\star y_b)\quad 
(\star=+, \cdot) 
\]
for dense orthogonal 
${\{\xi_a\}_{a\in I}, \{\chi_b\}_{b\in J}\subseteq B_{gr}(R), 
\{x_a\}_{a\in I}, \{y_b\}_{b\in J}\subseteq P_{gr}(R)}$;
4) $O_{gr}(P_{gr}(R))$ is a fully invariant $M(P_{gr}(R))'$--submodule of 
$I_{gr}(P_{gr}(R))$, 
\begin{multline*}
\End(O_{gr}(P_{gr}(R)))_{M(P_{gr}(R))'-gr}\ =
\\ 
\End(O_{gr}(P_{gr}(R)))_{M(O_{gr}(P_{gr}(R)))'-gr}\ =\ 
\CM_{gr}(R) \Id_{O_{gr}(P_{gr}(R))}\,;
\end{multline*} 
5) ${(O_{gr}(R), \{O_{gr}(R_{\gamma})\}_{\gamma\in \Gamma})}$ is a homogeneously 
semiprime $F$--subalgebra $O_{gr}(P_{gr}(R))$ and if any element of 
$P_{gr}(R)$ is a $\CM_{gr}(R)$--linear combination of $n$ elements of 
$R$ for some ${n\geq 1}$, then ${O_{gr}(P_{gr}(R))=\CM_{gr}(R) O_{gr}(R)}$; 
6) for any ${S=O_{gr}(S)\subseteq I_{gr}(P_{gr}(R))}$ there is ${x\in S}$, 
${\Ann_{\CM_{gr}(R)} S=\Ann_{\CM_{gr}(R)} x}$. 

Following \cite{OMol}, we call a $\Gamma$--graded Lie algebra 
${(Q, \{Q_{\gamma\in \Gamma}\})}$ a \emph{graded algebra of quotients of a 
$\Gamma$--graded Lie algebra} ${(R, \{R_{\gamma}\}_{\gamma\in \Gamma})}$ 
if $R$ is a homogeneous subalgebra of $Q$ and for any 
${0\ne q\in Q_{\gamma}, \gamma\in \Gamma}$ one can choose 
${I\lhd_{gr} R, \Ann_R I=\{0\}, \{0\}\ne [I, q]\subseteq R}$ 
(it is equivalent: for any ${u\in Q_{\alpha}, v\in Q_{\beta}, \alpha, \beta\in 
\Gamma, u\ne 0}$, there are ${x\in R_{\gamma}, \gamma\in \Gamma, 
[x, u]\ne 0, [x, (v)_R]\subseteq R}$), where 
${\Ann_R I=\{x\in R\mid [x, I]=\{0\}\}}$, ${(v)_R=v \Ad^Q(R)'}$.

The homogeneous semiprimeness (primeness) of $R$ and, for $R$ without 
$6$--torsion, its homo\-geneous non-degeneracy are inherited by $Q$ 
(\cite[Proposition~2.4]{OMol}; $R$ is essential among the homogeneous 
$\Ad(R)'$--submodules of $Q$; $Q$ inherits the non-degeneracy of $R$ 
without $6$--torsion \cite[Proposition~2.7]{Mol}, 
\cite[Proposition~2.7]{OMol}; 
homogeneous non-degeneracy of $R$ is the absence of such 
${0\ne x\in R_{\gamma}}$, ${\gamma\in \Gamma}$, that 
${\ad_x (F \Id_R+\ad(R)) \ad_x=\{0\}}$). 

For the homogeneously semiprime Lie algebra 
${(R, \{R_{\gamma}\}_{\gamma\in \Gamma})}$, the \emph{maximal graded algebra 
of quotients ${(Q_{m-gr}(R), \{Q_{m-gr}(R)_{\gamma}\}_{\gamma\in \Gamma})}$}, 
into which the graded algebras of $R$ are embedded by means of homogeneous 
monomorphisms identical on $R$, is determined uniquely up to a homogeneous 
isomorphism identical on $R$ by the axioms: 
\begin{enumerate}

\item $R$ is a homogeneous subalgebra of $Q_{m-gr}(R)$; 

\item for any ${x\in Q_{m-gr}(R)_{\gamma}, \gamma\in \Gamma}$ there exists 
${I\in \mathcal{E}_{gr}(R)}$, ${[I, x]\subseteq R}$; 

\item if ${x\in Q_{m-gr}(R)_{\gamma}, \gamma\in \Gamma, 
I\in \mathcal{E}_{gr}(R), [I, x]=\{0\}}$, then ${x=0}$; 

\item for any ${I\in \mathcal{E}_{gr}(R), D\in \PDer_{gr}(I, R)_{\gamma}, 
\gamma\in \Gamma}$ there is ${y\in Q_{m-gr}(R)_{\gamma}, \ad_y|_I=D}$, 

\end{enumerate}
where ${\mathcal{E}_{gr}(R)=\{I\lhd_{gr} R\mid \{0\}\ne I\cap J\ \forall 
\{0\}\ne J\lhd_{gr} R\}}$ 
(${\mathcal{E}_{gr}(R)=\{I\lhd_{gr} R\mid I\in \mathcal{E}(R)\}}$; 
if ${I\lhd_{gr} R}$, then ${\Ann_R I\lhd_{gr} R}$, 
${I\in \mathcal{E}(R)}$ is equivalent to ${I\cap J\ne \{0\}}$ for all 
${\{0\}\ne J\lhd_{gr} R}$ \cite[Lemma~3.1]{OMol}), 
\begin{gather*}
\PDer(I, R)\ =\ 
\{D\in \Hom_F(I, R)\mid (x y) D=(x D) y+x (y D)\ \forall x, y\in I\}\,,
\\
\PDer_{gr}(I, R)_{\gamma}\ =\ \{D\in \PDer(I, R)\mid (I_{\alpha})D\subseteq 
R_{\alpha+\gamma}\ \forall \alpha\in \Gamma\}\quad (\gamma\in \Gamma)\,,
\end{gather*}
${\ad(R)|_I\subseteq \PDer_{gr}(I, R)=
\bigoplus\limits_{\gamma\in \Gamma} \PDer_{gr}(I, R)_{\gamma}}$. The algebra 
$Q_{m-gr}(R)$ is realized as a quotient set 
$\mathcal{D}_{gr}(R)/\sim$ under the equivalence relation from the definition 
of ${}_R C$ (if ${D\in \PDer(I, R)}$, ${J\lhd R, J\subseteq I, 
\Ann_R J=(J)D=\{0\}}$, then ${(J I)D=J((I)D)=\{0\}, D=0}$) with the operations  
\begin{gather*}
[(D, I)]+[(T, J)]\ =\ [(D+T, I\cap J)]\,,\quad f [(D, I)]\ =\ [(f D, I)]\,,
\\
[[(D, I)], [(T, J)]]\ =\ [([D, T], (I\cap J)^2)]
\end{gather*}
${(D, I), (T, J)\in \mathcal{D}_{gr}(R)=\{(D, I)\mid D\in \PDer_{gr}(I, R),\ 
I\in \mathcal{E}_{gr}(R)\}}$, ${f\in F}$, $\Gamma$--grading
\[
Q_{m-gr}(R)_{\gamma}\ =\ \{[(D, I)]\mid D\in \PDer_{gr}(I, R)_{\gamma},\ 
I\in \mathcal{E}_{gr}(R)\}\quad (\gamma\in \Gamma)\,, 
\]
${R \hookrightarrow Q_{m-gr}(R), x\longmapsto [(\ad_x, R)], x\in R}$ 
\cite[Theorem~2.11~and~the~observations~before~it]{OMol}, 
\cite[Theorems~3.6,~3.8]{Mol}. These conclusions can be transferred to any, not necessarily 
homogeneously semiprime, ${(R, \{R_{\gamma}\}_{\gamma\in \Gamma})}$ by 
replacing $\mathcal{E}_{gr}(R)$ with a power filter $\mathcal{F}$ of 
homogeneous ideals of $R$ with zero annihilators and $Q_{m-gr}(R)$ with a 
maximal algebra of quotients $Q_{\mathcal{F}}(R)$ with respect to 
$\mathcal{F}$ with axioms and realization similar to $Q_{m-gr}(R)$, in which 
any graded\linebreak algebra of quotients of $R$ with respect to $\mathcal{F}$ 
(a $\Gamma$--graded Lie algebra ${(Q, \{Q_{\gamma}\}_{\gamma\in \Gamma})}$ such 
that $R$ is a homogeneous subalgebra of $Q$, for any ${0\ne q\in Q}$, 
${\gamma\in \Gamma}$ there is ${I\in \mathcal{F}, 
\{0\}\ne [I, x]\subseteq R}$) is embedded homogeneously and identically on 
$R$ \cite[Sec.~2.2]{JMQ}.

For an ungraded Lie algebra $R$, the definition of the algebra of quotients 
and the description of the maximal algebra of quotients $Q_m(R)$ of a 
semiprime $R$ from \cite{Mol} are obtained from those given above by a 
formal transition to ${\Gamma=\{0\}}$. The graded algebras of quotients of 
a homogeneously semiprime $R$ are its algebras of quotients as a Lie algebra 
without grading \cite[Proposition~2.7]{OMol}.

If ${(R, \{R_{\gamma}\}_{\gamma\in \Gamma})}$, ${I\lhd_{gr} R}$, $\pi_{\gamma}$ 
is the projection of $R$ onto $R_{\gamma}$, ${\gamma\in \Gamma}$, 
${D\in \PDer(I, R)}$, then, as is easy to verify,  
\[
D_{\alpha}\,:\ x\longmapsto 
\sum_{\gamma\in \Gamma} x \pi_{\gamma} D \pi_{\gamma+\alpha}\,,\quad
D_{\alpha}\in \PDer_{gr}(I, R)_{\alpha}\quad 
(\alpha\in \Gamma,\ x\in I)\,,
\] 
${D=\sum\limits_{\alpha\in \Gamma} D_{\alpha}}$ for ${D\in \PDer_{gr}(I, R)}$, 
${\PDer(I, R)=\PDer_{gr}(I, R)}$ for $R$ with finite $\Gamma$--grading 
\cite[Lemma~2.1]{QTS}. 

For example, let the Lie algebra ${(R, \{R_{\gamma}\}_{\gamma\in \Gamma})}$, 
${R\ne \{0\}}$, be Artinian (with the minimality condition) for homogeneous 
ideals and homogeneously non-degenerate, $\mathcal{M}$ be the set of minimal 
${I\lhd_{gr} R}$. Then ${\Ann_R J\lhd_{gr} R, 
\{0\}=(J\cap \Ann_R J)^2=\Ann_R J\cap J, J\cong \ad(J)}$ for all 
${J\lhd_{gr} R}$ (as a consequence, ${|\mathcal{M}|=1}$ for the homogeneously 
prime $R$), 
\begin{gather*}
I\ =\ I^2\lhd_{\Der} R\,,\quad 
\{0\}\ =\ \Bigl(I\cap \sum_{I'\ne I\in \mathcal{M}} I'\Bigr)^2\ =\ 
I\cap \sum_{I'\ne I\in \mathcal{M}} I'\quad (I\in \mathcal{M})\,,
\\ 
M\ =\ \sum_{I\in \mathcal{M}} I\ =\ \bigoplus_{I\in \mathcal{M}} I\ =\ 
\bigcap_{J\in \mathcal{E}_{gr}(R)} J\in \mathcal{E}_{gr}(R)\,,\quad 
M\ =\ M^2\lhd_{\Der} R\,,
\\
Q_{m-gr}(R)\cong \PDer_{gr}(M, R)\ =\ \Der_{gr}(M)\cong 
\prod_{I\in \mathcal{M}} \Der_{gr}(I)\,,
\end{gather*}
${J\lhd_{\Der} R}$ if ${(J) D\subseteq J}$ for all ${D\in \Der(R)}$, 
the isomorphism ${\phi: \Der_{gr}(M)\longrightarrow 
\prod\limits_{I\in \mathcal{M}} \Der_{gr}(I)}$ acts by the rule: 
${\phi(D): I\longmapsto D|_I\in \Der_{gr}(I)}$, ${I\in \mathcal{M}}$, 
${D\in \Der_{gr}(M)}$, and $R\cong \ad(R)|_M\subseteq 
\Der_{gr}(R)|_M\subseteq \Der_{gr}(M)$. 
If $R$ without $2$--torsion and ${I\in \mathcal{M}}$ is Artinian for 
homogeneous ideals, then for any ${J\lhd_{gr} I}$ there exists 
${k\geq 1, J^k=\bigcap\limits_{i\geq 1} J^i\lhd_{gr} R}$, and either 
${J=J^k=I}$ or ${J^k=\{0\}, J\subseteq K_h(I)\subseteq K_h(R)=\{0\}}$, 
where $K_h(R)$ is the smallest among ${K\lhd_{gr} R}$, $R/K$ is homogeneously 
non-degenerate \cite[Lemma~4]{Zel} (its proof is transferred to the 
graded case), the Lie algebra $I$ is homogeneously simple. 
If ${|\mathcal{M}|< \infty}$, ${\Der_{gr}(I)=\ad(I)}$ for all ${I\in \mathcal{M}}$, then 
\[
R\cong \ad(R) \cong \bigoplus_{I\in \mathcal{M}} \ad(R)|_I\ =\ 
\bigoplus_{I\in \mathcal{M}} \ad(I)\cong \Der_{gr}(M)\ =\ \ad(M)\cong M\,.
\]
If $F$ is a field, ${\dim_F R=n< \infty}$ and ${\Ch F=0}$ or ${\Ch F=p> k(n)}$, 
${n\geq 3}$, then the non-degeneracy of $R$ is equivalent to semiprimeness 
(semisimplicity) \cite[Proposition~2.15]{Gol2}. For such a semiprime 
$R$ and ${\Ch F=0}$, this implies the classical result on the decomposition 
of ${R=Q_m(R)=Q_{m-gr}(R)=M}$ into direct sums of its simple and 
homogeneously simple (as Lie algebras) ideals \cite[\S~5,~p.~84]{JacL}.
               
Each ${D\in \Der_{gr}(R)=\PDer_{gr}(R, R)}$ can be extended to 
${D\in \Der_{gr}(Q_{m-gr}(R))}$ according to the rule: if 
${q\in Q_{m-gr}(R)_{\gamma}, \gamma\in \Gamma, I\in \mathcal{E}_{gr}(R), 
[I, q]\subseteq R}$, then ${q D_{\alpha}\in Q_{m-gr}(R)_{\gamma+\alpha}}$, 
where ${\ad_{q D_{\alpha}}|_{I^2}\in \PDer_{gr}(I^2, R)_{\gamma+\alpha}, 
[x, q D_{\alpha}]=[x, q] D_{\alpha}-[x D_{\alpha}, q], 
q\in Q_{m-gr}(R)_{\gamma}, x\in I, \alpha\in \Gamma}$.

One of the arguments used in Propositions 3.1, 3.3, 3.5, 3.7, 3.9,
Theorem 3.10
(see below) can be summarized as follows. The algebra (system)
of quotients $Q$ of the algebra (system) $R$ with its semiprime or
nondegeneracy condition embeds into the injective hull ${I=I(R_A)}$ of its
essential $A$--submodule $R$, ${A=M^Q(R)'}$, identically on $R$
(the nature of the embedding is unimportant here and it appears in the text as inclusion).
By default, the transition from $R$ to ${R'=P(R)}$ and, correspondingly, from $Q$, $A$,
$I$ to $Q'$, $A'$, $I'$ implies defining the algebra (system) of fractions $Q'$
for the algebra (system) $R'$ over $\CM(R)$ with the faithful action of $\CM(R)$ and
considering $Q'$, $A'$, $I'$ over $\CM(R)$,
${\CM(R)\cong \CM(R) \Id_{I'}\subseteq \End(I')_{A'}}$. Deriving the orthogonal
completeness of $Q'$ over $B(R)$ reduces to establishing the coincidence of $Q'$ with its
orthogonal completion in $I'$ over $B(R)$, i.e. within the framework of general definitions of
\cite{BMich} modulo the arguments of Remark 2.2, Lemma 2.3,
Proposition 2.4, and their graded versions with $I(R')$ replaced by $I'$,
${I(R')\hookrightarrow I'}$ (the restriction
of the action of $A'$ to $R'$ is an epimorphism of $A'$ to $M(R')'$;
the orthogonal completion over $B(R)$ is uniquely defined up to
isomorphism (Proposition 2.4 (2))).

\begin{prop} 
For a homogeneously semiprime Lie algebra 
${(R, \{R_{\gamma}\}_{\gamma\in \Gamma})}$ with finite $\Gamma$--grading, 
$P_{gr}(R)$ is a $\Gamma$--graded algebra of quotients of
$R$, ${P_{gr}(R)\hookrightarrow Q_{m-gr}(R)}$ homogeneously  
identically on $R$, ${O_{gr}(P_{gr}(R))\subseteq Q_{m-gr}(P_{gr}(R))=
O_{gr}(Q_{m-gr}(P_{gr}(R)))}$, $Q_{m-gr}(P_{gr}(R))_B$ is a 
$\Gamma$--graded algebra of quotients of $O_{gr}(P_{gr}(R))_B$ for any 
ultrafilter $B$ in $B_{gr}(R)$. 
\end{prop}

\begin{proof}
If ${0\ne x=\sum\limits_{i=1}^k \alpha_i x_i\in P_{gr}(R)}$, ${k\geq 1}$, 
${x_i\in R_{\gamma}}$, ${\gamma\in \Gamma}$, ${\alpha_i\in \CM_{gr}(R)}$, 
then 
\begin{gather*}
I\ =\ \bigcap_{i=1}^k {}_{\alpha_i} R\in \mathcal{E}_{gr}(R)\,,\quad 
\{0\}\ \ne\ x \Ad(R)'\cap I\lhd_{gr} R\,,
\\ 
\{0\}\ \ne\ (x \Ad(R)'\cap I)^2\subseteq [x \Ad(R)', I]\ =\ [x, I] \Ad(R)'\subseteq R
\end{gather*}
(Proposition 2.1). Hence we can assume that 
\[
R\subseteq \End(Q_{m-gr}(R))_{\Ad(R)'-gr} R\subseteq 
P_{gr}(R)\subseteq Q_{m-gr}(R)\subseteq 
I_{gr}\bigl(R_{\Ad^{Q_{m-gr}(R)}(R)'}\bigr)\,.
\]
Since for any ${x\in Q_{m-gr}(P_{gr}(R))_{\gamma}}$, ${\gamma\in \Gamma}$ 
there are ${I\in \mathcal{E}_{gr}(P_{gr}(R))}$, ${[x, I]\subseteq P_{gr}(R)}$, 
and ${\alpha\in B_{gr}(R), 
\Ann_{\CM_{gr}(R)} x=\Ann_{\CM_{gr}(R)} [x, I]=\alpha \CM_{gr}(R),
Q_{m-gr}(P_{gr}(R))}$ is a non-singular $\CM_{gr}(R)$--module. 

For any dense orthogonal ${\{\xi_a\}_{a\in I}\subseteq B_{gr}(R)}$ and 
${\{x_a\}_{a\in I}\subseteq Q_{m-gr}(P_{gr}(R))_{\gamma}}$, 
${\gamma\in \Gamma}$ one can choose ${J_a\in \mathcal{E}_{gr}(P_{gr}(R))}$, 
${[J_a, \xi_a x_a]\subseteq P_{gr}(R)}$, ${a\in I}$. 
For ${J=\sum\limits_{a\in I} \xi_a J_a\in \mathcal{E}_{gr}(P_{gr}(R))}$ 
(if ${x\in \Ann_{P_{gr}(R)} J}$, then ${\xi_a x_a=0}$ for all ${a\in I}$, 
${x=0}$) and ${D\in \PDer_{gr}(J, P_{gr}(R))_{\gamma}}$, 
\[
D\,:\ \sum_{i=1}^k \xi_{a_i} y_{a_i}\longmapsto 
\sum_{i=1}^k \xi_{a_i} [y_{a_i}, x_{a_i}]\quad (y_a\in J_a)\,,
\]
there exists ${x\in Q_{m-gr}(P_{gr}(R))_{\gamma}, D=\ad_x|_J, 
[J_a, \xi_a x-\xi_a x_a]=\{0\}, \xi_a x=\xi_a x_a}$ for all ${a\in I}$. Therefore 
\[O_{gr}(P_{gr}(R))\subseteq 
Q_{m-gr}(P_{gr}(R))=O_{gr}(Q_{m-gr}(P_{gr}(R)))\subseteq 
I_{gr}\bigr(P_{gr}(R)_{\Ad^{Q_{m-gr}(P_{gr}(R))}(P_{gr}(R))'}\bigr)\,.\]

If ${x\in Q_{m-gr}(P_{gr}(R))_{\gamma}}$, ${\gamma\in \Gamma}$, then 
\begin{gather*}
H\ =\ \{z\in O_{gr}(P_{gr}(R))\mid [z, (x)_{O_{gr}(P_{gr}(R))}]\subseteq 
O_{gr}(P_{gr}(R))\}\ =\ O_{gr}(H) \lhd_{gr} O_{gr}(P_{gr}(R))\,,
\\
L\ =\ \{z\in P_{gr}(R)\mid [z, (x)_{P_{gr}(R)}]\subseteq P_{gr}(R)\}\in 
\mathcal{E}_{gr}(P_{gr}(R))\,,\ 
L\ \subseteq\ H\cap P_{gr}(R)\in \mathcal{E}_{gr}(P_{gr}(R))
\end{gather*} 
\cite[Lemma~2.10]{Mol}, ${H\in \mathcal{E}_{gr}(O_{gr}(P_{gr}(R)))}$. 
If ${P=B_{gr}(R)\setminus B}$, ${y\in Q_{m-gr}(P_{gr}(R))_{\sigma}}$, 
${\sigma\in \Gamma}$, ${[y, H]=O_{gr}([y, H])\subseteq P Q_{m-gr}(P_{gr}(R))}$, 
then there exists ${\alpha\in B}$, ${\alpha [y, H]=[\alpha y, H]=\{0\}}$, 
${\alpha y=0}$, ${y\in P Q_{m-gr}(P_{gr}(R))}$ (graded version of 
Proposition 2.4 (6)). So, 
\[
\Ann_{Q_{m-gr}(P_{gr}(R))_B} (H+P Q_{m-gr}(P_{gr}(R)))/P Q_{m-gr}(P_{gr}(R))\ =\ \{0\}\,,
\]
${(H+P O_{gr}(P_{gr}(R)))/P O_{gr}(P_{gr}(R))\in 
\mathcal{E}_{gr}(O_{gr}(P_{gr}(R))_B)}$ and $Q_{m-gr}(P_{gr}(R))_B$ is a 
$\Gamma$--graded algebra of quotients of $O_{gr}(P_{gr}(R))_B$.
\end{proof}

The ungraded version of Proposition 3.1 is: 

\begin{prop}
If a Lie algebra $R$ is semiprime, then $P(R)$ is an algebra of quotients of 
$R$, ${P(R)\hookrightarrow Q_m(R)}$ identically on $R$, 
${O(P(R))\subseteq Q_m(P(R))=O(Q_m(P(R))), Q_m(P(R))_B}$ is an algebra of 
quotients of $O(P(R))_B$ for any ultrafilter $B$ in $B(R)$.
\end{prop}

From this, Lemma 2.9 and \cite[Proposition~2.7]{Mol} it follows that for a 
non-degenerate $R$ without $6$--torsion $P(R)$ is non-degenerate, $O(P(R))_B$, 
$Q_m(P(R))_B$ are strongly prime for any ultrafilter $B$ in $B(R)$ 
($P_{gr}(R)$ inherits the (homogeneous) non-degeneracy of 
${(R, \{R_{\gamma}\}_{\gamma\in \Gamma})}$ without $6$--torsion 
\cite[Propositions~2.4,~2.7]{OMol}). For an ordered 
${(\Gamma, \leq)}$, the homogeneous non-degeneracy and non-degeneracy of 
${(R, \{R_{\gamma}\}_{\gamma\in \Gamma})}$ are equivalent, since for a 
homogeneously non-degenerate $R$ and ${x=\sum\limits_{i=1}^k x_{\gamma_i}\in R}$, 
${\gamma_j< \gamma_{j+1}}$, ${\ad_x (F \Id_R+\ad(R)) \ad_x=\{0\}}$, 
\[
\ad_{x_{\gamma_i}} (F \Id_R+\ad(R)) \ad_{x_{\gamma_i}}\ =\ \{0\}\,,\quad 
x_{\gamma_i}\ =\ 0\quad (i=k,\ldots, 1)
\]
(${\alpha+\alpha'< \beta+\beta'}$ for all ${\alpha, \alpha', \beta, \beta'\in 
\Gamma}$, ${\alpha< \beta}$, ${\alpha'\leq \beta'}$), ${x=0}$. 

If ${(R, \{R_i\}_{\mathbb{Z}})}$ is an algebra with finite 
$\mathbb{Z}$--grading and ${n\geq 1}$, ${R_i=\{0\}}$ for all ${|i|> n}$, 
then $R$ is also called a \emph{${2 n+1}$--graded algebra}.

In discussing the Jordan version of the maximal ring of quotients from 
\cite{JMQ, OMol} we restrict ourselves to the non-degenerate case, the 
filter of all essential ideals and $F$ with $1/6$. The latter allows us to 
consider linear triple Jordan systems (Jordan pairs) as quadratic and to 
realize quadratic triple Jordan systems (Jordan pairs) through their linear 
specializations (similar to linear and quadratic Jordan algebras over 
$F$ with $1/2$), preserves the $3$--grading of $Q_m(R)$ ($Q_{\mathcal{F}}(R)$) 
for the $3$--graded non-degenerate Lie algebra $R$ and connects the 
conditions of non-degeneracy (McCrimmon and Kostrikin radicals) of Jordan pairs 
and their Tits --- Kantor --- Koecher algebras \cite[Corollary~4.16]{Gol6}. 
More general approaches to systems of quotients of Jordan systems are 
described in \cite{Mont, AMc}, constructions of maximal systems of 
$\mathfrak{M}$--quotients (Martindale-like systems of quotients) of 
non-degenerate Jordan systems over $F$ with $1/6$ with respect to any power 
filter of their ideals with zero annihilators $\mathcal{F}$ are considered 
in \cite{JMQ}. 

A linear Jordan algebra $Q$ is an \emph{algebra of $\mathfrak{M}$--quotients of 
a linear Jordan algebra} $R$ (over $F$ with $1/2$) if $R$ is a subalgebra of $Q$ 
and for any ${0\ne q\in Q}$ one can choose ${I\lhd R}$, 
${\Ann_R I=\{x\in R\mid x I=\{x, I, R\}=\{0\}\}=\{0\}}$, 
${\{0\}\ne I q\subseteq R}$. If ${R_T=(R, \{\ ,\ ,\ \})}$ is $R$ as a linear 
triple Jordan system, then 
${(x y) z=1/2 (\{x, y, z\}+\{y, x, z\}), x, y, z\in R}$ and for all 
${I\lhd R_T}$ one has ${(I)_R=I+I R, (I)_R^3\lhd R, (I)_R^3\subseteq I}$ 
\cite[Lemma~3,~p.~109]{ZShS} (${(I R)^2\subseteq I}$).
From this and \cite[Lemma~0.4]{JMQ} it follows that for a semiprime $R$ any 
${I\in \mathcal{E}(R_T)}$ contains 
${(I)_R^3\in \mathcal{E}(R)}$ and 
${\{I_T\mid I\in \mathcal{E}(R)\}\subseteq \mathcal{E}(R_T)}$. 
If ${V(S)=(S^{\pm}=S, \{\ ,\ ,\ \}^{\pm}=\{\ ,\ ,\ \})}$ is a linear 
Jordan pair associated with a linear triple Jordan system ${(S, \{\ ,\ ,\ \})}$, 
then for the exchange automorphism 
${\phi_{ex}: (x^{+}, y^{-})\longmapsto (y^{+}, x^{-})}$, ${x, y\in S}$, 
of $V(S)$, ${(I^{+}, J^{-})\in \mathcal{E}(V(S))}$, 
\[
\phi_{ex}((I^{+}, J^{-}))\ =\ (J^{+}, I^{-})\,,\quad 
(I^{+}, J^{-})\cap (J^{+}, I^{-})\ =\ ((I\cap J)^{\pm})\in \mathcal{E}(V(S))\,.
\]
Taking this into account, the maximal algebra of $\mathfrak{M}$--quotients 
$Q_m(R)$ of a non-degenerate $R$ over $F$ with $1/6$, in which all algebras 
of $\mathfrak{M}$--quotients of $R$ are embedded identically on $R$, can be 
described up to an isomorphism identical on $R$, as in 
\cite[Theorem~4.11]{OMol} and \cite[Theorem~5.5]{JMQ}: 
\[
Q_m(R)\ =\ Q_m(R_T)\ =\ Q_m(\TKK(V(R_T)))_1\,,
\]
where the Lie algebra $\TKK(V(R_T))$ is the Tits --- Kantor --- Koecher 
algebra of $V(R_T)$, 
\[
Q_m(\TKK(V(R_T)))\ =\ Q_{m-gr}(\TKK(V(R_T)))\ =\ 
\bigoplus_{i=0, \pm 1} Q_{m-gr}(\TKK(V(R_T)))_i\,,
\]
${Q_m(R)=Q_m(\TKK(V(R_T)))_1}$ is a non-degenerate (strongly prime for the 
strongly prime $R$) linear Jordan algebra with 1 and  
${x y=\{x, 1, y\}, \{x, y, z\}=[z, [x, \phi_{ex}''(y)]], [\ ,\ ]}$ is the 
multiplication of $Q_m(\TKK(V(R_T)))$, ${x, y, z\in Q_m(\TKK(V(R_T)))_1}$,
\begin{gather*}
\phi_{ex}'\,:\ a^{+}+T+b^{-}\longmapsto b^{+}+T^{\phi_{ex}}+a^{-}\,,\quad 
T\ =\ \sum_{i=1}^k (\{a_i^{+}, b_i^{-},\ \}^{+}-\{b_i^{-}, a_i^{+},\ \}^{-})\,,
\\
T^{\phi_{ex}}\ =\ \phi_{ex} T \phi_{ex}\ =\ \sum_{i=1}^k 
(\{a_i^{-}, b_i^{+},\ \}^{-}-\{b_i^{+}, a_i^{-},\ \}^{+})\quad 
(a, b, a_i, b_i\in R,\ k\geq 1)\,,
\\ 
\phi_{ex}''\,:\ [(D, J)]\longmapsto [(D^{\phi_{ex}'}, \phi_{ex}'(J))]\quad 
(D\in \PDer(J, \TKK(V(R_T))),\ J\in \mathcal{E}_{gr}(\TKK(V(R_T))))\,,
\end{gather*} 
${D^{\phi_{ex}'}=\phi_{ex}' D \phi_{ex}'}$, 
${\PDer(J, \TKK(V(R_T)))=\PDer_{gr}(J, \TKK(V(R_T)))}$ 
\cite[Sec.~1.9]{JMQ}, 
($\phi_{ex}$, $\phi_{ex}'$ and $\phi_{ex}''$ are automorphisms of $V(R_T)$, 
$\TKK(V(R_T))$ and $Q_m(\TKK(V(R_T)))$ of order 2; 
${Q_m(\TKK(V(R_T)))=Q_{m-gr}(\TKK(V(R_T)))}$ up to a homogeneous isomorphism),
 
\begin{prop}
If $R$ is a non-degenerate linear Jordan algebra over a ring $F$ with $1/6$, 
\linebreak then $P(R)$ is an algebra of $\mathfrak{M}$--quotients of $R$, 
${P(R)\hookrightarrow Q_m(R)}$ identically on $R$,
$O(P(R))\subseteq Q_m(P(R))=O(Q_m(P(R)))$ and $Q_m(P(R))_B$ is an algebra 
of $\mathfrak{M}$--quotients of $O(P(R))_B$ for any ultrafilter $B$ in 
$B(R)$. 
\end{prop}

\begin{proof}
The first statement follows immediately from the properties of 
${P(R)\cong {}_R Q}$ and we can assume that 
${R\subseteq \End(Q_m(R))_{M(R)'} R\subseteq P(R)\subseteq 
Q_m(R)\subseteq I\bigl(R_{M^{Q_m(R)}(R)'}\bigr)}$. Due to 
\begin{multline*}
\CM(R) \Id_{P_{gr}(\TKK(V(P(R)_T)))}\subseteq \CM_{gr}(\TKK(V(P(R)_T)))\,,
\\
\shoveleft{
{E'}^{\perp}\ =\ \{x\in I_{gr}(P_{gr}(\TKK(V(P(R)_T))))\mid E' x=\{0\}\}\ =\ 
{E'}^{\perp}\cap \TKK(V(P(R)_T))\ =}
\\ 
\{x\in \TKK(V(P(R)_T))\mid E x=\{0\}\}\ =\ \{0\}\,,\quad 
\Ann_{\CM_{gr}(\TKK(V(P(R)_T)))} E'\ =\ \{0\}
\end{multline*}
for all ${E\subseteq \CM(R)}$, ${\Ann_{\CM(R)} E=\{0\}}$, 
${E'=E \Id_{P_{gr}(\TKK(V(P(R)_T)))}}$ (Remark 2.2),
it follows from the proof of Proposition 3.1 that up to isomorphism 
\begin{multline*}
({\TKK(V(O(P(R))_T))}_{\pm 1})\ =\ 
(O_{gr}'(\TKK(V(P(R)_T)))_{\pm 1})\,,
\\
\shoveleft{\TKK(V(O(P(R))_T))\subseteq O_{gr}'(\TKK(V(P(R)_T)))\subseteq}
\\ 
Q_{m-gr}(\TKK(V(P(R)_T)))\ =\ O_{gr}'(Q_{m-gr}(\TKK(V(P(R)_T))))\,,
\end{multline*}
where $O_{gr}'(S)$ is the set of all 
${{\sum\limits_{a\in I}}^{\perp} \xi_a x_a}$ for a dense orthogonal 
${\{\xi_a\}_{a\in I}\subseteq B(R)}$, $\{x_a\}_{a\in I}\subseteq S
\subseteq I_{gr}(P_{gr}(\TKK(V(P(R)_T))))$ 
($O_{gr}'(S)$ inherits the property of $S$ to be a $\CM(R)$--subalgebra of 
$I_{gr}(P_{gr}(\TKK(V(P(R)_T))))$), 
and ${O(P(R))\subseteq Q_m(P(R))=O(Q_m(P(R)))}$.

The epimorphism of linear Jordan algebras 
${\psi_B: O(P(R))\longrightarrow O(P(R))_B}$ induces an epimorphism of linear 
Jordan pairs ${\psi_{B, V}: V(O(P(R))_T)\longrightarrow V({O(P(R))_T}_B)}$ 
and a homogeneous epimorphism of 3--graded Lie algebras 
\begin{gather*}
\psi_{B, \TKK}\,:\ 
\TKK(V(O(P(R))_T))\longrightarrow \TKK(V({O(P(R))_T}_B))\,,
\\
\psi_{B, \TKK}\,:\ x^{+}+\sum_{i=1}^k [a_i^{+}, b_i^{-}]+y^{-}
\longmapsto 
\psi_B(x)^{+}+\sum_{i=1}^k [\psi_B(a_i)^{+}, \psi_B(b_i)^{-}]+\psi_B(y)^{-}
\end{gather*}
${x, y, a_i, b_i\in O(P(R)), k\geq 1}$ with kernel 
\[
\Ker \psi_{B, \TKK}\ =\ 
\Ker \psi_B^{+}+
\{z\in [O(P(R))^{+}, O(P(R))^{-}]\mid [O(P(R))^{\pm}, z]\subseteq 
\Ker \psi_B^{\pm}\}+\Ker \psi_B^{-}\,,
\]
${\Ker \psi_B=P O(P(R)), P=B(R)\setminus B}$ \cite[Sec.~4~and~the~observations~before~Example~14]{Gol6}. If 
${a_i, b_i\in O(P(R))}$, ${k\geq 1}$,
\[
z\ =\ \sum_{i=1}^k [a_i^{+}, b_i^{-}]\ =\ 
\sum_{i=1}^k (\{a_i^{+}, b_i^{-},\ \}^{+}-\{b_i^{-}, a_i^{+},\ \}^{-})\,,\quad 
[O(P(R))^{\pm}, z]\subseteq P O(P(R))^{\pm}\,,
\]
then ${\phi_{\pm}(O(P(R)))=O(\phi_{\pm}(O(P(R))))\subseteq P O(P(R))}$, 
${\phi_{+}=\sum\limits_{i=1}^k \{a_i, b_i,\ \}}$, 
${\phi_{-}=\sum\limits_{i=1}^k \{b_i, a_i,\ \}}$, \linebreak and there are 
${\alpha_{\pm}\in B, \alpha_{\pm} \phi_{\pm}(O(P(R)))=\{0\}}$, 
${\alpha z=0, \alpha=\alpha_{+} \alpha_{-}\in B}$ (Proposition 2.4 (6)), 
${z\in P [O(P(R))^{+}, O(P(R))^{-}]}$. Hence  
${\Ker \psi_{B, \TKK}=P \TKK(V(O(P(R))_T))}$,
\[ 
({\TKK(V(O(P(R))_T))_B}_{\pm 1})\ =\ 
({\TKK(V({O(P(R))_B}_T))}_{\pm 1})\ =\ 
({O_{gr}'(\TKK(V(P(R)_T)))_B}_{\pm 1})\,,
\]
${Q_{m-gr}(\TKK(V(P(R)_T)))_B=Q_m(\TKK(V(P(R)_T)))_B}$ is a $3$--graded 
algebra of quotients of $O_{gr}'(\TKK(V(P(R)_T)))_B$ and, due to the 
strong primeness of $O_{gr}'(\TKK(V(P(R)_T)))_B$, its algebra of quotients 
(the proofs of Proposition 3.1 and Lemma 2.9 (their adaptations for the 
$O_{gr}'$--version of orthogonal completion), 
\cite[Proposition~2.7]{OMol}, 
\cite[Remark~4.32~and~comments~to~it]{Gol6}). 

To complete the proof we will 
need some information about $3$--graded Lie algebras, concepts related to 
linear Jordan pairs are given below after Proposition 3.5. Let 
${(L, \{L_i\}_{|i|\leq 1})}$ be one of such Lie algebras, 
${K=\{x\in L_0\mid [L_{-1}+L_1, x]=\{0\}\}}$. If 
${\ad_x \ad(L) \ad_y\ne \{0\}}$ for all ${0\ne x, y\in L_i}$, 
${i=0}$, ${\pm 1}$ (this is true, for example, for the strong prime 
$L$ \cite[Theorem~1.3]{GK}), then either ${K=L}$, 
${L_{\pm 1}=\{0\}}$, or ${K=\{0\}}$, due to 
${K\lhd L, [[[L, x], L], y]\subseteq [K, y]=\{0\}}$ for all 
${x\in K, y\in L_{\pm 1}}$. If ${K=\{0\}}$, the linear Jordan pair 
${(L_{\pm 1}, \{\ ,\ ,\ \}^{\pm}=[[\ ,\ ],\ ])}$ is semi\-prime, then for any 
${\{0\}\ne I\lhd_{gr} L}$, ${\{0\}\ne (I_{\pm})\lhd (L_{\pm})}$ 
(otherwise ${I_{-1}+I_1=\{0\}}$, $I=I_0\subseteq K=\{0\}$, a 
contradiction),
\[
[[I_1, I_{-1}], I_1]+[[I_{-1}, I_1], I_{-1}]\ \ne\ \{0\}\,,\ 
I^2\ =\ [I_{-1}, I_0]\oplus ([I_0, I_0]+[I_{-1}, I_1])\oplus [I_1, I_0]\ \ne\ \{0\}\,,
\] 
$L$ is homogeneously semiprime. If ${K=\{0\}}$, ${I\lhd_{gr} L}$, 
${I_V=I_{-1}\oplus ([I_{-1}, L_1]+[I_1, L_{-1}])\oplus I_1}$, 
${\Ann_{(L_{\pm 1})} (I_{\pm 1})=(J_{\pm 1})}$, then 
${[[J_{\pm 1}, I_{\mp 1}], L_{\pm 1}]=
[[I_{\pm 1}, L_{\mp 1}], J_{\pm 1}]=\{0\}}$,  
${J_{\pm 1}\subseteq \Ann_L I_V}$ and ${I_V\lhd L}$.
If $L$ is homogeneously semiprime, ${K=\{0\}, 
(Q, \{Q_i\}_{|i|\leq 1})}$ is a $3$--graded algebra of quotients of $L$, 
then, as in \cite[the~proof~of~Theorem~4.5]{OMol}, it can be shown that 
$(Q_{\pm 1})$ is a pair of $\mathfrak{M}$--quotients of $(L_{\pm 1})$. 
Indeed, if ${0\ne q_{\pm 1}\in Q_{\pm 1}}$, ${I\lhd_{gr} L}$, 
${\Ann_L I=\{0\}}$ (equivalent to ${I\in \mathcal{E}(L)}$ (homogeneous 
semiprimeness of $L$)), ${\{0\}\ne [I, q_{\pm 1}]\subseteq L}$, then 
${I\cap J\ne \{0\}}$, ${I_{-1}\cap J_{-1}+I_1\cap J_1\ne \{0\}}$ for any 
${\{0\}\ne J\lhd_{gr} L}$ 
(otherwise ${I\cap J=I_0\cap J_0\subseteq K=\{0\}}$, a contradiction), 
${I_V\in \mathcal{E}(L), \Ann_L I_V=\Ann_Q I_V=\{0\}}$ 
\cite[Lemma~2.11]{Mol},  
${\Ann_{(L_{\pm 1})} (I_{\pm 1})=\{0\}}$ (see above), 
\[
\{0\}\ \ne\ [I_V, q_{\pm 1}]\ =\ [[I_{-1}, L_1]+[I_1, L_{-1}], q_{\pm 1}]\oplus 
[I_{\mp 1}, q_{\pm 1}]\subseteq [I, q_{\pm 1}]\subseteq L_{\pm 1}\oplus L_0\,,
\] 
${[[I_{\mp 1}, q_{\pm 1}], L_{\pm 1}]\ne \{0\}}$ (${K=\{0\}}$) and/or  
${[[q_{\pm 1}, I_{\mp 1}], L_{\pm 1}]+[[q_{\pm 1}, L_{\mp 1}], I_{\pm 1}]\ne 
\{0\}}$. Applying these conclusions to ${L=O_{gr}'(\TKK(V(P(R)_T)))_B}$ and 
${Q=Q_m(\TKK(V(P(R)_T)))_B}$, we obtain that 
${({Q_m(\TKK(V(P(R)_T)))_B}_{\pm 1})}$ is a pair of $\mathfrak{M}$--quotients 
of $V(O(P(R))_T)_B$, the linear triple Jordan system (algebra with 1) 
${{Q_m(\TKK(V(P(R)_T)))_B}_1}$ is an algebra of $\mathfrak{M}$--quotients 
of ${O(P(R))_T}_B$ ($O(P(R))_B$) \cite[Secs.~4,~5]{JMQ}.
\end{proof}

Let us briefly dwell on the analogs of the obtained conclusions for triple 
Jordan and Lie systems.
Let ${(R, \{\ ,\ ,\ \})}$ be a linear triple system over $F$ (a $F$--module 
$R$ with trilinear\linebreak composition ${\{\ ,\ ,\ \}: R^3\longrightarrow R}$). 
Recall that the ideals of $R$ are its $F$--submodules $I$ such that 
${\{I, R, R\}+\{R, I, R\}+\{R, R, I\}\subseteq I}$, $R$ is called semiprime 
($\mathcal{R}$--semiprime) if ${\{I, I, I\}\ne \{0\}}$ 
(${\{I, I, R\}+\{I, R, I\}+\{R, I, I\}\ne \{0\}}$) for all 
${\{0\}\ne I\lhd R}$, and prime ($\mathcal{R}$--prime) if $R$ semi\-prime 
($\mathcal{R}$--semiprime) and ${I\cap J\ne \{0\}}$ for all ${\{0\}\ne I, J\lhd R}$. 
For the $\mathcal{R}$--semiprime $R$, one can define the central closure 
${P(R)=\End(I(R))_{M(R)'}}$ and the Martindale centroid  
$\CM(R)=\End(P(R))_{M(R)'}$, where  
\[
M(R)'\ =\ F \Id_R+M(R)\,,\quad 
M(R)\ =\ \langle \{x, y,\ \}, \{x,\ ,y\}, \{\ , x, y\}\mid 
x, y\in R\rangle\,,
\]
$\CM(R)$ is commutative and regular (a field for the $\mathcal{R}$--prime $R$), 
${P(R)=\CM(R) R}$ is a $\mathcal{R}$--semiprime ($\mathcal{R}$--prime for the $\mathcal{R}$--prime $R$) 
linear triple system over $\CM(R)$ with composition ${\{\ ,\ ,\ \}}$, 
extended from $R$ by $\CM(R)$--linearity \cite{Raz, Gol6}. 

The construction of orthogonal completion for $\mathcal{R}$--semiprime linear triple 
systems is similar to that described above for semiprime linear algebras. 
Moreover, it can be easily generalized to semiprime algebras of signature 
$\Omega$ from \cite{Raz}.

For a non-degenerate linear triple Jordan system $R$  
(${U_x=\{x,\ , x\}\ne 0}$ for all ${0\ne x\in R}$) without $6$--torsion, 
the following are equivalent: primeness; $\mathcal{R}$--primeness; 
${I\cap J\ne \{0\}}$ for all ${\{0\}\ne I, J\lhd R}$; 
${\{U_x y\mid x\in I,\ y\in J\}\ne \{0\}}$ for all 
${\{0\}\ne I, J\lhd R}$ \cite{ACLM}, \cite[Secs.~2.1,~2.2]{LP}. 

\begin{lemma}
If a linear triple Jordan system ${(R, \{\ ,\ ,\ \})}$ over a ring $F$ 
with $1/6$ is non-dege\-nerate, $B$ is an ultrafilter in $B(R)$, then the 
linear triple Jordan systems $O(R)_B$, $O(P(R))_B$ are strongly prime.
\end{lemma}

\begin{proof}
Let ${P=B(R)\setminus B}$, ${S=O(R), O(P(R))}$, ${x, y\in S}$ and 
\[
I_{x, y}\ =\ \{z_1 U_y U_{z_2} U_x+z_3 U_y U_{z_4} U_{z_5} U_x
\mid z_i\in R\}\subseteq P S\,.
\]
Similar to linear Jordan algebras, it can be shown that $S$ and $P(R)$ inherit 
the non-degeneracy of $R$. Since ${I_{x, y}=O(I_{x, y})}$, there is 
${\alpha\in B}$, ${\alpha I_{x, y}=I_{\alpha x, y}=\{0\}}$ 
(Proposition 2.4 (6)). If ${Q\lhd S}$, $S/Q$ is 
strongly prime, then ${I_{\alpha x, y}\subseteq Q}$, ${\alpha x\in Q}$ or 
(and) ${y\in Q}$, ${(\alpha x)_S\cap (y)_S\subseteq Q}$ 
\cite[Theorem~2.9]{ACLM}. Consequently, due to the non-degeneracy of $S$ and the 
speciality of $Mc$ \cite{ZelM}, ${(\alpha x)_S\cap (y)_S=\{0\}}$. 
It remains to repeat the reasoning from the end of the proof of Lemma 2.8 
with replacement of \cite[Theorems~1,~2]{BMS} by \cite[Theorem~2.9]{ACLM}.
\end{proof}

A linear triple Jordan system ${(Q, \{\ ,\ ,\ \})}$ is a \emph{Jordan system 
of $\mathfrak{M}$--quotients of a linear Jordan system} ${(R, \{\ ,\ ,\ \})}$ 
if $R$ is a subsystem of $Q$, 
for any ${0\ne q\in Q}$ one can choose ${I\lhd R}$, 
${\Ann_R I=\{x\in R\mid \{x, R, I\}=\{x, I, R\}=\{R, x, I\}=\{0\}\}=\{0\}}$, 
\[
\{0\}\ne \{q, I, R\}+\{q, R, I\}+\{I, q, R\}\subseteq R\,.
\]

The maximal system of $\mathfrak{M}$--quotients $Q_m(R)$ of a non-degenerate 
$R$ over $F$ with $1/6$, into which all systems of $\mathfrak{M}$--quotients 
of $R$ are embedded identically on $R$, can be described up to an isomorphism,  
identical on $R$, as ${Q_m(R)=Q_m(\TKK(V(R)))_1}$, where  
$Q_m(\TKK(V(R)))=Q_{m-gr}(\TKK(V(R)))$ (see above, 
\cite[Theorem~4.10]{OMol},  
\cite[Theorem~4.6]{JMQ}). 

In view of the identities 
\begin{gather*}
\{x, y, z\}\ =\ \{z, y, x\}\,,
\\
\{x, y, \{a, b, c\}\}-\{a, b, \{x, y, c\}\}\ =\ 
\{\{x, y, a\}, b, c\}-\{a, \{y, x, b\}, c\} 
\end{gather*}
defining linear triple Jordan systems, if $R$ is a subsystem of the 
system $R'$, ${X, I\subseteq R'}$, ${(I)_R\subseteq I}$, where 
${(A)_R=A M^{R'}(R)'}$ for all ${A\subseteq R'}$, then 
\begin{multline*}
\{I, R, \{R, R, X\}\}\subseteq \{R, R, \{I, R, X\}\}+
\{\{I, R, R\}, R, X\}+\{R, \{R, I, R\}, X\}\subseteq
\\ 
\shoveright{
\{R, R, \{I, R, X\}\}+\{I, R, X\}+\{R, I, X\}\,,}
\\
\shoveleft{
\{R, I, \{R, R, X\}\}\subseteq \{R, R, \{R, I, X\}\}+
\{I, R, X\}+\{R, I, X\}\,,}
\\
\shoveleft{
\{I, R, \{R, X, R\}\}\subseteq \{R, X, \{I, R, R\}\}+
\{\{I, R, R\}, X, R\}+\{R, \{R, I, X\}, R\}\subseteq}
\\
\shoveright{
\{R, X, I\}+\{I, R, X\}+\{R, \{R, I, X\}, R\}\,,}
\\
\shoveleft{
\{R, I, \{R, X, R\}\}\subseteq \{R, X, I\}+\{I, X, R\}+
\{R, \{I, R, X\}, R\}\,,}
\\
\shoveleft{
\{R, \{R, R, X\}, I\}\subseteq \{\{R, R, R\}, X, I\}+
\{R, X, \{R, R, I\}\}+\{R, R, \{R, X, I\}\}\subseteq}
\\
\shoveright{
\{R, X, I\}+\{R, R, \{R, X, I\}\}\,,}
\\
\shoveleft{
\{R, \{R, X, R\}, I\}\subseteq \{\{X, R, R\}, R, I\}+
\{R, R, \{X, R, I\}\}+\{X, R, \{R, R, I\}\}\subseteq}
\\
\shoveright{
\{R, R, \{I, R, X\}\}+\{I, R, X\}+\{R, I, X\}\,,}
\\
\{(X)_R, I, R\}+\{(X)_R, R, I\}+\{I, (X)_R, R\}\subseteq  
(\{X, I, R\}+\{X, R, I\}+\{I, X, R\})_R\,.
\end{multline*}

\begin{prop}
If ${(R, \{\ ,\ ,\ \})}$ is a non-degenerate linear triple Jordan system over 
a ring $F$ with $1/6$, then $P(R)$ is a linear Jordan system of 
$\mathfrak{M}$--quotients of $R$, ${P(R)\hookrightarrow Q_m(R)}$ identically 
on ${R, O(P(R))\subseteq Q_m(P(R))=O(Q_m(P(R))), Q_m(P(R))_B}$ is a linear 
Jordan system of $\mathfrak{M}$--quotients of $O(P(R))_B$ for any ultrafilter 
$B$ in $B(R)$. 
\end{prop}

\begin{proof}
If ${0\ne x=\sum\limits_{i=1}^k \alpha_i x_i\in P(R)}$, ${k\geq 1}$, 
${x_i\in R}$, ${\alpha_i\in \CM(R)}$, then taking into account the 
observation before Proposition 3.5 for 
${I=\bigcap\limits_{i=1}^k {}_{\alpha_i} R\in \mathcal{E}(R)}$ 
\begin{multline*}
\{0\}\ne 
\{(x)_R\cap I, (x)_R\cap I, R\}+\{(x)_R\cap I, R, (x)_R\cap I\}\subseteq 
\\
\{(x)_R, I, R\}+\{(x)_R, R, I\}+\{I, (x)_R, R\}\subseteq  
(\{x, I, R\}+\{x, R, I\}+\{I, x, R\})_R\subseteq R
\end{multline*}
(Proposition 2.1). Hence ${R\subseteq \End(Q_m(R))_{M(R)'} R\subseteq P(R)\subseteq 
Q_m(R)\subseteq I\bigl(R_{M^{Q_m(R)}(R)'}\bigr)}$. 

The conclusion of the remaining statements is similar to the proof of 
Proposition 3.3. 
\end{proof}

To each linear Jordan pair ${(R^{\pm}, \{\ ,\ ,\ \}^{\pm})}$ there corresponds 
a linear triple Jordan system ${(\Sigma=R^{-}\oplus R^{+}, \{\ ,\ ,\ \})}$,
\[
\{x^{-}+x^{+}, y^{-}+y^{+}, z^{-}+z^{+}\}\ =\ 
\{x^{-}, y^{+}, z^{-}\}^{-}+\{x^{+}, y^{-}, z^{+}\}^{+}\quad 
(x^{\pm}, y^{\pm}, z^{\pm}\in R^{\pm})\,.
\]
The pair ${(R^{\pm})}$ is non-degenerate (semiprime, $\mathcal{R}$--semiprime, prime, 
$\mathcal{R}$--prime, strongly prime) if the system $\Sigma$ is such. For the 
$\mathcal{R}$--semiprime $(R^{\pm})$  
\[
P_{gr}(\TKK((R^{\pm})))\cong \TKK((\CM(\Sigma) R^{\pm}))\,,\quad 
\CM_{gr}(\TKK((R^{\pm})))\cong \CM(\Sigma)\,,
\]
where ${(\CM(\Sigma) R^{\pm})}$ is a subpair of $V(P(\Sigma))$ 
\cite[Proposition~4.37]{Gol6}, and  
\[
O(\CM(\Sigma) R^{-})\cap O(\CM(\Sigma) R^{+})\ =\ \{0\}
\]
(${S\subseteq O(S)\subseteq I(P(\Sigma))=O(I(P(\Sigma)))}$), since if 
${0\ne x\in O(\CM(\Sigma) R^{-})\cap O(\CM(\Sigma) R^{+})}$, then for 
some dense orthogonal ${\{\xi^{\pm}_a\}_{a\in I^{\pm}}\subseteq B(\Sigma)}$, 
${\{x^{\pm}_a\}_{a\in I^{\pm}}\subseteq \CM(\Sigma) R^{\pm}}$, 
${\xi^{\pm}_a x=\xi^{\pm}_a x^{\pm}_a}$, ${a\in I^{\pm}}$, 
\begin{multline*}
\xi^{-}_a \xi^{+}_b \{x, P(\Sigma), P(\Sigma)\}\ =\ 
\CM(\Sigma) \{\xi^{-}_a \xi^{+}_b x, \Sigma, \Sigma\}\ =
\\ 
\CM(\Sigma) \{\xi^{-}_a \xi^{+}_b x, R^{+}, R^{-}\}+
\CM(\Sigma) \{\xi^{-}_a \xi^{+}_b x, R^{-}, R^{+}\}\ =
\\ 
\CM(\Sigma) \{\xi^{-}_a \xi^{+}_b x^{+}_b, R^{+}, R^{-}\}+
\CM(\Sigma) \{\xi^{-}_a \xi^{+}_b x^{-}_a, R^{-}, R^{+}\}\ =\
\{0\}
\end{multline*} 
and analogically ${\{P(\Sigma), \xi^{-}_a \xi^{+}_b x, P(\Sigma)\}=\{0\}}$, 
${\xi^{-}_a \xi^{+}_b x\in \Ann_{P(\Sigma)} P(\Sigma)=\{0\}}$ for all 
${a\in I^{-}}$, ${b\in I^{+}}$ and ${x=0}$ 
(${\{\xi^{-}_a \xi^{+}_b\}_{a\in I^{-},\ b\in I^{+}}^{\perp}=\{0\}}$). 
Therefore, in this case ${O(\Sigma)=O(R^{-})\oplus O(R^{+})}$, 
${P(\Sigma)=\CM(\Sigma) R^{-}\oplus \CM(\Sigma) R^{+}}$ and  
${O(P(\Sigma))=O(\CM(\Sigma) R^{-})\oplus O(\CM(\Sigma) R^{+})}$ are 
Jordan systems of Jordan pairs 
$(O(R^{\pm}))$, $(\CM(\Sigma) R^{\pm})$ and $(O(\CM(\Sigma) R^{\pm}))$ for  
\[
\{x^{\pm}, y^{\mp}, z^{\pm}\}^{\pm}\ =\ \{x^{\pm}, y^{\mp}, z^{\pm}\}\quad 
(x^{\pm}, y^{\pm}, z^{\pm}\in O(\CM(\Sigma) R^{\pm}))\,.
\]

\begin{lemma}
If ${(R^{\pm}, \{\ ,\ ,\ \}^{\pm})}$ is a non-degenerate linear Jordan pair 
over a ring $F$ with $1/6$, $B$ is an ultrafilter in $B(\Sigma)$, then 
the linear Jordan pairs $(O(R^{\pm}))_B$, $(O(\CM(\Sigma) R^{\pm}))_B$ are 
strongly prime. 
\end{lemma}

\begin{proof}
The pairs ${(O(R^{\pm})_B)=(O(R^{\pm}))_B, 
(O(\CM(\Sigma) R^{\pm})_B)=(O(\CM(\Sigma) R^{\pm}))_B}$ are strongly prime 
because their 
Jordan systems $O(\Sigma)_B$, $O(P(\Sigma))_B$ are strongly prime 
(Lemma 3.4, \cite[Lemma~2.2]{ACLM}; due to  
${\alpha, \beta\leq \gamma=\alpha\oplus \beta (\alpha\oplus \Id_{P(\Sigma)})}$, 
${\alpha, \beta, \gamma\in P=B(\Sigma)\setminus B}$, 
$P (O(R^{\pm}))=(P O(R^{\pm}))$, 
${P (O(\CM(\Sigma) R^{\pm}))=(P O(\CM(\Sigma) R^{\pm}))}$).
\end{proof}

A linear Jordan pair ${(Q^{\pm}, \{\ ,\ ,\ \}^{\pm})}$ is a 
\emph{Jordan pair of $\mathfrak{M}$--quotients of a linear Jordan pair 
${(R^{\pm}, \{\ ,\ ,\ \}^{\pm})}$} if ${(R^{\pm})}$ is a subpair of 
${(Q^{\pm})}$ and for any ${0\ne q^{\pm}\in Q^{\pm}}$ there is 
${(I^{\pm})\lhd (R^{\pm})}$, ${\Ann_{(R^{\pm})} (I^{\pm})=\{0\}}$, 
\begin{enumerate}

\item ${\{q^{\pm}, I^{\mp}, R^{\pm}\}^{\pm}+\{q^{\pm}, R^{\mp}, I^{\pm}\}^{\pm}
\subseteq R^{\pm}}$, ${\{I^{\mp}, q^{\pm}, R^{\mp}\}^{\mp}\subseteq R^{\mp}}$;

\item ${\{q^{\pm}, I^{\mp}, R^{\pm}\}^{\pm}+
\{q^{\pm}, R^{\mp}, I^{\pm}\}^{\pm}\ne \{0\}}$ and/or 
${\{I^{\mp}, q^{\pm}, R^{\mp}\}^{\mp}\ne \{0\}}$, 

\end{enumerate}
where for any ${X^{\pm}\subseteq R^{\pm}}$, 
\begin{multline*}
\Ann_{(R^{\pm})} (X^{\pm})\ =
\\ 
\{(x^{\pm})\in (R^{\pm})\mid \{x^{\pm}, X^{\mp}, R^{\pm}\}^{\pm}=
\{x^{\pm}, R^{\mp}, X^{\pm}\}^{\pm}=
\{R^{\pm}, x^{\mp}, X^{\pm}\}^{\pm}=\{0\}\}\,.
\end{multline*}
This is equivalent to the fact that ${Q^{-}\oplus Q^{+}}$ is a system of 
$\mathfrak{M}$--quotients of ${\Sigma=R^{-}\oplus R^{+}}$, since for 
${0\ne q^{\pm}\in Q^{\pm}}$ and ${I\lhd \Sigma}$, ${\Ann_{\Sigma} I=\{0\}}$, 
\begin{multline*}
\{0\}\ne \{q^{\pm}, \Sigma, I\}+\{q^{\pm}, I, \Sigma\}+\{I, q^{\pm}, \Sigma\}\ =
\\
(\{q^{\pm}, R^{\mp}, I^{\pm}\}^{\pm}+\{q^{\pm}, I^{\mp}, R^{\pm}\}^{\pm})\oplus 
\{I^{\mp}, q^{\pm}, R^{\mp}\}^{\mp}\subseteq \Sigma\,,
\end{multline*}
where $I^{\pm}$ is the projection of $I$ onto $R^{\pm}$, 
${(I^{\pm})\lhd (R^{\pm})}$, ${\Ann_{(R^{\pm})} (I^{\pm})=\{0\}}$ 
(in the other direction follows from the definition of the pair 
of $\mathfrak{M}$--quotients; ${I^{-}\oplus I^{+}\lhd \Sigma}$ for 
${(I^{\pm})\lhd (R^{\pm})}$ and ${\Ann_{\Sigma} I^{-}\oplus I^{+}=\{0\}}$, 
if ${\Ann_{(R^{\pm})} (I^{\pm})=\{0\}}$). 

The maximal pair of $\mathfrak{M}$--quotients $Q_m((R^{\pm}))$ of a 
non-degenerate $(R^{\pm})$ over $F$ with $1/6$, into which all pairs 
of $\mathfrak{M}$--quotients of $(R^{\pm})$ are embedded identically on 
$(R^{\pm})$, coincides up to an isomorphism identical on $(R^{\pm})$, with   
${Q_m((R^{\pm}))=(Q_m(\TKK((R^{\pm})))_{\pm 1})}$, where 
${Q_m(\TKK((R^{\pm})))=Q_{m-gr}(\TKK((R^{\pm})))}$ (see above, \cite[Theorem~4.7]{OMol}, 
\cite[Theorem~3.2]{JMQ}). 

\begin{prop} 
If ${(R^{\pm}, \{\ ,\ ,\ \}^{\pm})}$ is a non-degenerate linear Jordan pair 
over a ring $F$ with $1/6$, then $(\CM(\Sigma) R^{\pm})$ is a linear Jordan 
pair of $\mathfrak{M}$--quotients of $(R^{\pm})$ and   
$(\CM(\Sigma) R^{\pm})\hookrightarrow Q_m((R^{\pm}))$ identically on 
${(R^{\pm})}$, 
\[
(O(\CM(\Sigma) R^{\pm}))\subseteq 
Q_m((\CM(\Sigma) R^{\pm}))\ =\ 
(O(Q_m(\TKK((\CM(\Sigma) R^{\pm})))_{\pm 1}))\,,
\]
$Q_m((\CM(\Sigma) R^{\pm}))_B$ is a linear Jordan pair of 
$\mathfrak{M}$--quotients of $(O(\CM(\Sigma) R^{\pm}))_B$ 
for any ultra\-filter $B$ in $B(\Sigma)$.
\end{prop}

\begin{proof}
The first statement follows immediately from the fact that the system 
$P(\Sigma)$ of the pair $(\CM(\Sigma) R^{\pm})$ is a system of 
$\mathfrak{M}$--quotients of the system $\Sigma$ (Proposition 3.5; 
\cite[Example~2.12]{JMQ}). Since 
\begin{multline*}
\CM_{gr}(\TKK((\CM(\Sigma) R^{\pm})))\ =\ \CM(P(\Sigma))\ =\ 
\CM(\Sigma)\ =\ \CM_{gr}(\TKK((R^{\pm})))\,,
\\
\shoveleft{
P_{gr}(\TKK((\CM(\Sigma) R^{\pm})))\ =\ 
\TKK((\CM(P(\Sigma)) \CM(\Sigma) R^{\pm}))\ =}
\\ 
\TKK((\CM(\Sigma) R^{\pm}))\ =\ P_{gr}(\TKK((R^{\pm})))
\end{multline*}
\cite[Proposition~4.37]{Gol6}, 
\begin{multline*}
\TKK((O(\CM(\Sigma) R^{\pm})))\subseteq
O_{gr}(P_{gr}(\TKK((R^{\pm}))))\ =\ O_{gr}(\TKK((\CM(\Sigma) R^{\pm})))
\subseteq 
\\
Q_{m-gr}(\TKK((\CM(\Sigma) R^{\pm})))\ =\ 
O_{gr}(Q_{m-gr}(\TKK(\CM(\Sigma) R^{\pm})))
\end{multline*}
(Proposition 3.1) and 
${Q_{m-gr}(\TKK(\CM(\Sigma) R^{\pm}))=Q_m(\TKK(\CM(\Sigma) R^{\pm}))}$,
\begin{multline*}
(O(\CM(\Sigma) R^{\pm}))\ =\ (O_{gr}(\TKK((\CM(\Sigma) R^{\pm})))_{\pm 1})
\subseteq
\\ 
Q_m((\CM(\Sigma) R^{\pm}))\ =\ 
(O(Q_m(\TKK((\CM(\Sigma) R^{\pm})))_{\pm 1}))\,.
\end{multline*}

It remains to note that  
${Q_{m-gr}(\TKK((\CM(\Sigma) R^{\pm})))_B=
Q_m(\TKK((\CM(\Sigma) R^{\pm})))_B}$ is an algebra of 
quotients of ${O_{gr}(\TKK((\CM(\Sigma) R^{\pm})))_B}$, 
\[
({O_{gr}(\TKK((\CM(\Sigma) R^{\pm})))_B}_{\pm 1})\ =\ 
(O(\CM(\Sigma) R^{\pm})_B)\ =\ (O(\CM(\Sigma) R^{\pm}))_B\,,
\]
and finally 
${Q_m((\CM(\Sigma) R^{\pm}))_B=(Q_m(\TKK((\CM(\Sigma) R^{\pm})))_{\pm 1})_B}$ 
is a pair of $\mathfrak{M}$--quotients of $(O(\CM(\Sigma) R^{\pm}))_B$ 
(Proposition 3.1, the proof of Proposition 3.3).
\end{proof}

Recall that the triple Lie system ${(R, [\ ,\ ,\ ])}$ is a linear triple system 
over $F$ with identities 
\begin{gather*}
[x, x, y]\ =\ 0\,,\quad [x, y, z]+[y, z, x]+[z, x, y]\ =\ 0\,,
\\
[x, y, [a, b, c]]\ =\ [[x, y, a], b, c]+[a, [x, y, b], c]+
[a, b, [x, y, c]]\quad (x, y, z, a, b, c\in R)\,,
\end{gather*}
its standard Lie enveloping ${\mathcal{L}(R)=R\oplus [R, R]}$ is a  
$\mathbb{Z}_2$--graded Lie algebra with ${\mathcal{L}(R)_1=R}$, 
${\mathcal{L}(R)_0=[R, R]=\biggl\{\sum\limits_{i=1}^k [x_i, y_i,\ ]\biggl| 
x_i, y_i\in R,\ k\geq 1\biggr\}\subseteq \Der(R)}$ and the multiplication 
\[
[x_1+D_1, x_2+D_2]\ =\ (x_1 D_2-x_2 D_1)+([x_1, x_2,\ ]+[D_1, D_2])
\quad (x_i\in R,\ D_i\in [R, R])\,,
\]
${[D_1, D_2]=D_1 D_2-D_2 D_1}$, the $F$--submodule ${I\subseteq R}$ is an ideal 
of $R$ if ${[I, R, R]\subseteq I}$, 
\[
\Ann_R I\ =\ \{x\mid [x, I, R]=[x, R, I]=\{0\}\}\ =\ 
R\cap \Ann_{\mathcal{L}(R)} (I)_{\mathcal{L}(R)}\,,\quad 
(I)_{\mathcal{L}(R)}\ =\ I\oplus [I, R]
\]
\cite[Sec.~4]{QTS}, for ${J\lhd_{gr} \mathcal{L}(R)}$ the conditions 
${J\ne \{0\}}$ and ${J_1=J\cap R\ne \{0\}}$ are equivalent (${J_1\lhd R}$, 
${(J_1)_{\mathcal{L}(R)}\subseteq J}$). So, 
${J\in \mathcal{E}(R)}$ if and only if 
${(J)_{\mathcal{L}(R)}\in \mathcal{E}_{gr}(\mathcal{L}(R))}$, and  
${J\in \mathcal{E}_{gr}(\mathcal{L}(R))}$ if and only if 
${J_1\in \mathcal{E}(R)}$                                         
(${(J_1)_{\mathcal{L}(R)}\in \mathcal{E}_{gr}(\mathcal{L}(R))}$). Since 
\begin{multline*}
[I, J, R]\ =\ [J, I, R]\subseteq [J, R, I]+[R, I, J]\ =\ 
[J, R, I]+[I, R, J]\subseteq
\\ 
[(I)_{\mathcal{L}(R)}, (J)_{\mathcal{L}(R)}]\ =\ 
([I, R, J]+[J, R, I])\oplus ([I, J]+[[I, R], [J, R]])\quad (I, J\lhd R)\,,
\end{multline*}
it follows that the $\mathcal{R}$--primeness ($\mathcal{R}$--semiprimeness) of $R$ is equivalent 
to the homogeneous primeness (semiprimeness) of $\mathcal{L}(R)$. For any 
${I\lhd R}$, 
\begin{gather*}
[[x, y, a], b, c]\ =\ 
[x, y, [a, b, c]]-[a, [x, y, b], c]-[a, b, [x, y, c]]\quad 
(x, a\in I,\ y, b, c\in R)\,,
\\
[[I, R, I], R, R]\subseteq 
[I, R, [I, R, R]]+[I, [I, R, R], R]\subseteq 
[I, R, I]+[I, I, R]\ =\ [I, R, I]\,,
\end{gather*}
${[I, R, I]\lhd R}$, ${[I, R, I]\ne \{0\}}$ (${[I, R, I]\in \mathcal{E}(R)}$)
for the $\mathcal{R}$--semiprime $R$ and ${I\ne \{0\}}$ (${I\in \mathcal{E}(R)}$).  

The system $R$ is \emph{non-degenerate} if ${[x,\ , x]\ne 0}$ for all 
${0\ne x\in R}$. For $R$ without $6$--torsion, non-degeneracy is equivalent to 
non-degeneracy (homogeneous non-degeneracy) of $\mathcal{L}(R)$ 
\cite[Theorems~2.3,~2.4]{NLS}. If $R$ without $6$--torsion is 
non-degenerate, ${I\lhd R}$, then the non-degeneracy of $\mathcal{L}(R)$ is 
inherited by its ideal ${I\oplus [I, R]}$ and subsystem $I$ 
\cite[Lemma~4]{Zel}, \cite[the~beginning~of~the~proof~of~Theorem~2.4] {NLS}. 
Consequently, for such $R$ the following are equivalent: primeness; 
$\mathcal{R}$--primeness; ${I\cap J\ne \{0\}}$ for all ${\{0\}\ne I, J\lhd R}$. 

If $R$ is $\mathcal{R}$--semiprime, then ${[P(R), P(R)]\subseteq \Der_{\CM(R)}(P(R))}$, 
$\mathcal{L}(R)$ can be identified with the homogeneous subalgebra of 
$\mathcal{L}(P(R))$ generated by $R$ (${(R)D=\{0\}}$, ${D\in [P(R), P(R)]}$
is equivalent to ${D=0}$), ${\phi|_{P(R)}\in \CM(R)}$ and
\[
[\phi [x, y], z]\ =\ [\phi x, y, z]\ =\ [x, \phi y, z]\ =\ 
[x, y, \phi z]\ =\ \phi [x, y, z]\quad (x, y, z\in P(R))
\]
for all ${\phi\in \End(\mathcal{L}(P(R)))_{M(\mathcal{L}(R))'-gr}}$, 
\[
\CM(R) \Id_{\mathcal{L}(P(R))}\ =\ 
\End(\mathcal{L}(P(R)))_{M(\mathcal{L}(R))'-gr}\ =\ 
\End(\mathcal{L}(P(R)))_{M(\mathcal{L}(P(R)))'-gr}\,.
\]
If ${\psi: M\longrightarrow \mathcal{L}(P(R))}$ is a homogeneous homomorphism 
of a homogeneous $M(\mathcal{L}(R))'$--sub\-module 
${M\subseteq \mathcal{L}(P(R))}$, then ${\psi_1=\psi|_{M\cap P(R)}\in 
\Hom(M\cap P(R), P(R))_{M(R)'}}$ can be extended to 
${\overline{\psi}_1\in \CM(R)}$ and 
${\overline{\psi}_1 \Id_{\mathcal{L}(P(R))}\in 
\End(\mathcal{L}(P(R)))_{M(\mathcal{L}(R))'-gr}}$, 
\[
[(\psi-\overline{\psi}_1 \Id_{\mathcal{L}(P(R))})(M\cap [P(R), P(R)]), R]\ =\ 
(\psi-\overline{\psi}_1)([M\cap [P(R), P(R)], R])\ =\ \{0\}\,,
\]
${\psi=\overline{\psi}_1 \Id_{\mathcal{L}(P(R))}}$. Thus, 
${P_{gr}(\mathcal{L}(R))=\mathcal{L}(P(R))}$, 
${\CM_{gr}(\mathcal{L}(R))=\CM(R) \Id_{\mathcal{L}(P(R))}}$. 

\begin{lemma} 
If a triple Lie system without $6$--torsion ${(R, [\ ,\ ,\ ])}$ is 
non-degenerate, $B$ is an ultrafilter in $B(R)$, then the triple 
Lie systems $O(R)_B$, $O(P(R))_B$ are strongly prime.
\end{lemma}

\begin{proof}
We begin with the following version of Lemma 2.9: if a finitely graded 
$6$--torsion-free Lie algebra ${(R, \{R_{\gamma}\}_{\gamma\in \Gamma})}$ is 
non-degenerate, then $O_{gr}(R)_U$ is homogeneously prime for any ultrafilter 
$U$ in $B_{gr}(R)$. Indeed, $P_{gr}(R)$, $O_{gr}(P_{gr}(R))$, 
$O_{gr}(R)$ are non-degenerate 
(Proposition 3.1, \cite[Propositions~2.4,~2.7]{OMol}) and in the presence 
of ${x\in O_{gr}(R)_{\gamma}, y\in O_{gr}(R)_{\eta}, \gamma, \eta\in \Gamma}$, 
${I_{x, y}=[[[O_{gr}(R), y], O_{gr}(R)], x]=O_{gr}(I_{x, y})\subseteq 
Q O_{gr}(R), Q=B_{gr}(R)\setminus U}$, one can find 
${\alpha\in U}$, ${\alpha I_{x, y}=I_{\alpha x, y}=\{0\}}$ 
(see the graded version of Proposition 2.4 (6)),
$(\alpha x)_{O_{gr}(R)}\cap (y)_{O_{gr}(R)}=\{0\}$ 
\cite[Theorem~7.1.7]{LopJ}, ${x\in Q O_{gr}(R)}$ or  
${y\in Q O_{gr}(R)}$ (the reasoning from the beginning of the proof of 
Lemma 2.5 in the graded version). 
Thus, ${[[[O_{gr}(R)_U, x], O_{gr}(R)_U], y]\ne \{0\}}$ for any 
${0\ne x\in {O_{gr}(R)_{\gamma}}_U}$, 
${0\ne y\in {O_{gr}(R)_{\eta}}_U}$, ${\gamma, \eta\in \Gamma}$, 
$O_{gr}(R)_U$ is homogeneously prime and non-degenerate 
(the observations after Lemma 2.6).

The non-degeneracy of $R$ guarantees the non-degeneracy of $\mathcal{L}(R)$ 
\cite[Theorem~2.4]{NLS}, the homogeneous non-degeneracy and non-degeneracy 
of ${\mathcal{L}(P(R))=P_{gr}(\mathcal{L}(R))}$ and the non-degeneracy $P(R)$ 
(Proposition 3.1, \cite[Proposition~2.4]{OMol}, \cite[Theorems~2.3,~2.4]{NLS}). The conclusion from this that ${S=O(R), O(P(R))}$ and $S_B$ 
are non-degenerate is similar to the observations after Lemma 2.6 and 
so $\mathcal{L}(S)$ is non-degenerate \cite[Theorem~2.4]{NLS},  
$O_{gr}(\mathcal{L}(S))$ is non-degenerate, 
$O_{gr}(\mathcal{L}(S))_B$ is non-degenerate and homogeneously prime (see 
above). If ${x\in \mathcal{L}(S)_i}$, ${y\in \mathcal{L}(S)_j}$, 
${i, j\in \mathbb{Z}_2}$, ${P=B(R)\setminus B}$, 
${J_{x, y}=[[[O_{gr}(\mathcal{L}(S)), x], O_{gr}(\mathcal{L}(S))], y]}$, 
\begin{multline*}
T_{x, y}\ =\ 
\{[[[t_1, x], t_2], y]\mid t_i=a_i+[b_i, c_i],\ a_i, b_i, c_i\in S\}\ =\ 
O_{gr}(T)\ \subseteq 
\\
J'_{x, y}\ =\ 
[[[\mathcal{L}(S), x], \mathcal{L}(S)], y]\ \subseteq\ 
P \mathcal{L}(S)\ \subseteq\ P O_{gr}(\mathcal{L}(S))\,,
\end{multline*}
then there is ${\beta\in B}$, 
${\{0\}=\beta T_{x, y}=\beta J'_{x, y}=\beta J_{x, y}=J_{\beta x, y}}$, 
${(\beta x)_{O_{gr}(\mathcal{L}(S))}\cap (y)_{O_{gr}(\mathcal{L}(S))}=\{0\}}$, 
${x\in P \mathcal{L}(S)}$ or ${y\in P \mathcal{L}(S)}$ (see above). 
Therefore $\mathcal{L}(S)_B$ is homogeneously prime.
An epimorphism of triple Lie systems ${\psi_B: S\longrightarrow S_B}$ induces 
an epimorphism of $\mathbb{Z}_2$--graded Lie algebras 
${\psi_{B, \mathcal{L}}: \mathcal{L}(S)\longrightarrow \mathcal{L}(S_B)}$, 
\begin{gather*}
\psi_{B, \mathcal{L}}\,:\ x+\sum_{i=1}^k [y_i, z_i]\longmapsto 
\psi_B(x)+\sum_{i=1}^k [\psi_B(y_i), \psi_B(z_i)]\quad 
(x, y_i, z_i\in S,\ k\geq 1)\,,
\\
\Ker \psi_{B, \mathcal{L}}\ =\ \Ker \psi_B+
\{z\in [S, S]\mid [z, S]\subseteq \Ker \psi_B=P S\}\,.
\end{gather*}
Since for any ${z\in [S, S]}$, ${[z, S]=O([z, S])
\subseteq P S}$, there is ${\alpha\in B}$, 
$\alpha [z, S]=[\alpha z, S]=\{0\}$, ${\alpha z=0}$, ${z\in P [S, S]}$ 
(Proposition 2.4 (6)), ${\Ker \psi_{B, \mathcal{L}}=P \mathcal{L}(S)}$, 
$\mathcal{L}(S_B)=\mathcal{L}(S)_B$. Due to the homogeneous primeness of 
$\mathcal{L}(S_B)$ and the observations before Lemma 3.8, 
\[
(I)_{\mathcal{L}(S_B)}\cap (J)_{\mathcal{L}(S_B)}\cap S_B\ =\ 
I\cap J\ \ne\ \{0\}\quad (\{0\}\ne I, J\lhd S_B)\,,
\]
$S_B$ is strongly prime.
\end{proof}

A triple Lie system ${(Q, [\ ,\ ,\ ])}$ is a \emph{Lie system of quotients of 
a triple Lie system} ${(R, [\ ,\ ,\ ])}$ if $R$ is a subsystem of $Q$, 
for any ${0\ne q\in Q}$ there exists ${I\lhd R}$, ${\Ann_R I=\{0\}}$, 
$\{0\}\ne [q, I, R]+[q, R, I]\subseteq R$. If $R$ is a subsystem of system 
$R'$, ${X, I\subseteq R'}$, ${(I)_R\subseteq I}$ (in the notation before 
Proposition 3.5), then, in view of 
\begin{multline*}
[I, R, X]+[R, I, X]\subseteq [R, X, I]+[I, X, R]+[X, I, R]+[X, R, I]\subseteq 
[X, I, R]+[X, R, I]\,,
\\
\shoveleft{
[[X, R, R], I, R]\subseteq 
[X, R, [R, I, R]]+[R, [X, R, I], R]+[R, I, [X, R, R]]\subseteq} 
\\
[X, R, I]+[R, [X, R, I], R]+[[R, I, X], R, R]+[X, [R, I, R], R]+
[X, R, [R, I, R]]\subseteq
\\
\shoveright{ 
[X, I, R]+[X, R, I]+[[X, I, R], R, R]+[[X, R, I], R, R]\,,}
\\
\shoveleft{
[[X, R, R], R, I]\subseteq [R, I, [X, R, R]]+[I, [X, R, R], R]\subseteq}
\\
[[R, I, X], R, R]+[X, [R, I, R], R]+[X, R, [R, I, R]]+
[[X, R, R], I, R]\subseteq
\\
[X, I, R]+[X, R, I]+[[X, I, R], R, R]+[[X, R, I], R, R]\,,
\end{multline*}
and, as a consequence, 
\[
[(X)_R, I, R]+[(X)_R, R, I]\subseteq ([X, I, R]+[X, R, I])_R\,,
\] 
the condition in the definition of the system of quotients is equivalent to 
\[
\{0\}\ne [(q)_R, I, R]+[(q)_R, R, I]\subseteq 
([q, I, R]+[q, R, I])_R\subseteq R
\]
for ${(q)_R=q M^Q(R)'}$. Moreover, if ${q\in Q}$, ${J\lhd R}$, 
${\Ann_R J=\{0\}}$ and 
\[
[(q)_R, J, R]+[(q)_R, R, J]\subseteq ([q, J, R]+[q, R, J])_R\ =\ \{0\}\,,
\]
then ${(q)_R\cap R\subseteq \Ann_R J=\{0\}}$, ${q=0}$. If ${0\ne D\in [Q, Q]}$, 
then ${(R)D\ne \{0\}}$, since otherwise for any ${q\in Q}$, 
${q D\ne 0}$, and ${I\lhd R}$, ${\Ann_R I=\{0\}}$, 
${\{0\}\ne [q, I, R]+[q, R, I]\subseteq R}$, 
\[
([q, I, R]+[q, R, I])D\ =\ [q D, I, R]+[q D, R, I]\ =\ \{0\}\,,\quad 
q\ =\ 0\,,
\]
a contradiction.
Therefore $\mathcal{L}(R)$ can be identified with 
${\langle R\rangle\subseteq \mathcal{L}(Q)}$.
If ${J\lhd R}$, ${\Ann_R J=\{0\}}$, ${(J)D=\{0\}}$ and  
${[[J, R], D]=[J, (R)D]=\{0\}}$, then  
\begin{gather*}
[J, R, (R)D]\subseteq ([J, R, R])D+[(J)D, R, R]+[J, (R)D, R]\ =\ \{0\}\,,
\\
[(R)D, J, R]+[(R)D, R, J]\subseteq [J, (R)D, R]+[J, R, (R)D]\ =\ \{0\}\,,
\\
((R)D)_R\cap R\subseteq \Ann_R J\ =\ \{0\}\,,\quad (R)D\ =\ \{0\}\,,
\end{gather*} 
a contradiction. 
For a $\mathcal{R}$--semiprime $R$, ${0\ne q\in Q}$, ${I\in \mathcal{E}(R)}$, 
${\{0\}\ne [q, I, R]+[q, R, I]\subseteq R}$, 
\begin{gather*}
[q, [I, R, I]]\subseteq [[q, [I, R]], I]+[[I, R], [q, I]]\subseteq
[I, R]+[[I, [q, I]], R]+[I, [R, [q, I]]]\subseteq [R, R]\,,
\\
[q, [[I, R, I], R]]\subseteq [[q, [I, R, I]], R]+[[I, R, I], [q, R]]
\subseteq [q, I, R]+[q, R, I]\subseteq R\,,
\\
\{0\}\ne [q, [I, R, I], R]+[q, R, [I, R, I]]\ =\ 
[q, [I, R, I], R]+[q, [[I, R, I], R]]\subseteq R\,.
\end{gather*}
If ${D=\sum\limits_{i=1}^k [q_{i1}, q_{i2}]\ne 0}$, ${q_{ij}\in Q}$, 
${I_{ij}\in \mathcal{E}(R)}$, 
${\{0\}\ne [q_{ij}, I_{ij}, R]+[q_{ij}, R, I_{ij}]\subseteq R}$, ${k\geq 1}$, 
then ${I=\bigcap\limits_{i=1, \ldots, k,\ j=1, 2} I_{ij}, 
K=[[I, R, I], R, [I, R, I]]\in \mathcal{E}(R), (K)D\ne \{0\}}$, and 
taking into account the previous comments for all ${q, q'\in \{q_{ij}\}}$ 
\begin{multline*}
[[q', [q, [I, R, I]]], R, I]+[I, R, [q', [q, [I, R, I]]]]
\subseteq 
[[q', [R, R]], R, I]+[I, R, [q', [R, R]]]\subseteq 
\\
\shoveright{
[(q')_R, R, I]+[I, R, (q')_R]\subseteq 
[(q')_R, I, R]+[(q')_R, R, I]\subseteq R\,,}
\\
\shoveleft{
[[I, R, I], [q', [q, R]], I]\subseteq 
[q', [[I, R, I], [q, R]], I]+[[q', [I, R, I]], [q, R], I]\subseteq}
\\
\shoveright{ 
[q', R, I]+[[R, R], [q, R], I]\subseteq R+[(q)_R, R, I]\subseteq R\,,}
\\
\shoveleft{
(K)D\subseteq 
[([I, R, I])D, R, [I, R, I]]+[[I, R, I], (R)D, [I, R, I]]+
[[I, R, I], R, ([I, R, I])D]\subseteq} 
\\
\shoveright{
\sum_{q\in \{q_{ij}\}} ([(q)_R, R, I]+[I, R, (q)_R])+
\sum_{q, q'\in \{q_{ij}\}} [[I, R, I], [q', [q, R]], I]\subseteq R\,,}
\\
\shoveleft{
[[I, R, I], [q, R, R]]\subseteq 
[[[q, R, R], I, R], I]+[[[q, R, R], I, I], R]+[[[q, R, R], R, I], I]
\subseteq}
\\
\shoveright{
[[(q)_R, I, R], I]+[[(q)_R, I, I], R]+[[(q)_R, R, I], I]
\subseteq [R, R]\,,}
\\
\shoveleft{
[q', [K, [q, R]]]\subseteq 
[q', [R, R, [I, R, I]]+[[I, R, I], [q, R, R], [I, R, I]]+
[[I, R, I], R, R]]\subseteq}
\\
[q', [[I, R, I], R, R]]\subseteq 
[[[I, R, I], q'], R, R]+[[I, R, I], [q', R], R]+[[I, R, I], R, [q', R]]
\subseteq
\\ 
\shoveright{
[[[I, R, I], q', R], R]+[[I, R, I], [q', R], R]+[[I, R, I], [q', R, R]]
\subseteq [R, R]\,,}
\\
\shoveleft{
[[q', K], [q, R]]\subseteq}
\\ 
[[[R, R, R], [I, R, I]]+
[[I, R, I], [q', R], [I, R, I]]+[[I, R, I], R, [R, R]], [q, R]]\subseteq
\\
[[[I, R, I], R]+[[[I, R, I], R, R], R], [q, R]]\subseteq
[[[I, R, I], R], [q, R]]\subseteq
\\
\shoveright{ 
[[[I, R, I], [q, R]], R]+[[I, R, I], [q, R, R]]\subseteq [R, R]\,,}
\\
[[K, R], D]\ =\ [R, R]+[K, (R)D]\subseteq [R, R]+
\sum_{q, q'\in \{q_{ij}\}} ([q', [K, [q, R]]]-[[q', K], [q, R]])\subseteq 
[R, R]\,.
\end{multline*}
Thus, $\mathcal{L}(Q)$ is a $\mathbb{Z}_2$--graded algebra of quotients of 
$\mathcal{L}(R)$. On the other hand, if ${Q_1\oplus Q_0}$ is a 
$\mathbb{Z}_2$--graded algebra of quotients of $\mathcal{L}(R)$, then for 
any ${0\ne q\in Q_1}$ there exists ${I\in \mathcal{E}(R)}$, 
${I\oplus [I, R]\in \mathcal{E}_{gr}(\mathcal{L}(R))}$, 
${\{0\}\ne [I\oplus [I, R], q]=[[I, R], q]\oplus [I, q]\subseteq 
\mathcal{L}(R)}$, and $\{0\}\ne [q, I, R]+[q, R, I]\subseteq R$ 
(otherwise ${[[I, R], q]=[[I, q], R]=[I, R]=\{0\}}$ (${[I, q]\subseteq [R, R]}$), a contradiction), 
${(Q_1, [[\ ,\ ],\ ])}$ is a system of quotients of $R$. This refinement of 
\cite[Proposition~4.3~(iv)]{QTS} can be rewritten in terms of \cite{QTS} for 
any power filter of ideals of $R$. Based on the above and the conclusion of 
\cite{QTS}, the maximal system of quotients $Q_m(R)$ of a $\mathcal{R}$--semiprime $R$, 
into which the systems of quotients of $R$ are embedded identically on $R$, can 
be described, up to an isomorphism identical on $R$, as 
${Q_m(R)=Q_{m-gr}(\mathcal{L}(R))_1}$.  

\begin{prop}
If ${(R, [\ ,\ ,\ ])}$ is a $\mathcal{R}$--semiprime triple Lie system, then $P(R)$ 
is a triple Lie system of quotients of $R$, 
${P(R)\hookrightarrow Q_m(R)}$ identically on $R$ and  
$O(P(R))\subseteq Q_m(P(R))=O(Q_m(P(R)))$, $Q_m(P(R))_B$ is a triple Lie 
system of quotients of $O(P(R))_B$ for any ultrafilter $B$ in $B(R)$. 
\end{prop}

\begin{proof}
Since ${P_{gr}(\mathcal{L}(R))=\mathcal{L}(P(R))}$ is a $\mathbb{Z}_2$--graded 
algebra of quotients of $\mathcal{L}(R)$, $P(R)$ is a system of quotients of 
$R$ (see above, Proposition 3.1). This can be deduced similarly to 
Proposition 3.5: if ${0\ne x=\sum\limits_{i=1}^k \alpha_i x_i\in P(R)}$, 
${k\geq 1}$, ${x_i\in R}$, ${\alpha_i\in \CM(R)}$, then 
${I=\bigcap\limits_{i=1}^k {}_{\alpha_i} R\in \mathcal{E}(R)}$, 
\begin{multline*}
\{0\}\ \ne\ 
[(x)_R\cap I, (x)_R\cap I, R]+[(x)_R\cap I, R, (x)_R\cap I]\subseteq
\\ 
[(x)_R, I, R]+[(x)_R, R, I]\subseteq ([x, I, R]+[x, R, I])_R\subseteq R\,.
\end{multline*}
Consequently, ${R\subseteq \End(Q_m(R))_{M(R)'} R\subseteq 
P(R)\subseteq Q_m(R)\subseteq 
I\bigl(R_{M^{Q_m(R)}(R)'}\bigr)}$. In view of 
\begin{gather*}
\CM_{gr}(\mathcal{L}(R))\ =\ \CM_{gr}(\mathcal{L}(P(R)))\ =\ \CM(R) 
\Id_{\mathcal{L}(P(R))}\,,
\\
Q_{m-gr}(P_{gr}(\mathcal{L}(R)))\ =\ O_{gr}(Q_{m-gr}(P_{gr}(\mathcal{L}(R))))
\end{gather*}
(Proposition 3.1), up to isomorphism  
\begin{multline*}
O(P(R))\ =\  
O_{gr}(\mathcal{L}(P(R)))_1\ =\ O_{gr}(P_{gr}(\mathcal{L}(R)))_1\subseteq 
Q_m(P(R))\ =
\\ 
Q_{m-gr}(\mathcal{L}(P(R)))_1\ =\ Q_{m-gr}(P_{gr}(\mathcal{L}(R)))_1\ =\ 
O_{gr}(Q_{m-gr}(\mathcal{L}(P(R))))_1\ =\ O(Q_m(P(R)))\,.
\end{multline*}

If ${D={\sum\limits_{a\in I}}^{\perp} \xi_a D_a\in O_{gr}(\mathcal{L}(P(R)))_0}$
for a dense orthogonal ${\{\xi_a\}_{a\in I}\subseteq B(R)}$, 
$\{D_a\}_{a\in I}\subseteq [P(R), P(R)]$ and 
${(O(P(R)))D=[O(P(R)), D]=\{0\}}$, then ${\xi_a D_a=0}$ for all ${a\in I}$ and 
${D=0}$ (${\mathcal{L}(P(R))\subseteq \mathcal{L}(O(P(R)))}$). If 
${(O(P(R)))D\subseteq P O(P(R))}$ for ${D\in O_{gr}(\mathcal{L}(P(R)))_0}$, 
$P=B(R)\setminus B$, then ${(O(P(R)))D=O((O(P(R)))D)}$ and there is 
${\alpha\in B}$, ${\alpha (O(P(R)))D=\{0\}}$, ${\alpha D=0}$, 
${D\in P O_{gr}(\mathcal{L}(P(R)))_0}$. So, 
${\{0\}\ne (I_1)_{\mathcal{L}(O(P(R)))}=I_1\oplus [I_1, O(P(R))]\subseteq I}$ 
for any $\{0\}\ne I=I_1\oplus I_0\lhd_{gr} O_{gr}(\mathcal{L}(P(R)))$. 
If ${I\in \mathcal{E}_{gr}(O_{gr}(\mathcal{L}(P(R))))}$,  
${I_1\in \mathcal{E}(O(P(R)))}$ and 
$(I_1)_{\mathcal{L}(O(P(R)))}\in \mathcal{E}_{gr}(O_{gr}(\mathcal{L}(P(R))))$. 
From this, Proposition 3.1 and the remarks before Proposition 3.9 it follows 
that $Q_m(P(R))$ is a system of quotients of $O(P(R))$. 

If ${x\in Q_m(P(R))\setminus P Q_m(P(R))}$, 
${H=O_{gr}(H)\in \mathcal{E}_{gr}(O_{gr}(\mathcal{L}(P(R))))}$, 
\begin{gather*}
H\ =\ H_1\oplus H_0\ =\ \{z\in O_{gr}(\mathcal{L}(P(R)))\mid 
[z, (x)_{O_{gr}(\mathcal{L}(P(R)))}]\subseteq O_{gr}(\mathcal{L}(P(R)))\}\,,
\\
\Ann_{Q_{m-gr}(\mathcal{L}(P(R)))_B} (H+P Q_{m-gr}(\mathcal{L}(P(R))))/
P Q_{m-gr}(\mathcal{L}(P(R)))\ =\ \{0\}\,,
\\
(H+P O_{gr}(\mathcal{L}(P(R))))/P O_{gr}(\mathcal{L}(P(R)))\in 
\mathcal{E}_{gr}(O_{gr}(\mathcal{L}(P(R)))_B)\,,
\end{gather*}
(see the proof of Proposition 3.1), 
${[x, H_1, O(P(R))]+[x, O(P(R)), H_1]\subseteq P Q_m(P(R))}$, then 
${[x, y, T]=O([x, y, T])\subseteq P Q_m(P(R)), 
\alpha_{y, T} [x, y, T]=\{0\}, \alpha_{y, T} [x, y]=0}$
for any ${y\in S}$ at some ${\alpha_{y, T}\in B}$, 
${[x, S]=O_{gr}([x, S])}$ and there is ${\alpha_T\in B}$,
${\alpha_T [x, S]=\{0\}}$, ${(S, T)=(H_1, O(P(R)))}$, ${(O(P(R)), H_1)}$ 
(Proposition 2.4 (6) (${T=O(T)}$, ${S=O(S)}$),  
the remarks before Proposition 3.9: ${(I)D\ne \{0\}}$ 
for all ${0\ne D\in [Q_m(P(R)), Q_m(P(R))], I\in \mathcal{E}(O(P(R)))}$ 
and, in particular, ${I=T}$), 
${[\alpha x, H_1, O(P(R))]+[\alpha x, O(P(R)), H_1]=\{0\}}$ for 
${\alpha=\alpha_{H_1} \alpha_{O(P(R))}\in B, \alpha x=0}$, 
$x\in P Q_m(P(R))$ (the remarks before Proposition 3.9), a contradiction. Thus, 
$Q_m(P(R))_B$ is a system of quotients of $O(P(R))_B$.
\end{proof}

Let us present a version of Propositions 3.1, 3.3, 3.5, 3.7, 3.9 that does not 
use the non-degeneracy condition and the constructions (facts of their existence) 
of maximal algebras, systems and pairs of $\mathfrak{M}$--quotients. The 
concepts of a $\Gamma$--graded algebra 
${(Q, \{Q_{\gamma}\}_{\gamma\in \Gamma})}$, a linear\linebreak triple system 
${(Q, \{\ ,\ ,\ \})}$ and a pair ${(Q^{\pm}, \{\ ,\ ,\ \}^{\pm})}$ of 
$\mathfrak{M}$--quotients for a homogeneously semiprime 
$\Gamma$--graded algebra ${(R, \{R_{\gamma}\}_{\gamma\in \Gamma})}$, the  
$\mathcal{R}$--semiprime linear triple system ${(R, \{\ ,\ ,\ \})}$ and pair 
${(R^{\pm}, \{\ ,\ ,\ \}^{\pm})}$ are similar to those given above: 
$R$ is a homogeneous subalgebra of $Q$ and for any 
${0\ne q\in Q_{\gamma}, \gamma\in \Gamma}$ there is 
${I\in \mathcal{E}_{gr}(R), \{0\}\ne (q)_R I+I (q)_R\subseteq R}$; $R$ is a 
subsystem of $Q$ and for any ${0\ne q\in Q}$ there exists 
${I\in \mathcal{E}(R)}$, 
\[
\{0\}\ne \{(q)_R, I, R\}+\{(q)_R, R, I\}+\{I, (q)_R, R\}+
\{R, (q)_R, I\}+\{I, R, (q)_R\}+\{R, I, (q)_R\}\subseteq R, 
\]
${(q)_R=q M^Q(R)'}$; ${Q^{-}\oplus Q^{+}}$ is a system of 
$\mathfrak{M}$--quotients of ${\Sigma=R^{-}\oplus R^{+}}$. 

\begin{teor}
If $(R, \{R_{\gamma}\}_{\gamma\in \Gamma})$, $(R, \{\ ,\ ,\ \})$ and 
$(R^{\pm}, \{\ ,\ ,\ \}^{\pm})$ are a homogeneously semi\-prime algebra with 
finite $\Gamma$--grading, a $\mathcal{R}$--semiprime linear triple system and a pair, 
then $P_{gr}(R)$, $P(R)$ and $(\CM(\Sigma) R^{\pm})$ are its $\Gamma$--graded 
algebra, system and pair of $\mathfrak{M}$--quotients, respectively.
If $(Q, \{Q_{\gamma}\}_{\gamma\in \Gamma})$, $(Q, \{\ ,\ ,\ \})$ and 
$(Q^{\pm}, \{\ ,\ ,\ \}^{\pm})$ are a $\Gamma$--graded algebra, a system 
and a pair of $\mathfrak{M}$--quotients of $P_{gr}(R)$, $P(R)$ and 
$(\CM(\Sigma) R^{\pm})$, then $O_{gr}(Q)$, $O(Q)$ and $(O(Q^{\pm}))$ are a 
$\Gamma$--graded algebra, a system and a pair of $\mathfrak{M}$--quotients of 
$O_{gr}(P_{gr}(R))$, $O(P(R))$ and $(O(\CM(\Sigma) R^{\pm}))$ and when the 
condition 
\begin{equation}
(x)_{O_{gr}(P_{gr}(R))} I+I (x)_{O_{gr}(P_{gr}(R))}\subseteq 
(x I+I x)_{O_{gr}(P_{gr}(R))}
\end{equation}
for all ${x\in O_{gr}(Q)_{\gamma}}$, ${\gamma\in \Gamma}$, 
${I\lhd_{gr} O_{gr}(P_{gr}(R))}$ is satisfied in $O_{gr}(Q)$,  
the condition  
\begin{equation}
(x)_{O(S)}\cdot (y)_{O(S)}\ =\ \sum_{a\in (x)_{O(S)},\ b\in (y)_{O(S)}} 
a\cdot b\subseteq (x\cdot y)_{O(S)}\quad (x\in O(S'),\ y\in O(S))
\end{equation}
is satisfied in $O(S')$, ${S'=Q, Q^{-}\oplus Q^{+}}$ for ${S=P(R), P(\Sigma)}$, 
where ${(z)_A=z M^C(A)'}$, $z\in C=O_{gr}(R)$, $O(S')$, 
${A=O_{gr}(P_{gr}(R)), O(S)}$, 
\[
a\cdot b\ =\ \{a, b, O(S)\}+\{a, O(S), b\}+\{b, a, O(S)\}+\{O(S), a, b\}+
\{b, O(S), a\}+\{O(S), b, a\}\,,
\]
$O_{gr}(Q)_B$, $O(Q)_B$ and $(O(Q^{\pm}))_B$ are a $\Gamma$--graded algebra, 
a system and a pair of $\mathfrak{M}$--quotients of $O_{gr}(P_{gr}(R))_B$, 
$O(P(R))_B$ and $(O(\CM(\Sigma) R^{\pm}))_B$ for any ultrafilter $B$ in  
$B_{gr}(R)$, $B(R)$ and $B(\Sigma)$. 
\end{teor}

\begin{proof}
The conclusion of the first statement is similar to Propositions 
3.1, 3.5, 3.9, 3.7, taking into account that $P(\Sigma)$ and 
$O(P(\Sigma))$ are systems of pairs $(\CM(\Sigma) R^{\pm})$ and 
$(O(\CM(\Sigma) R^{\pm}))$ (the remarks before Lemma 3.6). 
If ${0\ne x={\sum\limits_{a\in I}}^{\perp} \xi_a x_a\in 
Q'=O_{gr}(Q)_{\gamma}, O(Q), O(Q^{\sigma})}$ for a dense orthogonal 
${\{\xi_a\}_{a\in I}\subseteq B'=B_{gr}(R), B(R), B(\Sigma)}$, 
${\{x_a\}_{a\in I}\subseteq Q''=Q_{\gamma}, Q}$, $Q^{\sigma}$, 
${\gamma\in \Gamma}$, ${\sigma=\pm}$, and 
${I_a\in \mathcal{E}=\mathcal{E}_{gr}(P_{gr}(R)), \mathcal{E}(P(R)), 
\mathcal{E}(P(\Sigma))}$, where in the latter case we can assume 
that ${I_a=I_a^{-}\oplus I_a^{+}}$, 
${(I_a^{\pm})\in \mathcal{E}((\CM(\Sigma) R^{\pm}))}$, 
$I_a^{\pm}$ is the projection of $I_a$ onto $\CM(\Sigma) R^{\pm}$, 
\begin{multline*}
T_a\ =\ 
(\xi_a x_a)_{P_{gr}(R)} I_a+I_a (\xi_a x_a)_{P_{gr}(R)}\subseteq P_{gr}(R)\,,
\\
\shoveleft{
T_a\ =\ \{(\xi_a x_a)_S, I_a, S\}+\{(\xi_a x_a)_S, S, I_a\}+
\{I_a, (\xi_a x_a)_S, S\}+}
\\ 
\{S, (\xi_a x_a)_S, I_a\}+
\{I_a, S, (\xi_a x_a)_S\}+\{S, I_a, (\xi_a x_a)_S\}\subseteq S\,,
\end{multline*}
${S=P(R), P(\Sigma)}$, ${a\in I}$, ${T_a\ne \{0\}}$ for 
${a\in \{a\mid \xi_a x_a\ne 0\}\ne \emptyset}$ (${x\ne 0}$), 
then ${I=\sum\limits_{a\in I} \xi_a I_a\in \mathcal{E}}$, since 
${J\cap I=\{0\}}$ for ${J\lhd P_{gr}(R), S}$ is equivalent to 
${\{0\}=\xi_a (J\cap I_a)=(\xi_a J)\cap I_a=\xi_a J}$ for all ${a\in I}$,
${J=\{0\}}$, and ${T_a=\xi_a T}$, ${a\in I}$,
\begin{multline*}
\{0\}\ \ne\ T\ =\ (x)_{P_{gr}(R)} I+I (x)_{P_{gr}(R)}\subseteq 
\\
\shoveright{
(x)_{O_{gr}(P_{gr}(R))} O_{gr}(I)+O_{gr}(I) (x)_{O_{gr}(P_{gr}(R))}
\subseteq O_{gr}(P_{gr}(R))\,,}
\\
\shoveleft{
\{0\}\ \ne\ T\ =\ \{(x)_S, I, S\}+\{(x)_S, S, I\}+\{I, (x)_S, S\}+ 
\{S, (x)_S, I\}+}
\\ 
\{I, S, (x)_S\}+\{S, I, (x)_S\}\subseteq 
(x)_{O(S)}\cdot O(I)\ =\ \sum_{a\in (x)_{O(S)},\ b\in O(I)} a\cdot b\subseteq O(S)\,.
\end{multline*}

Moreover, ${O_{gr}(I)\in \mathcal{E}_{gr}(O_{gr}(P_{gr}(R)))}$, since if 
${J\lhd O_{gr}(P_{gr}(R)), J\cap O_{gr}(I)=\{0\}}$, then  
${\{0\}=(J\cap P_{gr}(R))\cap I=J\cap P_{gr}(R)=J}$, and, similarly, 
${O(I)\in \mathcal{E}(O(S))}$, and for ${S=P(\Sigma)}$, 
${O(I)=O\Bigl(\sum\limits_{a\in I} \xi_a I_a^{-}\Bigr)\oplus 
O\Bigl(\sum\limits_{a\in I} \xi_a I_a^{+}\Bigr)}$ and 
${\Bigl(O\Bigl(\sum\limits_{a\in I} \xi_a I_a^{\pm}\Bigr)\Bigr)\in 
\mathcal{E}((O(\CM(\Sigma) R^{\pm})))}$. As a result, 
$O_{gr}(Q)$, $O(Q)$ and $(O(Q^{\pm}))$ are a $\Gamma$--graded algebra, 
a system and a pair of $\mathfrak{M}$--quotients of $O_{gr}(P_{gr}(R))$, 
$O(P(R))$ and $(O(\CM(\Sigma) R^{\pm}))$.

If $O_{gr}(P_{gr}(R))$ with (1), ${x\in O_{gr}(Q)_{\gamma}}$, 
${\gamma\in \Gamma}$, $I$ is the largest among ${J\lhd O_{gr}(P_{gr}(R))}$, 
${J\subseteq H}$ (${J\subseteq H'}$, (1)), 
\begin{multline*}
H\ =\ \{z\in O_{gr}(P_{gr}(R))\mid 
(x)_{O_{gr}(P_{gr}(R))} z+z (x)_{O_{gr}(P_{gr}(R))}\subseteq 
O_{gr}(P_{gr}(R))\}\ =\ O_{gr}(H)\subseteq
\\ 
H'\ =\ \{z\in O_{gr}(P_{gr}(R))\mid x z, z x\in O_{gr}(P_{gr}(R))\}\ =\ 
O_{gr}(H')\,,
\end{multline*}
then ${I=O_{gr}(I)\in \mathcal{E}_{gr}(O_{gr}(P_{gr}(R)))}$ (see above), and 
if ${y I+I y=O_{gr}(y I+I y)\subseteq P O_{gr}(Q)}$ for 
${P=B_{gr}(R)\setminus B}$, ${y\in O_{gr}(Q)}$, 
there is ${\alpha\in B}$, ${\alpha y I+I \alpha y=\{0\}}$, 
\begin{multline*}
((\alpha y)_{O_{gr}(P_{gr}(R))}\cap I)^2\subseteq 
(\alpha y)_{O_{gr}(P_{gr}(R))} I+I (\alpha y)_{O_{gr}(P_{gr}(R))}\subseteq
\\ 
(\alpha y I+I \alpha y)_{O_{gr}(P_{gr}(R))}\ =\ \{0\}\,,
\end{multline*}
${\{0\}=(\alpha y)_{O_{gr}(P_{gr}(R))}\cap I=
(\alpha y)_{O_{gr}(P_{gr}(R))}, y\in P O_{gr}(Q)}$ 
(graded version of Proposition 2.4 (6), (1), homogeneous 
semiprimeness of $O_{gr}(P_{gr}(R))$ and the fact that $O_{gr}(Q)$ 
is a $\Gamma$--graded algebra of $\mathfrak{M}$--quotients of 
$O_{gr}(P_{gr}(R))$),
\[
(I+P O_{gr}(Q))/P O_{gr}(Q)\cong (I+P O_{gr}(P_{gr}(R)))/
P O_{gr}(P_{gr}(R))\in \mathcal{E}_{gr}(O_{gr}(P_{gr}(R))_B)\,.
\]
Thus, $O_{gr}(Q)_B$ is a $\Gamma$--graded algebra of $\mathfrak{M}$--quotients 
of $O_{gr}(P_{gr}(R))_B$. 

If $O(S')$ with (2), ${S'=Q, Q^{-}\oplus Q^{+}}$ for ${S=P(R), P(\Sigma), 
x\in O(Q), O(Q^{\sigma}), \sigma=\pm, I}$ is the largest among ${J\lhd O(S)}$, 
${J\subseteq K}$ (${J\subseteq K'}$, (2)),
\begin{multline*}
K\ =\ \{z\in O(S)\mid a\cdot z\in O(S)\ \forall a\in (x)_{O(S)}\}\ =\ 
O(K)\subseteq
\\ 
K'\ =\ \{z\in O(S)\mid x\cdot z\in O(S)\}\ =\ O(K')\,,
\end{multline*}
${I=O(I)\in \mathcal{E}(O(S))}$ (see above), ${I=I^{-}\oplus I^{+}, 
(I^{\pm})\in \mathcal{E}((O(\CM(\Sigma) R^{\pm}))), I^{\pm}}$ is 
the projection of $I$ onto $O(\CM(\Sigma) R^{\pm})$ for 
${S=P(\Sigma)}$, in case ${(y)_{O(S)}\cdot I\subseteq P O(S')}$ for 
${P=B(R)\setminus B}$, ${B(\Sigma)\setminus B}$, ${y\in O(S')}$,  
for any ${z\in I}$ there is ${\alpha_z\in B}$, 
\begin{multline*}
y\cdot z\ =\ O(y\cdot z)\subseteq P O(S')\,,\quad 
\alpha_z (y\cdot z)\ =\ (\alpha_z y)\cdot z\ =\ y\cdot (\alpha_z z)\ =\ \{0\}\,,
\\
\shoveleft{
((\alpha_z y)_{O(S)}\cap (z)_{O(S)})\cdot ((\alpha_z y)_{O(S)}\cap (z)_{O(S)})
\ =\ \sum_{a, b\in (\alpha_z y)_{O(S)}\cap (z)_{O(S)}} a\cdot b
\subseteq}
\\ 
(\alpha_z y)_{O(S)}\cdot (z)_{O(S)}\subseteq
((\alpha_z y)\cdot z)_{O(S)}\ =\ 
(\alpha_z y)_{O(S)}\cap (z)_{O(S)}\ =\ \{0\}
\end{multline*}
(Proposition 2.4 (6); $\mathcal{R}$--semiprimeness of $O(S)$). If there is 
${z\in I}$, ${\alpha_z y=0}$, then ${y\in P O(S')}$. So, one can assume 
that ${\alpha_z y\ne 0}$ and ${(\alpha_z y)_{O(S)}\cap O(S)\ne \{0\}}$ 
for all ${z\in I}$ (see above, $O(S')$\linebreak 
is a system of $\mathfrak{M}$--quotients of $O(S)$). For any 
${z\in I}$ one can choose ${\beta_z\in \CM(S)=\CM(R)}$, $\CM(\Sigma)$, 
${\beta_z u=u}$, ${\beta_z v=0}$ for all 
${u\in (\alpha_z y)_{O(S)}\cap O(S)}$, ${v\in (z)_{O(S)}}$ 
(Proposition 2.4 (4)). If ${\beta_z\in B}$ for all ${z\in I}$, then 
${I=O(I)\subseteq P O(S)}$, there is ${\beta\in B}$, ${\beta I=\{0\}}$ 
(Proposition 2.4 (6)), 
${\{0\}=\beta (O(S)\cap I)=(\beta O(S))\cap I=\beta O(S)}$ and ${\beta=0}$, a contradiction. 
Therefore, there exists ${z\in I}$, ${\beta_z\notin B}$, ${\Id_S-\beta_z\in B}$, 
\[
\{0\}\ =\ (\Id_S-\beta_z) ((\alpha_z y)_{O(S)}\cap O(S))\ =\ 
((\Id_S-\beta_z) \alpha_z y)_{O(S)}\cap O(S)\ =\ 
((\Id_S-\beta_z) \alpha_z y)_{O(S)}\,,
\]
${y\in P O(Q),}$ and, as a consequence, 
${(I+P O(Q))/P O(Q)\cong (I+P O(S))/P O(S)\in \mathcal{E}(O(S)_B)}$ 
and ${((I^{\pm})+P (\CM(\Sigma) R^{\pm}))/P (\CM(\Sigma) R^{\pm})\in 
\mathcal{E}((O(\CM(\Sigma) R^{\pm}))_B)}$ for ${S=P(\Sigma)}$. 
So, $O(Q)_B$ and $(O(Q^{\pm}))_B$ are a system and a pair of 
$\mathfrak{M}$--quotients of $O(P(R))_B$ and $(O(\CM(\Sigma) R^{\pm}))_B$. 
\end{proof}

In conclusion, we consider the implementation of orthogonal completion in the 
maximal right algebra of quotients of an alternative algebra. Following \cite{MLQA}, we call 
an alternative algebra $Q$ a \emph{right algebra of quotients of an alternative 
algebra $R$} if $R$ is a subalgebra of $Q$, ${N(R)\subseteq N(Q)}$ and for 
any ${x, y\in Q}$, ${x\ne 0}$, there is ${z\in N(R)}$, ${x z\ne 0}$, 
${y z\in R}$, where  
\[
N(R)\ =\ \{x\in R\mid (x, R, R)=(R, x, R)=(R, R, x)=\{0\}\}
\] 
is the \emph{associative center of $R$} (in our case 
${N(R)=\{x\in R\mid (x, R, R)=\{0\}\}}$, since the associator of elements of 
the alternative algebra is skew-symmetric in its arguments 
\cite[p.~49]{ZShS}). The presence of a right algebra of quotients for 
$R$, whose associator ideal  
\[
D(R)\ =\ ((R, R, R))_R\ =\ (R, R, R) R^1\ =\ R^1 (R, R, R)
\]
has no $2$--torsion or semiprime (${\prr(D(R))=\{0\}}$), is equivalent to the 
fact that $R$ is a right algebra of quotients of $R$, i.e. 
${\Ann_l N(R)=\{0\}}$ \cite[Theorem~2.13]{MLQA}, 
\cite[Proposition~8,~p.~203]{ZShS}, 
where ${R^1=F\oplus R}$ is the algebra obtained from 
$R$ by standard addition of 1. The semiprimeness (primeness) of $R$ is 
inherited by its ideals 
\cite[Theorem~4,~Corollary~1,~p.~214]{ZShS}. In 
the semiprime $R$, 
\begin{gather*}
N(I)\ =\ I\cap N(R)\,,\quad Z(I)\ =\ I\cap Z(R)\,,
\\
ZN(R)\ =\ ([N(R), R])_R\ =\ [N(R), R] R^1\ =\ R^1 [N(R), R]\subseteq U(R)
\end{gather*}
for any one-sided ideal ${I\subseteq R}$, where $U(R)$ is the 
greatest of ${J\lhd R}$, ${J\subseteq N(R)}$ (associative core of $R$) 
\cite[Theorems~2,~3,~p.~211,~212,~Lemma~4,~p.~168,~Theorem~11,~p.~205]{ZShS}. 

\begin{rem}
If an alternative algebra $R$ is non-degenerate, then ${\Ann_t N(R)=\{0\}}$, 
${t=l, r}$.  
\end{rem}

\begin{proof} 
Since for any ${x\in \Ann_l N(R)}$
\begin{gather*}
x (R N(R))\subseteq x (N(R) R-[N(R), R])\subseteq x (N(R) R+N(R))\
\subseteq (x N(R)) R^1\ =\ \{0\}\,,
\\
(x R) N(R)\ =\ x (R N(R))\ =\ \{0\}\,,\quad 
(R x) N(R)\ =\ R (x N(R))\ =\ \{0\}\,,
\end{gather*}
${\Ann_l N(R)\lhd R}$. From 
\cite[Lemma~7,~Theorem~7,~p.~224,~225]{ZShS} and 
the non-degeneracy of $\Ann_l N(R)$ follows either 
\[
\Ann_l N(R)\ = N(\Ann_l N(R))\subseteq U(R)\subseteq N(R)\,,\quad 
\{0\}\ =\ (\Ann_l N(R))^2\ =\ \Ann_l N(R)
\]
or
\[
Z(\Ann_l N(R))\ =\ \Ann_l N(R)\cap Z(R)\ne \{0\}\ =\ Z(\Ann_l N(R))^2\ =\ 
Z(\Ann_l N(R))\,,
\] 
a contradiction. The conclusion ${\Ann_r N(R)=\{0\}}$ is similar.
\end{proof}

Therefore, due to the inheritance of non-degeneracy of $R$ by its ideals 
\cite[Lemma~7,~p.~224]{ZShS}, the non-degenerate $R$ has left and 
right algebras of quotients. 

\begin{rem}
If $Q$ is a right algebra of quotients of an alternative algebra $R$, then 
\begin{gather*}
N(Q)\ =\ \{x\in Q\mid (x, R, R)=\{0\}\}\,,\quad 
K(Q)\ =\ \{x\in Q\mid [x, R]=\{0\}\}\,,\\
Z(Q)\ =\ \{x\in Q\mid (x, R, R)=[x, R]=\{0\}\}\,,
\end{gather*}
where ${K(Q)=\{x\in Q\mid [x, Q]=\{0\}\}}$ is the commutative center of $Q$.
\end{rem}

\begin{proof}
Since ${N(R)\subseteq N(Q)}$, for any ${x\in Q}$, 
${(x, R, R)=\{0\}}$, and ${y, z\in Q}$, ${(x, y, z)\ne 0}$, 
\[
0\ \ne\ ((x, y, z) a) b\ =\ (x, y, z) (a b)\ =\ (x, y a, z b)\ =\ 0
\] 
for some ${a, b\in N(R)}$ \cite[Lemma~1,~p.~164]{ZShS}, 
${y a, z b\in R}$, a contradiction. If ${x\in Q, [x, R]=\{0\}}$, and 
${y\in Q}$, ${[x, y]\ne 0}$, then there is ${a\in N(R)}$, ${y a\in R}$, 
\[
0\ne [x, y] a\ =\ (x y) a-(y x) a\ =\ x (y a)-y (x a)\ =\ 
x (y a)-y (a x)\ =\ x (y a)-(y a) x\ =\ 0\,, 
\] 
a contradiction. It remains to be noted that ${Z(Q)=N(Q)\cap K(Q)}$.
\end{proof}

A right ideal ${I\subseteq R}$ is called \emph{dense} if for any 
${x, y\in R}$, ${x\ne 0}$, there is ${z\in N(R)}$, ${x z\ne 0}$, 
${y z\in I}$. The density of $I$ is equivalent to the fact that $R$ is a 
right algebra of quotients of $I$ (we can assume that ${z\in N(I)=I\cap N(R)}$) 
\cite[Lemma~2.3]{MLQA}, and implies 
\[
N(I)\ \ne\  \{0\}\ =\ \Ann_l N(I)\ =\ \Ann_l I\,. 
\]
Denote by $\mathcal{D}_r(R)$ the set of dense right ideals of $R$, by 
$\mathcal{D}_{nr}(R)$ the set of right ideals $J$ of $N(R)$ such that for 
any ${0\ne x\in R}$, ${y\in N(R)}$, there is ${z\in N(R)}$ 
(equivalent to ${z\in J}$), ${x z\ne 0}$, ${y z\in J}$, put  
\[
\mathcal{D}_r^*(R)\ =\ \{N(J) R^1\mid J\in \mathcal{D}_r(R)\}\ =\ 
\{J R^1\mid J\in \mathcal{D}_{nr}(R)\}\subseteq \mathcal{D}_r(R)
\]
\cite[Proposition~2.6,~Remark~2.7]{MLQA} and for 
${I\in \mathcal{D}_r^*(R)}$ 
\begin{multline*} 
\Hom^*(I, R)_{N(R)}\ =
\\ 
\Bigl\{\phi\in \Hom(I, R)_{N(R)}\Bigl| 
\phi(y x)=(\phi y) x,\ \phi[y, z]\in N(R)\ \forall x\in R,\ 
y, z\in N(I)\Bigr\}\,.
\end{multline*}
The condition ${\Ann_l N(R)=\{0\}}$ is equivalent to any of the following 
conditions: ${\mathcal{D}_{nr}(R)\ne \emptyset}$; 
${N(R)\in \mathcal{D}_{nr}(R)}$; ${\mathcal{D}_r(R)\ne \emptyset}$; 
${R\in \mathcal{D}_r(R)}$; ${\mathcal{D}_r^*(R)\ne \emptyset}$; 
${N(R) R^1\in \mathcal{D}_r^*(R)}$. 
When this holds for all ${x\in R^1}$, ${I\in \mathcal{D}_r(R)}$,
${l_x|_{N(I) R^1}\in \Hom^*(N(I) R^1, R)_{N(R)}}$, since 
\[
[N(R), R]\,,\ R^1 [N(R), N(R)]\subseteq N(R)\,,\quad 
N(I)\ =\ I\cap N(R)
\] 
\cite[Lemma~1,~p.~164]{ZShS}, \cite[Lemma~2.3]{MLQA}. If $Q$ is a 
right algebra of quotients of $R$, then 
${x^{-1} N(R)=\{y\in N(R)\mid x y\in R\}\in \mathcal{D}_{nr}(R)}$ for all 
${x\in Q}$ \cite[Lemma~2.12]{MLQA}. In view of 
$\{0\}\ne I\cap R\lhd R$ for all ${\{0\}\ne I\lhd Q}$, $Q$ inherits the 
semiprimeness (primeness) $R$, the semisimpli-\linebreak city of $R$ with respect to 
any radical hereditary on subalgebras, and, in particular, its 
non-degeneracy. 

For an alternative $R$ with ${\Ann_l N(R)=\{0\}, \prr(D(R))=\{0\}}$ or 
$D(R)$ without 2--torsion, the complete (maximal) right algebra of quotients 
$Q(R)$, into which all right algebras of quotients of $R$ are embedded 
identically on $R$, is uniquely determined up to an isomorphism identical 
on $R$, and can be constructed as a quotient set $\mathcal{C}_{dr}^*(R)/\sim$ 
of the set 
\[
\mathcal{C}_{dr}^*(R)\ =\ \{(\phi, I)\mid 
\phi\in \Hom^*(I, R)_{N(R)},\ I\in \mathcal{D}_r^*(R)\}
\] 
under the equivalence relation $\sim$: ${(\phi, I)\sim (\phi', I')}$ if 
${\phi=\phi'}$ on ${I\cap I'}$ (equivalent on some 
${I''\in \mathcal{D}_r^*(R)}$, ${I''\subseteq I\cap I'}$, since for 
${\phi''=\phi-\phi'}$, ${x\in I\cap I'}$, ${\phi'' x\notin \phi''(I'')=\{0\}}$, 
there exists ${y\in N(R)}$, ${x y\in I''}$, 
${0\ne (\phi'' x) y=\phi'' (x y)=0}$, a contradiction), with the operations 
\begin{gather*}  
[(\phi, I R^1)]+[(\psi, J R^1)]\ =\ [(\phi+\psi, (I\cap J) R^1)]\,,
\\ 
[(\phi, I R^1)] [(\psi, J R^1)]\ =\ 
[(\phi \psi, (\psi^{-1}(I R^1)\cap J) R^1)]\,,
\\
f [(\phi, I R^1)]\ =\ [(f \Id_R, N(R) R^1)] [(\phi, I R^1)]\ =\ 
[(\phi, I R^1)] [(f \Id_R, N(R) R^1)]\ =\ [(f \phi, I R^1)]\,,
\end{gather*}
${(\phi, I R^1), (\psi, J R^1)\in \mathcal{C}_{dr}^*(R)}$, 
${I, J\in \mathcal{D}_{nr}(R)}$, ${f\in F}$, where  
\begin{gather*}
\phi \psi\Bigl(\sum_i x_i y_i\Bigr)\ =\ 
\sum_i (\phi \psi x_i) y_i\quad 
(x_i\in \psi^{-1}(I R^1)\cap J=\{x\in J\mid \psi x\in I R^1\},\ y_i\in R^1)\,,
\\
[(\phi, I R^1)] [(l_x, N(R) R^1)]\ =\ [(l_{\phi x}, I R^1)]\quad (x\in I R^1)\,,
\end{gather*}
unity ${[(\Id_R, N(R) R^1)]}$, ${R\hookrightarrow Q(R)}$, 
${x\longmapsto [(l_x, N(R) R^1)]}$, ${x\in R}$ 
\cite[Theorem~2.11,~Sec.~2.10]{MLQA} (${(I R^1) (N(R) R^1)\subseteq (I ([R^1, N(R)]+N(R) R^1)) R^1
\subseteq I R^1+(I R^1) R^1=I R^1}$). As in the associative case, $Q(R)$ can be 
defined by the axioms: 
\begin{enumerate}

\item $R$ is a subalgebra of $Q(R)$; 
                   
\item for any ${x\in Q(R)}$ there is ${I\in \mathcal{D}_{nr}(R)}$, 
${x I\subseteq R}$; 

\item if ${x\in Q(R)}$, ${I\in \mathcal{D}_{nr}(R)}$, ${x I=\{0\}}$, 
then ${x=0}$; 

\item for any ${I\in \mathcal{D}_{nr}(R)}$, ${\phi\in \Hom^*(I R^1, R)_{N(R)}}$ 
there is ${q\in Q(R)}$, ${l_q|_{I R^1}=\phi}$.

\end{enumerate}

\begin{lemma}
If an alternative algebra $R$ is semiprime and ${\Ann_l N(P(R))=\{0\}}$, then 
${Q(P(R))=O(Q(P(R)))}$ is a right algebra of quotients of $O(P(R))$, 
$Q(P(R))_B$ is a right algebra of quotients of $O(P(R))_B$ for any 
ultrafilter $B$ in $B(R)$.
\end{lemma}

\begin{proof}
Note that both conditions are satisfied on $R$ for non-degenerate $R$ (and 
so, $P(R)$). For any dense orthogonal ${\{\xi_a\}_{a\in I}\subseteq B(R)}$, 
${\{x_a\}_{a\in I}\subseteq Q(P(R))}$ there 
exist ${J_a\in \mathcal{D}_{nr}(P(R))}$, ${x_a J_a\subseteq P(R)}$, 
${a\in I}$. If ${J=\sum\limits_{a\in I} \xi_a J_a}$ and ${x\in P(R)}$, 
${x J=\{0\}}$, then ${\xi_a x J_a=\{0\}}$, ${\xi_a x=0}$ for 
all ${a\in I}$, ${x=0}$, ${\Ann_l J=\{0\}}$. If ${y\in N(P(R))}$, then 
\[
J_{y, a}\ =\ \{z\in N(P(R))\mid y z\in J_a\}\,,\
J'_a\ =\ J_{y, a}\cap J_a
\in
\mathcal{D}_{nr}(P(R))\quad (a\in I)
\]
\cite[Sec.~2.10~(2,~4),~Remark~2.5]{MLQA} (for any ${0\ne x\in P(R)}$,
${u\in N(P(R))}$ there is ${z\in N(P(R))}$, ${x z\ne 0}$, ${y u z\in J_a}$,
${u z\in J_{y, a}}$),
${J'=\sum\limits_{a\in I} \xi_a J'_a\subseteq J}$, ${\Ann_l J'=\{0\}}$ and so,
${J\in \mathcal{D}_{nr}(P(R))}$. Since
${\psi\in \Hom^*(J P(R)^1, P(R))_{N(P(R))}}$,
\[
\psi\,:\ \sum_i \xi_{a_i} y_{a_i} z_{a_i}\longmapsto 
\sum_i \xi_{a_i} (x_{a_i} y_{a_i}) z_{a_i}\ =\ 
\sum_i \xi_{a_i} x_{a_i} (y_{a_i} z_{a_i})\quad (y_a\in J_a,\ z_a\in P(R)^1)\,,
\]
there exists ${q\in Q(P(R))}$, ${l_q|_{J P(R)^1}=\psi}$,
${\xi_a (q-x_a) J_a P(R)^1=\{0\}}$, ${\xi_a q=\xi_a x_a}$ for all ${a\in I}$.
Therefore, ${Q(P(R))=O(Q(P(R)))}$ is a right algebra of quotients of $P(R)$ and
$O(P(R))$, $O(P(R))$ is a right algebra of quotients of $P(R)$ \cite[Proposition~1.12]{MLQA}.

If ${P=B(R)\setminus B}$, ${x, y\in Q(P(R))}$,
${x I=O(x I)\subseteq P Q(P(R))}$,
\[
I\ =\ y^{-1} N(O(P(R)))\ =\ 
O(I)\in \mathcal{D}_{nr}(O(P(R)))
\]
(see above using
${O(P(R))=O(O(P(R)))}$, ${N(O(P(R)))=O(N(O(P(R))))}$), then we can choose
${\alpha\in B}$, ${\alpha x I=\{0\}}$, ${\alpha x=0}$, ${x\in P Q(P(R))}$
(Proposition 2.4 (6)). If ${z\in Q(P(R))}$,
${T(z)=(z, O(P(R)), O(P(R)))\subseteq P Q(P(R))}$, then ${T(z)=O(T(z))}$,
${\beta T(z)=T(\beta z)=\{0\}}$ for some ${\beta\in B}$,
${\beta z\in N(Q(P(R)))}$ (Proposition 2.4 (6), Remark 3.12). 
Due to
$N(O(P(R)))=O(P(R))\cap N(Q(P(R)))$, this gives
\[
N(O(P(R))_B)\ =\ N(O(P(R)))_B\subseteq N(Q(P(R))_B)\ =\ N(Q(P(R)))_B\,,
\]
where $A_B$, ${A=O(P(R)), N(O(P(R))), N(Q(P(R)))}$, is identified with the image of $A$ in
$Q(P(R))_B$.
Thus,
$Q(P(R))_B$ is a right algebra of quotients
$O(P(R))_B$. If $R$ is non-degenerate, then $O(P(R))_B$ and $Q(P(R))_B$ are strongly prime
(Lemma 2.8 and the remarks above).
\end{proof}

In \cite{MLQA}, the maximal right (left) algebra of quotients is actually 
constructed for all algebras $R$ with an associator of elements that is 
skew-symmetric in arguments and ${\Ann_l N(R)=\{0\}}$ (${\Ann_r N(R)=\{0\}}$), 
and therefore Lemma 3.13 also remains true for such semiprime $R$.

\end{document}